\documentclass[11pt]{amsart}
\usepackage{amsfonts,amssymb,amsmath,amsthm,float}
\usepackage[colorlinks,citecolor=blue,urlcolor=black,linkcolor=black,pdfusetitle,{pdfauthor={Ari Meir Brodsky, Assaf Rinot, Shira Yadai}}]{hyperref}
\usepackage{pdflscape}

\newcounter{condition}

\newtheorem{thm}{Theorem}[section]
\newtheorem{mainthm}{Theorem}

\newtheorem{claim}{Claim}[thm]
\newtheorem{prop}[thm]{Proposition}
\newtheorem{lemma}[thm]{Lemma}
\newtheorem{cor}[thm]{Corollary}

\newtheorem{fact}[thm]{Fact}

\newtheorem{subclaim}{Subclaim}[claim]

\theoremstyle{definition}
\newtheorem{defn}[thm]{Definition}
\newtheorem{notation}[thm]{Notation}

\newtheorem{indef}[claim]{Notation}

\theoremstyle{remark}
\newtheorem{remark}[thm]{Remark}
\newtheorem{conv}[thm]{Convention}

\renewcommand{\mid}{\mathrel{|}\allowbreak}
\renewcommand{\restriction}{\mathbin\upharpoonright}

\DeclareMathOperator{\acc}{acc}
\DeclareMathOperator{\reg}{Reg}
\DeclareMathOperator{\nacc}{nacc}
\DeclareMathOperator{\cf}{cf}
\DeclareMathOperator{\dom}{dom}
\DeclareMathOperator{\im}{Im}
\DeclareMathOperator{\otp}{otp}
\DeclareMathOperator{\h}{ht}
\DeclareMathOperator{\add}{Add}
\DeclareMathOperator{\chr}{Chr}
\DeclareMathOperator{\p}{P}
\DeclareMathOperator{\id}{id}

\DeclareMathOperator{\pr}{Pr}

\DeclareMathOperator{\cof}{cof}

\DeclareMathOperator\suc{succ}
\DeclareMathOperator\refl{Refl}
\newcommand\xbox\boxtimes
\newcommand\s{\subseteq}
\newcommand\fin{\textup{fin}}
\newcommand{\sqx}{\sqleft{\chi}}
\newcommand\ind{\textup{ind}}
\newcommand*\fork[1]{{\pitchfork}(#1)}
\newcommand*\cvec[1]{\vec{\mathcal#1}}
\newcommand\sq{\sqsubseteq}
\newcommand\sqleft[1]{\mathrel{_{#1}{\sqsubseteq}}}
\newcommand\sqleftup[1]{\mathrel{^{#1}{\sqsubseteq}}}
\newcommand*\sqleftboth[2]{\mathrel{^{#1}_{#2}{\sqsubseteq}}}
\newcommand*\sqstarleftboth[2]{\mathrel{^{#1}_{#2}{\sqsubseteq^*}}}
\newcommand\br{\blacktriangleright}
\newcommand\onehalf{1\frac{1}{2}}
\newcommand\defaultaction{\textsf{extend}}

\newcommand\pvec{\fbox{${\cdots}$}\hspace{1pt}}
\newcommand*\axiomfont[1]{\textsf{\textup{#1}}}

\newcommand\pid{\axiomfont{PID}}
\newcommand\gch{\axiomfont{GCH}}
\newcommand\REG{\axiomfont{REG}}
\newcommand\SING{\axiomfont{SING}}
\newcommand\ch{\axiomfont{CH}}
\newcommand\ns{\textup{NS}}
\newcommand\bd{\textup{bd}}
\renewcommand\ll{<_{\mathbb{Q}_\lambda}}
\newcommand\VisL{\ensuremath{\axiomfont{V}=\axiomfont{L}}}

\newcommand\stree{\subsetneq}
\newcommand\sect[1]{\hypersetup{linkcolor=blue}\textcolor{blue}{\hyperref[#1]{Section~\ref{#1}}}\hypersetup{linkcolor=black}}
\newcommand\sealantichain{\textsf{anti}}
\newcommand\sd{\framebox[3.0mm][l]{$\diamondsuit$}\hspace{0.5mm}{}}

\numberwithin{table}{section}

\title{Proxy principles in combinatorial set theory}

\author{Ari Meir Brodsky}
\address{Mathematics Department, Shamoon College of Engineering, 56 Bialik St., Be'er Sheva, Israel.}
\urladdr{\url{https://en.sce.ac.il/faculty/ari_brodsky}}
\email{arimebr@sce.ac.il}

\author{Assaf Rinot}
\address{Department of Mathematics, Bar-Ilan University, Ramat-Gan 52900, Israel.}
\urladdr{http://www.assafrinot.com}

\author{Shira Yadai}
\address{Department of Mathematics, Bar-Ilan University, Ramat-Gan 52900, Israel.}
\email{greenss@biu.ac.il}


\begin{document}
\begin{abstract}
The parameterized proxy principles were introduced by Brodsky and Rinot in a 2017 paper,
as new foundations for the construction of $\kappa$\nobreakdash-Souslin trees
in a uniform way that does not depend on the nature of the (regular uncountable) cardinal $\kappa$.
Since their introduction, these principles have facilitated construction of Souslin trees with complex combinations of features,
and have enabled the discovery of completely new scenarios in which Souslin trees must exist.
Furthermore, the proxy principles have found new applications beyond the construction of trees.

This paper opens with a comprehensive exposition of the proxy principles.
We motivate their very definition, emphasizing the utility of each of the parameters and the consequent flexibility that they provide.
We then survey the findings surrounding them, presenting a rich spectrum of unrelated models and configurations in which the proxy principles are known to hold,
and showcasing a gallery of Souslin trees constructed from the principles.

The last two sections of the paper offer new results. In particular,
for every positive integer $n$, we give a construction of a $\lambda^+$-Souslin tree
all of whose $n$-derived trees are Souslin, but all of whose $(n+1)$-derived trees are special.
\end{abstract}
\date{Preprint as of September 8, 2025. For updates, visit \textsf{http://p.assafrinot.com/65}.}
\maketitle
\tableofcontents
\section{Introduction}
\label{section:introduction}

In his trailblazing paper analyzing the fine structure of the constructible hierarchy, appearing more than fifty years ago,
Jensen proved \cite[Theorem~6.2]{MR309729} that in G\"odel's constructible universe $\axiomfont{L}$,
there exists a $\kappa$-Souslin tree for every regular uncountable cardinal $\kappa$ that is not weakly compact.
Jensen's proof goes through two newly-minted combinatorial principles,
\emph{diamond} ($\diamondsuit$) and \emph{square} ($\square$),
introduced in \S5--6 of that paper.
The isolation and formulation of these new axioms have made the combinatorial properties of $\axiomfont{L}$
accessible to generations of set theorists,
enabling combinatorial constructions of complicated objects and leading to the settling of
open problems in fields including topology, measure theory, and group theory.

Let $\kappa$ denote a regular uncountable cardinal.
Recall that a \emph{coherent $C$-sequence over $\kappa$} is a sequence
$\langle C_\alpha \mid \alpha<\kappa \rangle$
such that, for every limit ordinal $\alpha<\kappa$:
\begin{itemize}
\item $C_\alpha$ is a club subset of $\alpha$; and
\item $C_\alpha\cap\bar\alpha = C_{\bar\alpha}$ for every $\bar\alpha\in\acc(C_\alpha)$.\footnote{For any set $C$ of ordinals, $\acc(C)$ stands for the set
$\{\beta\in C \mid \sup(C\cap\beta)=\beta>0\}$ of its accumulation points.
In particular, for an ordinal $\gamma$, $\acc(\gamma)$ stands for the set of nonzero limit ordinals below $\gamma$.}
\end{itemize}

An easy way to obtain such a sequence is to fix at the outset some club $D$ in $\kappa$, and then
let $C_\alpha:=D\cap\alpha$ for every $\alpha\in\acc(D)$ and $C_\alpha:=\alpha\setminus\sup(D\cap\alpha)$ for any other $\alpha$.
More interesting, however, are principles asserting the existence of coherent $C$-sequences
satisfying some \emph{non-triviality} condition.
For example:
\begin{itemize}
\item Jensen's square principle $\square_\lambda$ of \cite[\S5.1]{MR309729}
asserts the existence of a coherent $C$-sequence over $\lambda^+$,
$\langle C_\alpha \mid \alpha<\lambda^+ \rangle$,
such that $\otp(C_\alpha)\leq\lambda$ for every $\alpha<\lambda^+$.

\item For a stationary set $E\s\acc(\kappa)$,
the principle described in the conclusion of~\cite[Theorem~6.1]{MR309729},
which is commonly denoted $\square(E)$,
asserts the existence of a coherent $C$-sequence over $\kappa$,
$\langle C_\alpha \mid \alpha<\kappa \rangle$,
that \emph{avoids} $E$, meaning that $\acc(C_\alpha)\cap E =\emptyset$ for every $\alpha<\kappa$.

\item Todor\v{c}evi\'c's square principle $\square(\kappa)$ \cite[p.~267]{MR908147}
asserts the existence of a coherent $C$-sequence over $\kappa$,
$\langle C_\alpha \mid \alpha<\kappa \rangle$,
that is \emph{unthreadable} --- meaning
that there is no club $D\subseteq\kappa$ such that
$D\cap\alpha = C_\alpha$ for every $\alpha\in\acc(D)$.
\end{itemize}

In the decades ensuing since~\cite{MR309729},
many variants of both diamond and square have appeared ---
strengthening, weakening, or adapting each one as needed to solve various combinatorial problems.\footnote{Variants of square are surveyed in~\cite{magidor2012strengths}. Variants of diamond are surveyed in \cite{rinot_s01},
including the close connection between diamond principles and cardinal arithmetic.}
Strong combinations of square and diamond, such as Gray's principle $\sd_\lambda$ from \cite{MR2940957} and its further strengthening $\sd_\lambda^+$ from \cite{rinot21}, have appeared as well.
Square principles are primarily concerned with coherence,
whereas diamond principles are \emph{prediction} principles,
asserting that objects of size $\kappa$ can be predicted by means of their initial segments.

The construction of complicated combinatorial objects such as $\kappa$-Souslin trees
requires both prediction and coherence.
Classical constructions of Souslin trees have followed Jensen's lead in requiring the $\diamondsuit$-sequence's predictions
to occur in some nonreflecting stationary set $E$,
which must then be avoided by the square sequence in order not to interfere with building higher levels of the tree.
There are particular scenarios where the coherence requirements are transparent,
such as for $\kappa=\aleph_1$ where $\square_{\aleph_0}$ holds trivially,
and we thus find many classical constructions tailored to such cases alone.

Examining the classical literature,
one sees that construction of a $\kappa$\nobreakdash-Souslin tree with an additional property
(such as complete or regressive; rigid or homogeneous; specializable or non-specializable; admitting an ascent path or omitting an ascending path; free or uniformly coherent)
often depends on the nature of the cardinal $\kappa$ (be it a successor of a regular, a successor of a singular, or an inaccessible ---
in some cases even depending on whether $\kappa$ is the successor of a singular cardinal
of countable or of uncountable cofinality).
To obtain the additional features, constructions include extensive
bookkeeping, counters, timers, coding and decoding,
whose particular nature makes it difficult to transfer the process from one type of cardinal to another.

What happens if we want to replace an axiom known to imply the existence of a $\kappa_0$-Souslin tree with strong properties
by an axiom
from which a plain $\kappa_1$-Souslin tree can be constructed?
Do we have to revisit each scenario and tailor each of these particular constructions
in order to derive a tree with strong properties?

The parameterized proxy principles were introduced by Brodsky and Rinot in \cite{paper22} with the goal of overcoming this problem
by offering new foundations for constructing $\kappa$-Souslin trees for an arbitrary regular uncountable cardinal $\kappa$.
So far, they have been used to construct $\kappa$-Souslin trees in
\cite{paper22,rinot20,paper26,paper32,paper23,paper62,paper58,yadai}.
The core feature of the proxy principles
is that the non-triviality of a square-like sequence is ensured by a \emph{hitting} requirement ---
a weak form of prediction, to be explained in Section~\ref{section:what-are-proxy} ---
that is tailored for the desired construction,
rather than by the classical non-triviality conditions which were not flexible enough to obtain the optimal conclusions in many cases.
This tailoring enables \emph{uniform} construction of $\kappa$-Souslin trees and other combinatorial objects,
oblivious to the nature of~$\kappa$.

By incorporating such a
hitting feature into the square-like proxy principle,
one can reduce the requirements on the $\diamondsuit$-sequence:
In~\cite{paper22}, $\chi$\nobreakdash-complete $\kappa$-Souslin trees
were constructed using $\diamondsuit(\kappa)$ instead of $\diamondsuit(E)$ for some nonreflecting stationary
subset $E$ of $E^\kappa_{\ge\chi}$,\footnote{$E^\kappa_{\ge\chi}$ denotes the set $\{\alpha<\kappa\mid\cf(\alpha)\ge\chi\}$. The sets $E^\kappa_{>\chi}$, $E^\kappa_{\chi}$, $E^\kappa_{\neq\chi}$, $E^\kappa_{<\chi}$, and $E^\kappa_{\leq\chi}$ are defined analogously.}
and in~\cite{paper23}, $\diamondsuit(\kappa)$ was further relaxed
to the arithmetic hypothesis $\kappa^{<\kappa}=\kappa$.

\medskip

Since their introduction, the proxy principles have found new applications beyond the construction of Souslin trees.
In conjunction with $\diamondsuit$, we observe the following:
\begin{itemize}
\item In~\cite{paper29},
these principles were used to construct distributive Aronszajn trees, as well as special trees with a non-special projection.
\item In \cite{MR4530628}, these principles were used to a construct a large pairwise far family of Aronszajn trees.
\item In \cite{MR4833803}, these principles were used to construct minimal non-$\sigma$-scattered linear orders.
\end{itemize}

Furthermore, as a result of incorporating the hitting feature into the square-like sequence,
applications of the proxy principles in the absence of an arithmetic hypothesis, let alone a prediction principle,
have emerged, as follows:

\begin{itemize}\label{secondbatch}
\item In \cite{paper28}, these principles were used to construct a highly chromatic graph all of whose smaller subgraphs are
countably chromatic.
\item In \cite[Lemma~5.9]{paper47}, these principles were used to construct Ulam-type matrices.
\item In \cite[\S5]{paper48}, these principles were used to construct a Dowker space whose square is still Dowker.
\item In \cite[\S7]{paper45}, these principles were used to construct a $C$-sequence suitable for conducting walks on ordinals.\footnote{By walking along the outcome $C$-sequence,
the extreme instance $\pr_1(\kappa,\kappa,\kappa,\kappa)$ of Shelah's strong coloring principle was shown to be consistent.}
\end{itemize}

Alongside the wealth of applications of the proxy principles,
we turn our attention to the obvious question:
How do these new proxy principles compare
to the classical combinatorial axioms?
In~\cite{paper22,paper23},
a bridge to the classical foundations was built,
establishing that all previously known $\diamondsuit$-based constructions of $\kappa$-Souslin trees
may be redirected through the new foundations.
Concurrently, in \cite{paper22,MR3724382,paper24,rinot21,paper28,paper29,paper26,paper32,paper37,paper23,paper51},
instances of the proxy principles have been shown to hold in many unrelated configurations,
so that any conclusion derived from those instances is known to hold in completely new, unrelated scenarios.
Significantly, in addition to scenarios conforming to the spirit of \VisL,
it is shown that some instances of the proxy principles may consistently hold above large cardinals,
at a cardinal satisfying stationary reflection,
or in models of strong forcing axioms such as Martin's Maximum.
Furthermore, various notions of forcing inadvertently add instances of the proxy principles.
Thus, any application of a proxy principle
will automatically be known to hold in a rich spectrum of unrelated models.

Altogether, the proxy principles provide a successful disconnection between the combinatorial constructions and the study of the hypotheses themselves.
This project thus has two independent tasks:
Deriving rich applications of the proxy principles, and proving instances of the proxy principles in various scenarios.

\subsection{This paper}\label{sec11}
The main goal of this paper is to make the proxy principles accessible to anyone with experience in combinatorial set theory.
Until now, the various definitions and related results have been scattered throughout
multiple lengthy papers, making it difficult for the interested researcher
to adopt these principles as a starting point for deriving desired results.
We believe that the proxy principles have attained a significant level of maturity,
and we hope that by presenting this comprehensive exposition we can engage the reader and encourage them
to join us in our adventure of applying the proxy principles to obtain optimal results.

We now present the breakdown of the current paper.

In \sect{section:what-are-proxy}, we give a simple example from infinite graph theory that motivates the very need for a parameterized proxy principle,
and then patiently discuss each of the eight parameters of the proxy principle $\p(\ldots)$.
By the end of this section, the reader will hopefully agree that all of the parameters are quite natural indeed.

In \sect{section:proxyholds}, we gather configurations in which instances of the proxy principle hold, as established in the literature.

In \sect{section:gallery}, we list various types of Souslin trees that have been constructed using the proxy principles,
indicate where each of these results may be found in the literature and what vector of parameters is known to
be sufficient for the relevant construction.

The last two sections of this paper are dedicated to new proxy-based constructions of Souslin trees.
In \sect{section:Zakrzewski}, we present a proxy-based construction of a large family of pairwise-Souslin trees.
A sample corollary of the latter reads as follows.
\begin{mainthm}\label{thma} Assuming $\p(\kappa,2,{\sq},\kappa)$,
if there exists a $\kappa$-Kurepa tree,
then there exists a $\kappa$-Aronszajn tree $\mathbf T$ admitting $\kappa^+$-many $\kappa$-Souslin subtrees
such that the product of any finitely many of them is again Souslin.
\end{mainthm}
We shall also show that this result is optimal in the sense that the tree $\mathbf T$ itself cannot be $\kappa$-Souslin.

In \sect{section:special}, we present a proxy-based construction of a Souslin tree whose square is special. More generally:
\begin{mainthm}\label{thmb} For an infinite cardinal $\lambda$, assuming $\p_\lambda(\lambda^+,2,{\sq},\lambda^+)$, for every positive integer $n$,
there exists a $\lambda^+$-Souslin tree $\mathbf T$ satisfying the following:
\begin{itemize}
\item all $n$-derived trees of $\mathbf T$ are Souslin;
\item all $(n+1)$-derived trees of $\mathbf T$ are special.
\end{itemize}
\end{mainthm}

Such an $\aleph_1$-tree (i.e., the case $\lambda=\omega$) was constructed by Abraham and Shelah in \cite[\S2]{Sh:403}
building on the approach from \cite{MR384542,MR419229} of taking generics over countable models;
hence the construction does not generalize to $\lambda$ singular.
A construction for $\lambda$ singular (and $n=1$) was given by Abraham, Shelah and Solovay in \cite[\S4]{sh:221}
exploiting the fact that $\square_\lambda$ for $\lambda$ singular may be witnessed by a $C$-sequence $\langle C_\alpha\mid\alpha<\lambda^+\rangle$
with $\otp(C_\alpha)<\lambda$ for all $\alpha<\lambda^+$. As such, it does not apply to $\lambda$ regular.
The construction that will be given here is the first that works uniformly for $\lambda$ both regular and singular.

\subsection{Notation and conventions}\label{nocon}
Throughout this paper, $\kappa$ denotes a regular uncountable cardinal,
and $H_\kappa$ denotes the collection of all sets of hereditary cardinality less than $\kappa$.
The Greek letters $\lambda,\Lambda,\mu,\nu,\chi,\theta,\vartheta$ will denote (possibly finite) cardinals,
and $\alpha,\beta,\gamma,\delta,\epsilon,\varepsilon,\iota,\sigma,\varsigma,\xi$ will denote ordinals.

The class of all infinite regular (resp.~singular) cardinals is denoted by $\REG$ (resp.~$\SING$),
and we write $\reg(\kappa)$ for $\REG\cap\kappa$.
For a set $X$, write $[X]^\theta$ for the collection of all subsets of $X$ of size $\theta$,
and define $[X]^{<\theta}$ in a similar fashion.

In order to maintain the flow of the text, we decided not to pause to give the definitions of standard objects, giving them in footnotes, instead.

\section{What are the proxy principles, anyway?}
\label{section:what-are-proxy}
\subsection{Motivation}\label{sub:motivation}
Recall that a graph $\mathcal G$ is a pair $(V,E)$ where $E\s [V]^2$,
and that the chromatic number of $\mathcal G$, denoted $\chr(\mathcal G)$, is the least cardinal $\theta$ for which there exists a coloring $c:V\rightarrow\theta$
such that for every $\{x,y\}\in E$, $c(x)\neq c(y)$.
We say that a graph $\mathcal{G}$ is \emph{countably chromatic} iff $\chr(\mathcal{G})\leq\aleph_0$;
otherwise, it is \emph{uncountably chromatic}.

A \emph{Hajnal--M\'at\'e graph} is a graph $\mathcal G=(V,E)$ in which $V=\omega_1$
and, for every $\alpha<\omega_1$,
the set $A_\alpha:=\{ \beta<\alpha\mid \{\alpha,\beta\}\in E\}$ is either finite or a cofinal subset of $\alpha$ of order-type $\omega$.

Martin's axiom at the level of $\aleph_1$ implies that all Hajnal--M\'at\'e graphs are countably chromatic
(see~\cite[Proposition~31G]{fremlin1984consequences}),
and the same assertion is also consistent with the continuum hypothesis (see \cite[Theorem~2.1]{Sh:81}).

So, what does it take for $\mathcal G$ to be uncountably chromatic?
It can be verified that the following are equivalent:
\begin{enumerate}
\item $\mathcal G$ is uncountably chromatic;
\item For every partition $\langle B_n\mid n<\omega\rangle$ of $\omega_1$ into uncountable sets,
there is some $n<\omega$ and some $\alpha\in B_n$ such that $A_\alpha$ meets $B_n$;
\item For every partition $\langle B_n\mid n<\omega\rangle$ of $\omega_1$ into uncountable sets,
there is a nonzero $\alpha<\omega_1$ such that for the unique $n<\omega$ with $\alpha\in B_n$, it is the case that $\sup(A_\alpha\cap B_n)=\alpha$.
\end{enumerate}

In particular, if for every sequence $\langle B_n\mid n<\omega\rangle$ of uncountable subsets of $\omega_1$,
there is an $\alpha<\omega_1$ such that $A_\alpha$ meets $B_n$ for every $n<\omega$,
then $\mathcal G$ is uncountably chromatic.

The above connection between the \emph{hitting} property of $\vec A=\langle A_\alpha\mid\alpha<\omega_1\rangle$
and the chromatic number of the associated graph generalizes, as follows.
Let $\vec A=\langle A_\alpha\mid\alpha\in S\rangle$ be a \emph{ladder system}
over some subset $S$ of a regular uncountable cardinal $\kappa$.\footnote{This means that for every $\alpha\in S$, $A_\alpha$ is a subset of $\alpha$ and there is no $\beta<\alpha$ such that $A_\alpha\s\beta$.}
Derive a graph $\mathcal G:=(\kappa,E)$ by letting $E:=\{ \{\alpha,\beta\}\mid \alpha\in S,\allowbreak\,\beta\in A_\alpha\}$.
Then, for every cardinal $\theta$,
$\chr(\mathcal G)>\theta$ provided that for every sequence $\vec B=\langle B_i\mid i<\theta\rangle$ of cofinal subsets of $\kappa$,
there exists some $\alpha\in S$ such that $\bigwedge_{i<\theta} A_\alpha\cap B_i\neq\emptyset$.

Now, what happens if one wants, say, a ladder-system graph on $\omega_2$ of chromatic number $\omega_2$
such that, in addition, all of its smaller subgraphs are countably chromatic?\footnote{This is a nontrivial requirement. By \cite{MR925267}, it is consistent that every graph of size and chromatic number $\omega_2$
has a subgraph of size and chromatic number $\omega_1$.}
This quest for incompactness highlights a second feature that a ladder system may possess, namely, \emph{coherence}.
However, coherence properties are typically imposed upon ladder systems in which the $\alpha^{\text{th}}$ ladder is moreover a \emph{closed} subset of $\alpha$;
these are better known as \emph{$C$-sequences}.\footnote{So in a $C$-sequence, each $C_\alpha$ is a closed subset of $\alpha$ satisfying $\sup(C_\alpha)=\sup(\alpha)$.}
For two sets of ordinals $x,y$, write $x\sq y$ iff $x=y\cap\varepsilon$ for some ordinal $\varepsilon$.
A $C$-sequence $\langle C_\alpha\mid\alpha<\kappa\rangle$ is \emph{coherent} iff
for all $\alpha<\kappa$ and $\bar\alpha\in\acc(C_\alpha)$, it is the case that $C_{\bar\alpha}\sq C_\alpha$.
In order to obtain a graph that may satisfy the desired incompactness property,
we impose an additional constraint on the pairs in $E$, as follows:\footnote{See~\cite[Remark~2.6]{paper28}
for the history of Definition~\ref{CseqGraph} and its connection to the Hajnal--M\'at\'e graphs.}

\begin{defn}[The $C$-sequence graph, \cite{rinot12,paper28}]\label{CseqGraph} Given a $C$-sequence $\vec C=\langle C_\alpha\mid\alpha<\kappa\rangle$,
and a subset $G\s\acc(\kappa)$,
the graph $G(\vec C)$ is the pair $(G,E)$, where
$$E:=\{\{\alpha,\gamma\}\in[G]^2\mid \gamma\in C_\alpha,\,\min(C_\gamma)>\sup(C_\alpha\cap\gamma)\ge\min(C_\alpha)\}.$$
\end{defn}
\begin{remark}\label{motivatenacc} For every pair $\gamma<\alpha$ of vertices that are adjacent in $G(\vec C)$,
it is the case that $\gamma$ is an element of $\nacc(C_\alpha)$.\footnote{For any set $C$ of ordinals, $\nacc(C)$ denotes the set $C\setminus\acc(C)$ of its non-accumulation points.}
\end{remark}

Consider $G(\vec C)$ for a given $C$-sequence $\vec{C}=\langle C_\alpha\mid\alpha<\kappa\rangle$
and subset $G\s\acc(\kappa)$. For an ordinal $\delta<\kappa$, let $G(\vec C)\restriction\delta$ denote the initial-segment graph $({G\cap\delta},\allowbreak{E\cap[\delta]^2})$.
The next fact tells us, in particular, that if $\vec C$ is coherent, then every proper initial segment of $G(\vec C)$ is countably chromatic.

\begin{fact}[{\cite[Lemma~2.11(1)]{paper28}}]\label{fact22}
Let $\chi\in\reg(\kappa)$.
If $C_{\bar\alpha}\sq C_\alpha$ for all $\alpha\in G$ and $\bar\alpha\in\acc(C_\alpha)\cap E^\kappa_\chi$,
then $\chr(G(\vec C)\restriction\delta)\le\chi$ for every $\delta<\kappa$.
\end{fact}

Next, what does it take for the whole of $G(\vec C)$ to have a large chromatic number?
Based on what we saw earlier, one may guess that $\chr(G(\vec C))>\theta$ provided that for every sequence $\vec B=\langle B_i\mid i<\theta\rangle$ of cofinal subsets of $\kappa$,
there exists some $\alpha\in G$ such that $C_\alpha$ meets each of the $B_i$'s. However, due to the particular nature of $E$ (recall Definition~\ref{CseqGraph}),
here we would want $C_\alpha$ to meet each of the $B_i$'s in \emph{two consecutive points}.
Specifically:

\begin{fact}[{\cite[Lemma~2.13]{paper28}}]\label{fact23} For an infinite $\theta<\kappa$,
$\chr(G(\vec C))>\theta$, provided that for every sequence $\vec B=\langle B_i\mid i<\theta\rangle$ of cofinal subsets of $\kappa$,
there exists an $\alpha\in G$
with $\min(C_\alpha)\ge\min(B_0)$
such that, for every $i<\theta$, there are $\beta,\gamma\in C_\alpha\cap B_i$ such that $\gamma=\min(C_\alpha\setminus(\beta+1))$.
\end{fact}

Altogether, to obtain a graph of the form $G(\vec{C})$ satisfying a desired incompactness property ---
large chromatic number for the whole graph, along with small chromatic number for its proper initial segments ---
it suffices to begin with a $C$-sequence $\vec{C}$ satisfying a coherence property as in the hypothesis of Fact~\ref{fact22}
along with a hitting property as in the hypothesis of Fact~\ref{fact23}.
An axiom asserting the existence of a sequence satisfying a combination of coherence and hitting properties
is what we call a \emph{proxy principle}.

\subsection{The proxy principles}\label{s22}
In order to capture the considerations of the previous subsection while maintaining the flexibility to vary both the
coherence and hitting features as needed to prove various desired results,
one would like to introduce a concise parameterized notation for the proxy principles.
We would like to be able to express something along the following lines:
\begin{itemize}
\item There exists a system $\langle\mathcal C_\alpha\mid\alpha<\kappa\rangle$ with each $\mathcal C_\alpha$ a nonempty collection of closed cofinal subsets of $\alpha$;
\item There is a prescribed bound for how many sets are there at each level, e.g., $|\mathcal C_\alpha|=1$ for every $\alpha<\kappa$ (as in the examples we have seen thus far),
or more generally, for some fixed cardinal $\mu$, $|\mathcal C_\alpha|<\mu$ for every $\alpha<\kappa$;\footnote{Compare this with Jensen's weak square principle $\square^*_\lambda$
\cite[\S5.1]{MR309729} and Schimmerling's generalization $\square_{\lambda,\mu}$ \cite[\S5]{schimmerling1995combinatorial}.
The utility of \emph{multi-ladder systems}, i.e.,
systems in which there is more than one ladder assigned to each level, is demonstrated in \cite[Theorem~4.13]{paper48}.}
\item The elements of any level are compatible with the ones from below, i.e.,
there is a prescribed binary coherence relation $\mathcal R$ (such as $\sq$)
such that, for every $\alpha<\kappa$, every $C\in\mathcal C_\alpha$,
and every $\bar\alpha\in\acc(C)$, there exists a $D\in\mathcal C_{\bar\alpha}$ with $D\mathrel{\mathcal R} C$;
\item For some prescribed cardinal $\theta$, every family $\mathcal B\s[\kappa]^\kappa$ of size $\theta$ gets ``hit'' at some level $\alpha$,
i.e., each $C\in\mathcal C_\alpha$ meets each $B\in \mathcal B$.
Looking at Fact~\ref{fact23}, we may also want the $\alpha$ of interest to come from some prescribed set $G$,
and we may want a meeting that is successful twice in a row, or, more generally, $\sigma$ many times in a row for some prescribed ordinal $\sigma$.
\end{itemize}

The above considerations lead us to the definition of the parameterized proxy principle.
First, let us establish some notational conventions that we shall use throughout the rest of the paper:
\begin{itemize}
\item $\kappa$ is a regular uncountable cardinal;
\item $\nu$ and $\mu$ are cardinals $\le\kappa^+$ (typically, $2\le \nu\le\mu$);
\item $\chi$ and $\theta$ are cardinals $\le\kappa$;
\item $\mathcal R$ is a binary relation over $[\kappa]^{<\kappa}$;
\item $\mathcal S$ is a nonempty collection of stationary subsets of $\kappa$;
\item $\xi$ and $\sigma$ are ordinals $\leq \kappa$.
\end{itemize}

\begin{defn}[\cite{paper22,paper23}]\label{proxy}
The proxy principle $\p_\xi^-(\kappa, \mu, \mathcal R, \theta, \mathcal S, \nu,\sigma)$
asserts the existence of a sequence
$\vec{\mathcal C}=\langle \mathcal C_\alpha \mid \alpha < \kappa \rangle$
such that the following three requirements are satisfied:
\begin{enumerate}
\item for every $\alpha < \kappa$, $\mathcal C_\alpha$ is a nonempty collection of less than $\mu$ many closed subsets $C$ of $\alpha$ with $\sup(C)=\sup(\alpha)$ and $\otp(C)\le\xi$;
\item for all $\alpha < \kappa$, $C \in \mathcal C_\alpha$ and $\bar\alpha\in \acc(C)$,
there is a $D \in \mathcal C_{\bar\alpha}$ such that $D \mathrel{\mathcal R} C$;
\item
for every sequence $\langle B_i \mid i < \theta \rangle$ of cofinal subsets of $\kappa$,
for every $S \in \mathcal S$, there exist stationarily many $\alpha \in S$ for which:\footnote{If $\min\{\theta,\sigma\}>0$,
then merely requiring that ``there exists a nonzero $\alpha \in S$'' has an equivalent effect. See the proof of \cite[Theorem~4.3]{devlin1979variations}.}
\begin{itemize}
\item $\left| \mathcal C_\alpha \right| < \nu$, and
\item
for all $C \in \mathcal C_\alpha$ and $i < \min\{\alpha, \theta\}$:
\begin{equation}\label{hittingeqn}\tag{$\star$}
\sup\{ \beta \in C \mid \suc_\sigma (C \setminus \beta) \subseteq B_i \} = \alpha.
\end{equation}
\end{itemize}
\end{enumerate}
\end{defn}
\begin{remark}\label{remark-succ}
$\suc_\sigma(D) := \{ \delta\in \nacc(D)\mid 0<\otp(D\cap\delta)\le\sigma\}$
is the set of the first $\sigma$ many successor elements of $D$, should they exist.
In particular, for every $\beta\in C$ such that $\sup(\otp(C\setminus\beta))\ge\sigma$,
$\suc_\sigma(C\setminus\beta)$ is nothing but the next $\sigma$ many successor elements of $C$ above $\beta$.
In the special case $\sigma=1$, requirement~\eqref{hittingeqn} above is
equivalent to asserting that $\sup(\nacc(C) \cap B_i) = \alpha$.
\end{remark}
\begin{remark} One should view Clause~(3) of Definition~\ref{proxy} as a genericity-type feature.
This is because the forcing to add a sequence satisfying Clauses (1) and (2) via bounded approximations
will introduce one also satisfying Clause~(3) with $\theta=\kappa$, $\nu=\mu$, and arbitrarily large $\sigma<\kappa$ (see \cite[Lemma~3.9]{MR3724382} for a proof template).
\end{remark}

One can consider the proxy principle's eight parameters together as a \emph{vector of parameters}
$(\xi,\kappa,\mu,\mathcal R,\theta,\mathcal S,\nu,\sigma)$,\footnote{Some of our friends complained that our principle has $8$ parameters. Our response: $\infty^8=\infty$.}
and then divide the vector's components into three groups,
according to the clause of Definition~\ref{proxy} in which each parameter first appears.
The first three parameters $\xi,\kappa,\mu$ are in Clause~(1),
which amounts to saying that $\vec{\mathcal C}$ is a \emph{$\xi$-bounded $\mathcal C$-sequence} over $\kappa$ of \emph{width} less than $\mu$.
The fourth parameter $\mathcal R$ appears in Clause~(2), which amounts to saying that $\vec{\mathcal C}$ is $\mathcal R$-coherent.
The remaining parameters capture the hitting characteristics of $\vec{\mathcal C}$:
$\theta$ tells us how many sets can be hit simultaneously,
each element of $\mathcal S$ prescribes
the range in which hitting must take place,
$\nu$ forces the width to be locally small upon a successful hit,
and $\sigma$ sets a minimum for the number of serial successful meets.

The special case $\xi=\kappa$ imposes no order-type restriction on $\vec{\mathcal C}$, in which case we can freely omit it,
writing $\p^-(\kappa,\pvec)$ instead of $\p^-_\xi(\kappa,\pvec)$.
A small $\xi$ is indeed stronger (see~\cite[Lemma~3.13]{MR2013395}), and
imposing it enables stronger properties in the constructed object (see~\cite[Main Theorem]{MR4833803}).
In case $\kappa=\lambda^+$ is a successor cardinal,
it is tempting to view $\xi:=\lambda$ as the ultimate requirement; however, there are scenarios in which a particular choice of $\xi$ with $\lambda<\xi<\lambda^+$
turns out to be the optimal one (see \cite[\S3.3]{paper32}).

The special case $\mu=2$ implies that each $\mathcal C_\alpha$ is a singleton, say $\{C_\alpha\}$,
in which case we identify the $\mathcal C$-sequence $\langle \{C_\alpha\}\mid \alpha<\kappa\rangle$
with the $C$-sequence $\langle C_\alpha\mid\alpha<\kappa\rangle$. On the other extreme is
the case $\mu=\kappa^+$, which may seem pointless, but is nevertheless valuable when combined with a small value for $\nu$ (see \cite[Theorems 7.2 and 7.6]{paper45}),
or with an arithmetic hypothesis (see \cite[Theorem~5.14]{paper53}).

The basic coherence relation $\mathcal{R}$ is the \emph{end-extension} relation, $\sq$,
introduced in Subsection~\ref{sub:motivation},
indicating that $\cvec{C}$ is a \emph{coherent $\mathcal{C}$-sequence}.
A close examination of proxy-based constructions reveals that full coherence is not always necessary,
and the $\sq$ relation can be weakened in several ways, as follows.

First, considering some $C\in\bigcup_{\alpha<\kappa}\mathcal{C}_\alpha$ and some $\bar\alpha\in\acc(C)$,
it may be that all we require is for some $D\in\mathcal{C}_{\bar\alpha}$ to agree with $C$
\emph{at the final approach to~$\bar\alpha$}.
If this is the case, then the construction will work just as well from a $\sq^*$-coherent
instance of the proxy principle, where $D \sq^* C$ iff
there is some $\epsilon < \sup(D)$ such that
${D\setminus\epsilon}\sq{C\setminus\epsilon}$.

In another direction,
some proxy-based constructions can be designed to require genuine coherence only for
\emph{some} of the clubs in $\cvec{C}$, or only at \emph{some} of their accumulation points.\footnote{Recall Fact~\ref{fact22}.}
Indeed, there are contexts in which, for some infinite cardinal $\chi$,
there is no need to require coherence for clubs of order-type $<\chi$,
or possibly, there is no need to require coherence at accumulation points of cofinality $<\chi$.
Thus, in such cases we may weaken $\sq$
to either $\sqleft{\chi}$ or $\sq_\chi$,

where for a coherence relation $\mathcal R$:
\begin{itemize}
\item $D \mathrel{_{\chi}{\mathcal R}}C$ iff ((${D}\mathrel{\mathcal{R}}{C}$) or ($\cf(\sup(D))<\chi$)); and
\item $D \mathrel{\mathcal{R}_\chi} C$ iff ((${D} \mathrel{\mathcal{R}} {C}$) or ($\otp(C)<\chi$ and $\nacc(C)$ consists only of successor ordinals)).\footnote{The condition
``$\nacc(C)$ consists only of successor ordinals''
indicates that the club $C$ may be a ``dummy club'' that is not part of the genuine coherence structure of~$\cvec{C}$.
Thus, any construction from $\cvec{C}$ should ensure that the hitting does not occur at such a $C$.}
\end{itemize}
The significance of such a weakening is that
unlike coherent square sequences that are typically refuted by reflection principles,\footnote{Such as large cardinals \cite[Theorem~4.1]{MR1976595}, \cite[Proposition~8]{MR2374763},
strong forcing axioms \cite[Theorem~1.2]{MR2811288} and simultaneous reflection of stationary sets \cite[\S2]{MR3730566}.}
$\sq_\chi$-coherent proxy principles are compatible with a gallery of reflection principles and
provide an effective means of obtaining optimal incompactness results (see for instance \cite[Corollary~1.20]{paper22}, \cite[Theorem~A]{paper28}, and Table~\ref{table:proxy-consistent-models} below).

Note that the extreme case $\mathcal{R} = {\sqleft{\kappa}}$ amounts to saying that no coherence is needed at all,
and we call it the \emph{trivial coherence relation}.
In this case, every $\mathcal{C}_\alpha$ may be shrunk to a singleton, yielding a
proxy sequence with $\mu=2$.

In yet another direction, there are circumstances in which it is helpful to indicate that $\cvec{C}$ \emph{avoids} a particular class of ordinals $\Omega$,
meaning that $\acc(C)\cap\Omega=\emptyset$ for every $C\in\bigcup_{\alpha<\kappa}\mathcal{C}_\alpha$.
This requirement is indicated by prepending $\Omega$ as a superscript to the coherence relation $\mathcal{R}$,
thereby strengthening it to $^\Omega\mathcal{R}$.%
\footnote{By convention, this superscript-prepending is understood as being applied last, so that, for instance, $\sqleftup{\Omega}_\chi$ is to be parsed as $^\Omega(\sq_\chi)$.} %
In the context of walks on ordinals, the utility of avoiding a stationary subset of $\kappa$ is demonstrated by \cite[Theorem~6.2.7]{TodWalks} and \cite[Lemma~6.7]{MR3620068}.
In general, if $\vec{\mathcal C}$ is a $\sqleftup\Omega$-coherent
proxy sequence, then for any $\alpha\in\Omega$, one is free to shrink $\mathcal C_\alpha$ to a singleton,
and this has important ramifications (see \cite[Lemma~3.8]{paper32} and \cite[Corollary~4.27]{paper23}).

The weakest nontrivial value for the hitting parameters is $(\theta,\mathcal S,\nu,\sigma):=(1,\{\kappa\},\allowbreak\mu,1)$;
this minimal amount of hitting ensures that $\cvec{C}$ is unthreadable,\footnote{That is, for every club $D\s\kappa$, there is an $\alpha\in\acc(D)$ such that $D\cap\alpha\notin\mathcal C_\alpha$.} implying in particular that $\kappa$ is not weakly compact.
Furthermore, for every $\chi\in\reg(\kappa)$,
$\p^-(\kappa,\kappa^+,{\sqleft{\chi}^*},1,\{\kappa\},\kappa^+,1)$ implies that $\kappa$ admits a nontrivial $C$-sequence in the sense of \cite[Definition~6.3.1]{TodWalks}.%
\footnote{This is easily seen by taking a purported trivializing club $D$ to any transversal of the proxy sequence, and applying Definition~\ref{proxy}(3) to the cofinal set $B_0 := \acc(D)$, somewhat similar to the proof of \cite[Lemma~3.2]{paper22}. %
This proof also highlights the necessity for successor (i.e., non-accumulation) points of $C$
in \eqref{hittingeqn} of Definition~\ref{proxy}(3).}

We have already seen the utility of a large $\theta$ in the previous subsection,
and here the extreme case $\theta=\kappa$ is understood as a diagonal requirement, where each $C$ in $\mathcal C_\alpha$ is required to hit each $B_i$ for all of the $i$'s that are smaller than $\alpha$.
When one constructs a $\chi$-complete or a $\chi$-free $\kappa$-Souslin tree, it is natural to require $E^\kappa_{\ge\chi}$ to be in $\mathcal S$ (see \cite[Proposition~2.2]{paper26} and~\cite[Theorem~4.12]{paper32}).
An extreme case is requiring $\mathcal S$ to contain \emph{all} stationary subsets of a given stationary $S^*\s \kappa$,
e.g., $S^*:=\{\alpha<\kappa\mid \cf(\alpha)=\cf(|\alpha|)\}$
as was done in \cite[Theorem~3.1]{MR4530628}.
The significance of distinguishing $\nu$ from $\mu$ is highlighted in \cite{paper32}, \cite[\S4.3]{paper23},
and Tables \ref{table:derived-combinatorial} and~\ref{table:mu=kappa;nu=2} below,
where instances of the proxy principle with locally-small width are used to derive optimal results from the existence of nonreflecting stationary sets
in scenarios where we cannot guarantee the desired coherence together with small width overall.
An application of $\nu=2$ may be found in \cite[Theorem~6.17]{paper23}, where the narrowness requirement at the hitting ordinals enables sealing potential automorphisms of a tree;\footnote{In fact, it is open whether
a $\kappa$-Souslin tree constructed from an instance of the proxy principle having $\nu>2$ can be secured to be rigid on a club.}
another application of $\nu=2$ is given by Corollary~\ref{cor46} below.
The utility of $\sigma>1$ is demonstrated by Fact~\ref{fact23} above.
Applications of $\sigma=\omega$ may be found throughout \cite{rinot20}.
Nota bene that in some cases $\sigma=\omega$ implies the existence of a nonreflecting stationary subset of $E^\kappa_\omega$
(see \cite[Theorem~4.1]{MR3724382} for a primary scenario).
\begin{conv}
When defining a sequence $\vec{\mathcal{C}}$ witnessing an instance of the proxy principle,
it suffices to specify $\langle \mathcal{C}_\alpha \mid \alpha\in\acc(\kappa) \rangle$;
in so doing,
we make the tacit assumption that $\mathcal{C}_0 := \{\emptyset\}$ and
$\mathcal{C}_{\beta+1} := \{\{\beta\}\}$ for every $\beta<\kappa$.
\end{conv}

\subsection{Monotonicity}\label{monotonicity}
The reader can verify that the proxy principle satisfies monotonicity properties with respect to most of its parameters, as follows:
\begin{itemize}
\item Any sequence witnessing
$\p_\xi^-(\kappa, \mu, \mathcal R, \theta, \mathcal S, \nu,\sigma)$
remains a witness to the principle if any of $\xi$, $\mu$, or $\nu$ is increased;
if $\theta$ or $\sigma$ is decreased; if $\mathcal R$ is weakened;
if $\mathcal S$ is shrunk;
or if any element of $\mathcal S$ is expanded.
\item The case $\mathcal R={\sq}$ coincides with $\mathcal R={\sqleftboth{\Omega}{\chi}}$ for
$(\Omega,\chi):=(\emptyset,\omega)$;
increasing $\chi$
weakens each of the relations $\sqleftboth{\Omega}{\chi}$, $\sqstarleftboth{\Omega}{\chi}$, $\sqleftup{\Omega}_{\chi}$,
and $^{\Omega}{\sq}^*_\chi$,
while expanding $\Omega$ strengthens the same relations.
\item For $\mathcal R \in \{ {\sq}, {\sq^*} \}$, any $\chi$, and any $\Omega$,
$\p_\xi^-(\kappa,\mu, {^{\Omega}\mathcal{R}_{\chi}},\pvec)$ entails
$\p_\xi^-(\kappa,\mu,\allowbreak{^\Omega_\chi{\mathcal{R}}},\pvec)$.
\end{itemize}

In addition, in some cases, strong instances of the proxy principle follow from apparently weaker ones.
See for instance \cite[Lemmas 3.8 and 3.9 and \S3.2]{paper28}, \cite[Lemmas 4.9 and 5.3]{paper29},
\cite[Lemmas 3.8 and 3.20]{paper32}, and \cite[Theorem 4.15 and Corollary~4.40]{paper23}.

\subsection{Simplifications}\label{simplifications}
To make the parameterized proxy principle more accessible, a few of its main instances have been given abbreviations in the literature.
For $S$ a stationary subset of $\kappa$, the abbreviations are as follows:
\begin{itemize}
\item
$\xbox_\xi^-(S)$ denotes $\p_\xi^-(\kappa,2,{\sq},1,\{S\},2,1)$,
i.e., the instance asserting the existence of a coherent $\xi$-bounded $C$-sequence with a minimal nontrivial hitting feature.
Together with $\diamondsuit(\kappa)$ this enables a very simple construction of a $\kappa$-Souslin tree
(see~\cite[\S2.7]{paper23}).
\item
$\xbox_\xi^*(S)$ denotes $\p_\xi^-(\kappa,\kappa,{\sqleft{\chi}^*},1,\{S\},\kappa,1)$,
where $\chi := \min\{\cf(\alpha) \mid \alpha \in S\cap\acc(\kappa) \}$.
This is a weakening of $\xbox^-_\xi(S)$ in the spirit of Jensen's $\square_\xi^*$ that is nonetheless sufficient for various constructions
(see \cite[Proposition~2.2]{paper26}).
In case that $\kappa$ happens to be $({<}\chi)$-closed,\footnote{A cardinal $\kappa$ is \emph{$\tau$-closed} (for a cardinal $\tau$) iff $\lambda^{\tau}<\kappa$ for every $\lambda<\kappa$.
A cardinal $\kappa$ is \emph{$({<}\chi)$-closed} iff it is $\tau$-closed for every $\tau<\chi$.}
\cite[Theorem~4.39]{paper23} tells us that $\p_\xi^-(\kappa,\kappa,{\sqleft{\chi}^*},1,\{S\},\kappa,1)$
is no weaker than $\p^-_\xi(\kappa,\kappa,\allowbreak{\sq},\allowbreak1,\{S\},\kappa,1)$.

In the special case where $\kappa$ is the successor of a regular cardinal $\xi$
and $S\subseteq E^{\xi^+}_\xi$,
the principle $\xbox_\xi^*(S)$ becomes $\p_\xi^-(\kappa,2,{\sqleft\kappa},1,\{S\},2,1)$
as stated in~\cite[Definition~1.8]{MR4530628}.
This is because $\min\{\cf(\alpha) \mid {\alpha \in S}\}=\xi$,
and no club in a witness for $\xbox_\xi^*(S)$
has accumulation points of cofinality $\ge\xi$,
so here $\sqleft{\chi}^*$ coincides with the trivial coherence relation $\sqleft\kappa$,
thereby allowing $\mu$ to be shrunk to $2$.

\item One may replace the stationary set $S$ by a collection $\mathcal S$ of stationary sets,
and/or add an indication for the width $\mu$, e.g., writing $\xbox_\xi^-(\mathcal S,{<}\mu)$ for $\p_\xi^-(\kappa,\mu,{\sq},1,\mathcal S,\mu,1)$.
\end{itemize}

Whenever possible and in order to reduce an unnecessary load, a convention for the omission of parameters has been established.
As already mentioned earlier, if we omit the parameter $\xi$, then we mean that $\xi=\kappa$, i.e., $\p^-(\kappa,\pvec)$ stands for $\p^-_\kappa(\kappa,\pvec)$.\footnote{Likewise $\boxtimes^-(S),\boxtimes^-(\mathcal S),\boxtimes^*(S)$ and $\boxtimes^*(\mathcal S)$ stand for
$\boxtimes_\kappa^-(S),\ldots,\boxtimes_\kappa^*(\mathcal S)$, respectively.}
Independently, we may omit a \emph{tail} of parameters, as follows:

\begin{itemize}
\item If we omit $\sigma$, then $\sigma = ``{<}\omega"$, which we will discuss shortly;
\item If in addition we omit $\nu$, then $\nu = \kappa^+$;
\item If in addition we omit $\mathcal S$, then $\mathcal S = \{\kappa\}$;
\item If in addition we omit $\theta$, then $\theta = 1$.
\end{itemize}

We are left with discussing our choice for a default value of $\sigma$.

Initially, in \cite[p.~1953]{paper22}, the authors stated that the omission of $\sigma$ would amount to putting $\sigma=1$
as this seemed to be the weakest possible value that is still useful (following the rationale of the other omitted parameters).
Later on, it was realized that this was an illusion,
since all of the applications of $\p^-(\kappa,\ldots)$ at the time were in the context of $\diamondsuit(\kappa)$,
and in this context any instance $\p^-_\xi(\kappa, \pvec,1)$
may be witnessed by a sequence simultaneously witnessing $\p^-_\xi(\kappa, \pvec,n)$ for all $n<\omega$;
that is, $\diamondsuit(\kappa)$ implies that $\p^-_\xi(\kappa, \pvec,1)$ is no weaker than $\p^-_\xi(\kappa, \pvec,{<}\omega)$.\footnote{See \cite[Theorem~4.16(1)]{paper23} for a stronger result.}
Through the work surrounding Fact~\ref{fact23}, one learns to appreciate the possibility of having $\p^-_\xi(\kappa, \pvec,\sigma)$ holding with $\sigma$ slightly greater than $1$.
The hitting feature of Fact~\ref{fact23} --- namely, requiring two consecutive meets of $C_\alpha$ with $B_i$,
but not insisting that the smaller of the two be a non-accumulation point of $C_\alpha$ ---
may be expressed as something like ``$\sigma=\onehalf$'',
but fortunately such an awkward notation can be avoided, as \cite[Theorem~4.15]{paper23} shows that
$\p^-_\xi(\kappa, \pvec,\onehalf)$, $\p^-_\xi(\kappa, \pvec,2)$, and $\p^-_\xi(\kappa, \pvec,{<}\omega)$ are all logically equivalent.

With the revised convention of setting a default value of $\sigma = ``{<}\omega"$ as in \cite[Convention~4.18]{paper23}, a door was opened to proxy-based constructions of $\kappa$-Souslin trees using $\kappa^{<\kappa}=\kappa$ instead
of $\diamondsuit(\kappa)$ (see \cite[Theorems 5.13 and 6.8]{paper23}), and to the second batch of applications mentioned in the paper's introduction (see Page~\pageref{secondbatch}).\footnote{Strictly speaking, $\sigma=1$ suffices for the construction of the said Ulam-type matrix.}

Finally, in order to facilitate the use of the proxy principles in conjunction with other common hypotheses,
we adopt the following:
\begin{itemize}
\item $\p_\xi(\kappa, \mu, \mathcal R, \theta, \mathcal S, \nu,\sigma)$
denotes the conjunction of $\p_\xi^-(\kappa, \mu, \mathcal R, \theta, \mathcal S, \nu,\sigma)$ and $\diamondsuit(\kappa)$;
\item $\p^\bullet_\xi(\kappa, \mu, \mathcal R, \theta, \mathcal S, \nu)$
denotes the conjunction of $\p_\xi^-(\kappa, \mu, \mathcal R, \allowbreak\theta, \mathcal S, \nu,{<}\omega)$ and $\kappa^{<\kappa}=\kappa$.\footnote{See
Definition~5.9 and Corollary~5.14 of~\cite{paper23}. Note that here the value of $\sigma$ is hardcoded to be ${<}\omega$; this is because the obvious generalization to $\sigma\ge\omega$ (together with $\mu<\kappa$) will already imply $\diamondsuit(\kappa)$
(see \cite[Proposition~5.17]{paper23}).}
\end{itemize}

\subsection{To bullet or not to bullet?}\label{subsection:bullet}
In our applications of $\p_\xi(\kappa,\ldots)$, we shall prefer to use a more versatile version $\diamondsuit(H_\kappa)$ of $\diamondsuit(\kappa)$, as follows:
\begin{fact}[{\cite[Lemma~2.2]{paper22}}]\label{diamondhkappa} $\diamondsuit(\kappa)$ is equivalent to the principle $\diamondsuit(H_\kappa)$
asserting the existence of a partition $\langle R_i \mid i < \kappa \rangle$ of $\kappa$
and a sequence $\langle \Omega_\beta \mid \beta < \kappa \rangle$ of elements of $H_\kappa$
such that for all $i<\kappa$, $\Omega \subseteq H_\kappa$ and $p\in H_{\kappa^{+}}$,
the following set is stationary in $\kappa$:
$$B_i(\Omega,p):=\{\beta\in R_i\mid \exists \mathcal M\prec H_{\kappa^+}\,(p\in \mathcal M, \beta=\kappa\cap\mathcal M, \Omega_\beta=\Omega\cap \mathcal M)\}.$$
\end{fact}

The way that the combination of $\p_\xi^-(\kappa, \ldots)$ and $\diamondsuit(H_\kappa)$
is typically used in recursive constructions of length $\kappa$ is as follows:
At limit stage $\alpha<\kappa$, one uses the ladders of $\mathcal C_\alpha$ in order to climb up and eventually determine the $\alpha^{\text{th}}$ level of the ultimate object.
Clause~(2) of Definition~\ref{proxy} ensures that this climbing procedure will not reach a dead end.
Finally, Clause~(3) of Definition~\ref{proxy} is invoked with sets $B_i$ that arise from an application of $\diamondsuit(H_\kappa)$, namely $B_i:=B_i(\Omega,p)$
for an educated choice of $\Omega$ and $p$.

Looking at \cite[Definition~5.9]{paper23}, we see that the principle $\p^\bullet_\xi(\kappa, \pvec)$
can be understood as a weakening of $\p_\xi^-(\kappa, \pvec)\land\diamondsuit(H_\kappa)$ that is tailored to hit only sets $B_i$ of the above particular form.

Strictly speaking, $\p^\bullet_\xi(\kappa, \pvec)$ is weaker than $\p_\xi(\kappa,\pvec)$,\footnote{$\p^\bullet_\omega(\omega_1,2,{\sq},\omega_1)$ holds after adding a single Cohen real to a model of $\ch$ (see Table~\ref{table:derived-forcing} below).
If $\diamondsuit(\omega_1)$ failed in the ground model, then
by~\cite[Exercise~VII.H9]{kunen1980set} it remains failing in the extension, so that $\p_\omega(\omega_1,\ldots)$ will fail in the extension.}
but due to the nature of our constructions (as roughly described above), we do not know of any application of $\p_\xi(\kappa,\pvec)$
that cannot be transformed into an application of the weaker principle $\p^\bullet_\xi(\kappa, \pvec)$.
In particular, all of the $\kappa$-trees constructed from $\p_\xi(\kappa,\pvec)$ in Sections \ref{section:Zakrzewski} and \ref{section:special} below,
may as well be constructed assuming $\p^\bullet_\xi(\kappa, \pvec)$, instead.

\subsection{For the adventurous readers}\label{adventuroussubsection}
We mention that there are a few additional values that can be assigned to the parameters of the proxy principle.
First, by letting $\xi:={<}\lambda$, we mean that all ladders in the witnessing proxy sequence have order-type strictly less than $\lambda$.
The fact that singular cardinals may admit a ladder system having small order-type everywhere was exploited by Shelah and his co-authors a long time ago (see \cite{Sh:236,sh:221} and recall the final paragraph of Subsection~\ref{sec11}).
Second, by letting $\xi:={<}$, we mean that the set of lower-regressive levels (see Definition~\ref{Ccharacteristics} below) of the witnessing proxy sequence covers a club.
Third, by letting $\mu:=\infty$ we mean that $|\mathcal C_\alpha|\le|\alpha|$ for every nonzero $\alpha<\kappa$.
Fourth, if we write ${<}\theta$ instead of $\theta$,
then we mean that the proxy sequence simultaneously satisfies Clause~(3) of Definition~\ref{proxy} with $\theta$ replaced by $\vartheta$ for all $\vartheta<\theta$.
An analogous interpretation applies when writing ${<}\sigma$ instead of $\sigma$.
Fifth, by letting $\sigma:={<}\infty$, we mean to replace the assertion of Equation~\eqref{hittingeqn} of Definition~\ref{proxy} by:
$$\forall\sigma<\otp(C)\,\sup\{ \beta \in C \mid \suc_\sigma (C \setminus \beta) \subseteq B_i \} = \alpha.$$

Coming back to $\mu$, we have the following (indexed) strengthening of $\p_\xi^-(\kappa, \mu^+,{\sq},\allowbreak\theta,\mathcal{S})$, which aids in constructions of $\kappa$-trees with a $\mu$-ascent path (see Definition~\ref{ascentlaver} below).
It reads as follows.

\begin{defn}[{\cite[\S4.6]{paper23}}]\label{indexedP}
For an infinite cardinal $\mu<\kappa$,
the principle $\p_\xi^-(\kappa,\allowbreak \mu^{\ind},\allowbreak{\sq},\theta,\mathcal{S})$ asserts the existence of a $\xi$-bounded $\mathcal C$-sequence $\langle \mathcal{C}_\alpha\mid \alpha<\kappa\rangle$
together with a sequence $\langle i(\alpha)\mid \alpha<\kappa\rangle$ of ordinals in $\mu$, such that:
\begin{itemize}
\item for every $\alpha<\kappa$, there exists a canonical enumeration $\langle C_{\alpha,i}\mid i(\alpha)\le i<\mu\rangle$ of $\mathcal C_\alpha$ (possibly with repetition)
satisfying that the sequence $\langle \acc({C}_{\alpha,i})\mid i(\alpha)\le i<\mu\rangle$
is $\s$-increasing with $\bigcup_{i\in[i(\alpha),\mu)}\acc(C_{\alpha,i})=\acc(\alpha)$;
\item for all $\alpha<\kappa$, $i\in[i(\alpha),\mu)$ and $\bar\alpha\in\acc({C}_{\alpha,i})$, it is the case that $i\ge i(\bar\alpha)$ and $C_{\bar\alpha,i}\sq C_{\alpha,i}$;
\item for every sequence $\langle B_\tau\mid \tau<\theta\rangle$ of cofinal subsets of $\kappa$, and every $S\in\mathcal{S}$, there are stationarily many $\alpha\in S$
such that for all $C\in\mathcal C_\alpha$ and $\tau<\min\{\alpha,\theta\}$, $\sup(\nacc(C)\cap B_\tau)=\alpha$.
\end{itemize}
\end{defn}

In addition to coherence as alluded to in the introduction, another concept that is invisible at the level of $\kappa=\aleph_1$
is that of a \emph{proxy-respecting} tree.
This concept arises when one tries to construct a companion $\kappa$-tree $S$ for
a given $\kappa$-tree $T$
in such a way that the product tree $S\otimes T$ becomes Souslin \cite[\S5]{paper58},
or that the cofinal branches of $T$
would injectively embed into the automorphism group of $S$ \cite[\S7]{yadai},
or that a designated reduced power of $S$ would contain a copy of $T$ \cite[\S6]{rinot20}.
The point is that if we were to construct the new $\kappa$-tree $S$ using an instance $\p^-(\pvec)$ of the proxy principle,
then it will be useful to be able to interpret (even if artificially) the other $\kappa$-tree $T$ as an outcome of a similar application of $\p^-(\pvec)$.
This leads to the following definition (that makes use of some concepts that are defined in Subsection~\ref{basics} below).

\begin{defn}[\cite{rinot20}]\label{respecting}
A streamlined $\kappa$-tree $T$
is \emph{$\p_\xi^-(\kappa, \mu,\mathcal R, \theta, \mathcal S,\nu,\sigma)$-respecting}
if there exists a subset $\S\s\kappa$ and a system of mappings
$\langle b ^C:(T\restriction C)\rightarrow {}^\alpha H_\kappa\cup\{\emptyset\}\mid \alpha<\kappa, C\in\mathcal C_\alpha\rangle$ such that:
\begin{enumerate}
\item\label{respectingonto} for all $\alpha\in\S$ and $C\in\mathcal C_\alpha$, $T_\alpha\s \im(b^C)$;
\item $\vec{\mathcal C}=\langle \mathcal C_\alpha\mid\alpha<\kappa\rangle$ witnesses $\p_\xi^-(\kappa, \mu,\mathcal R, \theta, \{S\cap\S\mid S\in \mathcal S\},\nu,\sigma)$;
\item \label{respectcohere} for all sets $D\sq C$ from $\vec{\mathcal C}$ and $x\in T\restriction D$, $b^D(x)=b^C(x)\restriction\sup(D)$.
\end{enumerate}
\end{defn}

Typically, a $\kappa$-Souslin tree obtained from a proxy principle $\p_\xi^-(\pvec)$ with $\nu=2$
using the so-called \emph{microscopic approach} will be $\p^-_\xi(\pvec)$-respecting.
In this case, the witnessing $\S$ is a subset of $\{\alpha\in\acc(\kappa)\mid |\mathcal C_\alpha|=1\}$ in which the hitting must occur,
and $b^C(x)$ (for the unique $C\in \mathcal C_\alpha$ whenever $\alpha\in\S$)
is a distinguished node in $T_\alpha$ extending $x$.
The microscopic nature of the construction ensures that Clause~(3) will hold,
and that for all $C$ and $x$, $b^C(x)$ will be comparable to $x$.

Another case of interest is $\kappa=\lambda^+$ for an infinite regular cardinal $\lambda$
such that $\p_\lambda^-(\kappa, \mu,{\sqleft{\lambda}},\allowbreak\theta, \{E^\kappa_\lambda\},\allowbreak\nu,\sigma)$ holds.
In this case, \emph{every} $\kappa$-tree is $\p_\lambda^-(\kappa, \mu,\allowbreak{\sqleft{\lambda}}, \theta, \{E^\kappa_\lambda\},\allowbreak\nu,\sigma)$-respecting \cite[Lemma~2.22]{yadai}.
In particular, $\diamondsuit(\aleph_1)$ implies that all $\aleph_1$-trees are $\p_\omega^-(\aleph_1, 2,\allowbreak{\sq},\allowbreak \aleph_1, \{\aleph_1\})$-respecting.

Unorthodox examples of respecting trees, including Kurepa trees and the trees recording characteristics of walks on ordinals, may be found in \cite[Theorems 4.10 and 4.15]{rinot21}.

\medskip

We close this subsection by briefly touching upon a dual approach in which instead of formulating an axiom asserting the existence of a $\mathcal C$-sequence of a certain type,
one defines \emph{characteristics} of $\mathcal C$-sequences and studies their impact on the objects derived from them.
Preliminary examples may be found in \cite[Definition~1.1]{paper29}, \cite[Definition~4.2]{paper35}, \cite[Definition~3.13]{paper45},
and \cite[Definition~4.1]{paper46}.
A more systematic study of this approach was recently initiated in \cite{paper71}.
We settle here for including two samples.

\begin{defn}[{\cite[\S4]{paper71}}]\label{Ccharacteristics}
For a $\mathcal{C}$-sequence, $\vec{\mathcal{C}}=\langle \mathcal{C}_\alpha\mid\alpha<\kappa\rangle$:
\begin{itemize}
\item the set of \emph{avoiding levels} of $\vec{\mathcal{C}}$ is the following:
$$A(\vec{\mathcal{C}}):=\{\alpha\in\acc(\kappa)\mid \forall\beta\in(\alpha,\kappa)\forall C\in\mathcal{C}_\beta\,(\alpha\notin\acc(C))\};$$
\item the set of \emph{lower-regressive levels} of $\vec{\mathcal{C}}$ is the following:
$$R(\vec{\mathcal{C}}):=\{\alpha\in\acc(\kappa)\mid \forall\beta\in(\alpha,\kappa)\forall C\in\mathcal{C}_\beta\,(\otp(C\cap\alpha)<\alpha)\}.$$
\end{itemize}
\end{defn}

\subsection{What's next?}\label{subsection27}
As mentioned in the paper's Introduction,
the proxy principles provide a \emph{disconnection}
between the combinatorial constructions and the study of the hypotheses themselves.
This is a well-known approach and is no different from other axioms such as $\diamondsuit$, $\square$ or the P-Ideal Dichotomy (\pid).
In all of these cases, by matching any \emph{application} of the axiom
with an appropriate \emph{configuration} in which the axioms is known to hold,
one obtains a \emph{conclusion} of possible interest.
Arguably, one factor determining the success of an axiom of this sort is the exact cut point at which the disconnection is introduced.
We think that a good cut point is one in which the study of applications and configurations is equally wealthy.
With this view, in the upcoming two sections we shall establish that the proxy principles are indeed a successful collection of axioms
by demonstrating the rich findings in the
two independent sides of this project.
At no point will we try to list all of the resulting conclusions,
as their number has order of magnitude equal to the product of the two.
In particular, we leave to the reader the task of \emph{reconnecting} applications and configurations as needed
or for the joy of verifying that all classical $\diamondsuit$-based Souslin-tree constructions can now be redirected through the proxy principles.\footnote{Bear in mind Subsection~\ref{monotonicity}.}

\newpage
\section{Deriving instances of the proxy principle}\label{section:proxyholds}
In this section, we shall give three tables demonstrating that instances of the proxy principles hold in many different configurations.
For a stationary $S\s\kappa$,
$\refl(S)$ asserts that every stationary subset of $S$ reflects at some ordinal in $E^\kappa_{>\omega}$,
$\ns_\kappa^+\restriction S$ stands for the collection of all stationary subsets of $S$,
and $\ns_\kappa^+$ stands for $\ns_\kappa^+\restriction\kappa$.
We also write $\ch_\lambda$ for the assertion that $2^\lambda=\lambda^+$,
and $\ch(\lambda)$ for the assertion that $2^{<\lambda}=\lambda$.

\begin{table}[H]
\resizebox{\textwidth}{!}{
\begin{tabular}{l|l|l}
Hypothesis & Instance of proxy obtained & Citation \\
\hline\hline
$\clubsuit(S)$ for $\kappa=\sup(S)$ & $\p^-(\kappa,2,{\sqleft{\kappa}},1,\{S\},2,\kappa)$ &\cite[\S5]{paper22} \\
$\clubsuit(S)$ for $\lambda^+=\sup(S)$ & $\p^-_\lambda(\lambda^+,2,{\sqleft{\lambda}},1,\{S\},2,\lambda^+)$
&\cite[Thm~5.1(1)]{paper22} \\

$\clubsuit(E^{\lambda^+}_\lambda)$ for $\lambda\in\REG$
& $\xbox^*_\lambda(E^{\lambda^+}_\lambda)$
&\cite[Thm~5.1(1)]{paper22} \\

$\diamondsuit(S)$ for $\kappa=\sup(S)$ & $\p(\kappa,2,{\sqleft{\kappa}},1,\{S\},2,\kappa)$ &\cite[\S5]{paper22} \\
$\diamondsuit(S)$ for $\lambda^+=\sup(S)$ & $\p_\lambda(\lambda^+,2,{\sqleft{\lambda}},1,\{S\},2,\lambda^+)$
&\cite[Thm~5.1(2)]{paper22} \\

$\diamondsuit(S)$ for $S\subseteq E^{\lambda^+}_{\cf(\lambda)}$, $\lambda^+=\sup(S)$
& $\p_\lambda(\lambda^+,2,{\sqleft{\lambda}},\lambda^+,\{S\},2,{<}\lambda)$
&\cite[Thm~5.6]{paper22} \\

$\diamondsuit(S)$ for $\omega_1=\sup(S)$ & $\p_\omega(\aleph_1,2,{\sq},\aleph_1,\{S\},2,{<}\omega)$
&\cite[Thm~3.7]{paper22} \\

$\diamondsuit(S)$ for $\omega_1=\sup(S)$ & $\p_{\omega^2}(\aleph_1,2,{\sq},\aleph_1,\{S\},2,{<}\omega^2)$
&\cite[Thm~3.6]{paper22} \\

$\diamondsuit^*(E^{\lambda^+}_\lambda)$ for $\lambda\in\REG$
& $\p_\lambda(\lambda^+,2,{\sqleft{\lambda}},\lambda^+, \mathcal{S},2,{<}\infty)$
&\cite[\S4.4]{paper23} \\

$\diamondsuit(S), \lnot\refl(S)$, $\sup(S)=\kappa$ inaccessible
& $\p(\kappa,\kappa,{\sqleftup{S}},1,\{S\},2,\kappa)$ &\cite[Thm~4.26(1)]{paper23} \\
$\sd_\lambda$ for $\lambda\geq\aleph_1$
& $\p_\lambda(\lambda^+,2,{\sq},\lambda^+,\{E^{\lambda^+}_{\cf(\lambda)}\},2,{<}\lambda)$
&\cite[Thm~3.6]{paper22} \\

$\square_\lambda \land \ch_\lambda$, $\lambda\in\REG\setminus\{\aleph_0\}$
& $\p(\lambda^+,2,{\sq}^*,\omega,\{E^{\lambda^+}_\lambda\},2, {<}\omega)$
&\cite[Cor~6.2(2)]{paper22}\\

$\square_\lambda \land \ch_\lambda$, $\lambda\in\SING$
& $\p_\lambda(\lambda^+,2,{\sq},\lambda^+,\{E^{\lambda^+}_{\cf(\lambda)}\},2, {<}\lambda)$
&\cite[Cor~3.10]{paper22}\\

$\square_\lambda \land \ch_\lambda$ for $\lambda\geq\aleph_1$
&$\p_\lambda(\lambda^+,2,{\sq},{<}\lambda,$
&\cite[Cor~3.9]{paper22} \\
&\hfill$\{E^{\lambda^+}_\rho\mid \rho\in\reg(\lambda)\},2,{<}\lambda)$&\\

$\square(\lambda^+,\sq_\chi) \land \ch_\lambda$ for $\lambda\ge\aleph_1$ \text{strong limit}
&$\p(\lambda^+,2,{\sq_\chi},1,$
&\cite[Cor~4.15]{paper29} \\
&\hfill$\{E^{\lambda^+}_\rho\mid \rho\in\reg(\lambda)\},2,{<}\omega)$&\\

$\square(\lambda^+) \land \ch_\lambda$ with $\lambda^{\aleph_0}=\lambda$
& $\xbox^-(E^{\lambda^+}_\omega)$
&\cite[Cor~4.4]{paper24} \\

$\square(\lambda^+) \land \ch_\lambda$ with $\mathfrak b\le\lambda<\aleph_\omega$
& $\xbox^-(E^{\lambda^+}_\omega)$
&\cite[Cor~5.12]{paper51} \\

$\square(\lambda^+,{<}\lambda) \land \ch_\lambda$ with $\lambda^{<\lambda}=\lambda\ge\aleph_1$
& $\forall\rho\in\reg(\lambda)\;{\xbox^*(E^{\lambda^+}_\rho)}$
&\cite[Thm~3.5]{paper37} \\

$\square(\lambda^+) \land \gch$ with $\lambda\geq\aleph_1$
& $\forall\rho\in\reg(\lambda)\;{\xbox^-(E^{\lambda^+}_\rho)}$
&\cite[Cor~4.5]{paper24} \\

$\square(\lambda^+) \land \ch_\lambda$ with $\lambda\geq\beth_\omega$
& $\forall\rho\in\reg(\beth_\omega)\;{\xbox^-(E^{\lambda^+}_\rho)}$
&\cite[Cor~4.7]{paper24} \\

$\boxtimes^-(S)$, $\refl(S)$, $S\in\ns_\kappa^+$, $\diamondsuit(\kappa)$
& $\p(\kappa,2,{\sq},\kappa,\{S\},2,{<}\omega)$ &\cite[Thm~3.11(2)]{paper28} \\
$\square(\lambda^+) \land \ch_\lambda\land\ch(\lambda)$ for $\lambda\in\SING$
& $\p(\lambda^+,2,{\sq},\lambda^+,\{\lambda^+\},2,{<}\omega)$
&\cite[Cor~4.22]{paper29} \\

$\square(E) \land \diamondsuit(E)$, $\kappa=\sup(E)\geq\aleph_2$
& $\forall S\in\ns^+_{\kappa}$
&\cite[Cor~4.19(2)]{paper23} \\
&\hfill$\p(\kappa,2,{\sq}^*,1,\{S\},2,{<}\omega)$&\\

$\square^{\ind}(\lambda^+,\mu) \land \ch_\lambda \land \aleph_0\leq\mu<\lambda$
& $\forall\rho\in\reg(\lambda)\setminus\mu \text{ s.t.~} \lambda^\rho=\lambda$
&\cite[Thm~4.44]{paper23} \\
&\hfill$\p(\lambda^+, \mu^{\ind},{\sq},1,\{E^{\lambda^+}_\rho\})$&\\

$\ch_\lambda, \ns\restriction E^\lambda_\theta$ saturated, $\lambda=\theta^+$, $\theta\in\REG$
& $\p(\lambda^+,2,{\sqleft{\lambda}}^*,\theta,\{E^{\lambda^+}_\lambda\},2, \theta)$
&\cite[Thm~6.4]{paper22} \\

$\lambda\in\REG\setminus\{\aleph_0\},\ch(\lambda), \ch_\lambda, \lnot\refl(E^{\lambda^+}_{\neq\lambda})$
& $\p_\lambda(\lambda^+,\lambda^+,{\sq},{<}\lambda,\{\lambda^+\},2,{<}\lambda)$
&\cite[Thm~A]{paper32} \\

$\lambda\in\REG\setminus\{\aleph_0\},\ch(\lambda), \ch_\lambda, \lnot\refl(E^{\lambda^+}_{\neq\lambda})$
& $\p(\lambda^+,\lambda^+,{\sq^*},1,\{E^{\lambda^+}_\lambda\},2,{<}\omega)$
&\cite[Thm~A]{paper32} \\

$\lambda\in\SING,\ch(\lambda),\ch_\lambda, \square^*_\lambda, \lnot\refl(E^{\lambda^+}_{\neq\cf(\lambda)})$
& $\p_{\lambda^2}(\lambda^+,\lambda^+,{\sq},\lambda^+,\{\lambda^+\},2,{<}\lambda)$
&\cite[Thm~B]{paper32} \\

\VisL, $\kappa$ not weakly compact
& $\p(\kappa,2,{\sq},\kappa, \mathcal{S}^*, 2, \omega)$

&\cite[Cor~1.10(5)]{paper22} \\
\hline
\end{tabular}
}
\smallskip
\caption{Instances of proxy derived from combinatorial hypotheses.
$\mathcal{S}$ stands for $\ns^+_{\lambda^+}\restriction E^{\lambda^+}_\lambda$;
$\mathcal{S}^*$ stands for $\{ E^\kappa_{\geq\chi} \mid \chi\in\reg(\kappa) \land \kappa \text{ is $({<}\chi)$-closed} \}$.}\label{table:derived-combinatorial}
\end{table}

\begin{remark}
By \cite[Corollary~4.13]{paper24}, $\boxtimes^-(\kappa)\land\diamondsuit(\kappa)$
implies $\p(\kappa,2,\allowbreak{\sq}^*,\allowbreak 1,\{S\})$ for every stationary $S\s\kappa$.
\end{remark}

\begin{table}[H]
\resizebox{\textwidth}{!}{
\begin{tabular}{l|l|l|l}
Properties of forcing & Properties of ground model & Instance of proxy obtained & Citation \\
\hline\hline
$\add(\lambda,1)$ & $\ch(\lambda)$
& $\p^-_\lambda(\lambda^+,2,{\sqleft{\lambda}},\lambda^+, \mathcal{S}, 2, {<}\lambda)$
&\cite[Cor~7.4]{paper45} \\

& $\ch(\lambda),\square_\lambda$
& $\p^-_\lambda(\lambda^+,2,{\sq},\lambda^+, \mathcal{S}, 2, {<}\lambda)$
&\cite[Thm~2.3]{rinot12} \\

& $\ch(\lambda),\square_\lambda,\ch_\lambda$
& $\p^\bullet_\lambda(\lambda^+,2,{\sq},\lambda^+, \mathcal{S}, 2, {<}\lambda)$
&\cite[Thm~6.1(11)]{paper23} \\

& $\ch(\lambda),\square_\lambda,\ch_\lambda,\lambda>\aleph_0$
& $\p_\lambda(\lambda^+,2,{\sq},\lambda^+, \mathcal{S}, 2, {<}\lambda)$
&\cite[Thm~4.2(2)]{paper22} \\
\hline
$({<}\lambda)$-distributive, $\kappa$-cc,
& $\ch(\lambda)$ and
& $\p_\lambda(\lambda^+,\infty,{\sq},\lambda^+, $
&\cite[Prop~3.10]{paper26} \\
collapsing $\kappa$ to $\lambda^+$
& $\kappa$ strongly inaccessible $>\lambda$ & \hfill $\mathcal{S}, 2, {<}\infty)$ & \\
\hline
$\lambda^+$-cc, size $\leq\lambda^+$, & & $\p_\lambda(\lambda^+,\infty,{\sq},\lambda^+, $ & \\
preserves regularity of $\lambda$, & $\ch(\lambda)\land\diamondsuit(\lambda^+)$
& \hfill $\mathcal{S}, 2, {<}\infty)$
&\cite[Thm~3.4]{paper26} \\
not $^\lambda\lambda$-bounding & & & \\
\hline
$\lambda^+$-cc, size $\leq\lambda^+$, & $\ch(\lambda)\land\ch_\lambda$ & $\p(\kappa,\infty,{\sq},\kappa, $ & \\
forces $\cf(\lambda)<|\lambda|$ & $\kappa=\lambda^+\ge\aleph_2$ and $S=E^{\kappa}_\lambda$
& \hfill $ {\ns^+_{\kappa}\restriction S}, 2, {<}\infty)$
&\cite[Thm~3.4]{paper26} \\
(e.g., Prikry, & & & \\
\hfill Magidor, Radin) &&&\\
\hline
L\'evy-collapsing $\lambda$ to $\chi$ & $\lambda^{<\chi}=\lambda>\chi \land\ch_\lambda$
& $\p_\chi(\kappa,\infty,{\sq},\kappa, $ & \\
& $\kappa=\lambda^+$ and $S=E^{\kappa}_\lambda$
& \hfill $ {\ns^+_{\kappa}\restriction S}, 2, {<}\infty)$
&\cite[Prop~3.9]{paper26} \\

\hline
\end{tabular}
}
\smallskip
\caption{Instances of proxy obtained in forcing extensions;
in all cases $\lambda$ and $\chi$ stand for infinite regular cardinals,
$\mathcal{S}$ stands for $\ns^+_{\lambda^+}\restriction E^{\lambda^+}_\lambda$,
and the improvement from the parameter $\theta=1$ to $\theta=\kappa$ (or $\theta=\lambda^+$) is secured by~\cite[Lemma~4.32]{paper23}.}
\label{table:derived-forcing}
\end{table}

\begin{remark} In \cite[\S3.3]{paper28}, one can find, for $\chi\in\reg(\kappa)$, a $\chi$-directed-closed and $\kappa$-strategically-closed forcing poset
for introducing a witness to $\p^-(\kappa,2,{\sq_\chi},\kappa,\allowbreak(\ns^+_\kappa)^V,2,\sigma)$.
This has been used in \cite[\S4]{paper28} to prove the compatibility of strong instances of the proxy principle with a gallery of reflection principles, such as Martin's Maximum, Rado Conjecture, Chang's conjecture, the Fodor-type Reflection Principle,
and $\Delta$-reflection.
\end{remark}

\begin{table}[H]
\resizebox{\textwidth}{!}{
\begin{tabular}{l|l|l}
Features of model & Instance(s) of proxy satisfied & Citation \\
\hline\hline
Martin's Maximum
& & \cite[Cor~1.20]{paper22} \\

$\forall\lambda\in\SING\cap\cof(\omega)$\hfill$\lnot\square^*_\lambda$
& $\forall \lambda\in\SING$
& \\

&\hfill$\p_\lambda(\lambda^+,2,{\sq_{\aleph_2}},\lambda^+,\{E^{\lambda^+}_{\cf(\lambda)}\},2,{<}\lambda)$ \\
&&\\
$\forall\lambda\in\REG\setminus\{\aleph_0\}$
& $\forall \lambda\in\REG\setminus\{\aleph_0\}$ \\
$\hfill\diamondsuit(E^{\lambda^+}_\lambda)$, $\lnot\square_{\lambda,\aleph_1}$
&\hfill $\p_\lambda(\lambda^+,2,{\sqleft{\lambda}},\lambda^+,\{E^{\lambda^+}_\lambda\},2, {<}\lambda)$ \\

\hline
$\lambda$ supercompact, $\refl(E^{\lambda^+}_{<\lambda})$,
&$\p_\lambda(\lambda^+,2,{\sq_{\lambda}},\lambda^+,\{E^{\lambda^+}_{\lambda}\},2,{<}\lambda)$
& \cite[Cor~1.24]{paper22} \\
\hfill$\lnot\diamondsuit(E^{\lambda^+}_\lambda)$, $\neg\square_\lambda$&& \\
\hline
$\lambda = \chi^{+\omega}$, $\chi$ supercompact,
&$\p_\lambda(\lambda^+,2,{\sq_{\chi}},\lambda^+,\{E^{\lambda^+}_{\cf(\lambda)}\},2,{<}\lambda)$
& \cite[Cor~4.7]{paper22} \\
\hfill$\lnot\square^*_\lambda$&& \\
\hline
$\lambda=\aleph_\omega$, $\refl(\lambda^+)$
&$\boxtimes^-(\ns_{\aleph_{\omega+1}}^+)$\text{ and }
& \cite[Thm~1.12]{MR3724382} \\
&$\p_{\aleph_\omega}(\aleph_{\omega+1},\aleph_1,{\sq},\aleph_{\omega+1}, \{E^{\aleph_{\omega+1}}_{\aleph_n}\mid n<\omega\})$&\\
\hline
$\lambda\in\REG\setminus\{\aleph_0\}$, $\refl(E^{\lambda^+}_{<\lambda})$
&$\boxtimes^-(\ns_{\lambda^+}^+)$& \cite[Rmk~1.13]{MR3724382} \\
\hline
\end{tabular}
}
\smallskip
\caption{Five models in which $\square_\lambda$ fails,
yet strong instances of the proxy principle at $\lambda^+$ hold. In the last two models, $\gch$ holds.}
\label{table:proxy-consistent-models}
\end{table}

\section{A gallery of Souslin-tree constructions}\label{section:gallery}

As described in the paper's Introduction,
the original catalyst for the formulation of the proxy principles was the desire for a uniform combinatorial construction of $\kappa$-Souslin trees,
and indeed the principles have served this purpose successfully.
In this section we showcase the various Souslin trees that have been built using the proxy principles,
and provide references to where these constructions can be found.

\subsection{The basics}\label{basics}
The reader is probably familiar with the abstract definition of a set-theoretic tree as a poset $(T,{<_T})$ all of whose downward cones are well-ordered.
In this project, we opt to work with a particular form of trees which we call \emph{streamlined},
in which nodes of the tree are (transfinite) sequences,
and the tree-order is nothing but the initial-sequence ordering.
This choice does not restrict our study,
but it does make some of the considerations smoother.\footnote{See~\cite[\S2.3]{paper23} for a comparison of streamlined trees
with abstract trees, highlighting the properties of streamlined trees, the advantages of constructing them,
and the fact that we lose no generality by restricting our attention to them.
Note that the proof of Lemma~2.5(1) there has a minor glitch, but that the statement is correct.
Ur Ya'ar has drawn our attention to~\cite[\S2]{MR168484},
where trees of this form are called \emph{sequential trees},
and where a proof of the aforementioned Lemma can be found.}
For instance, in such a tree $T$, the $\alpha^{\text{th}}$ level of $T$
coincides with the set $T_\alpha:=\{x\in T\mid \dom(x)=\alpha\}$,
and we may likewise define $T\restriction C$ to be $\{x\in T\mid \dom(x)\in C\}$.
Furthermore,
for any $\stree$-increasing sequence $\eta$ of nodes of $T$,
the unique limit of the sequence, which may or may not be in $T$, is $\bigcup\im(\eta)$.

\begin{defn}[{\cite[Definition~2.3]{paper23}}]
A \emph{streamlined tree} is a subset $T \subseteq {}^{<\kappa}H_\kappa$ for some cardinal $\kappa$
such that $T$ is \emph{downward-closed}, that is, $\{ t\restriction\beta \mid \beta<\dom(t) \} \subseteq T$ for every $t \in T$.
The \emph{height} of $T$, denoted $\h(T)$,
is the least ordinal $\alpha$ such that $T_\alpha=\emptyset$.
For any $\alpha\le\h(T)$, an \emph{$\alpha$-branch} 
is a subset $B\s T$ that is linearly ordered by $\stree$ and such that $\{\dom(t)\mid t\in B\}=\alpha$;
$\h(T)$-branches are usually referred to as \emph{cofinal branches through $T$}.

If $\h(T)$ is a limit ordinal,
we denote by $\mathcal B(T)$ the collection of all functions
$f:\h(T)\rightarrow H_\kappa$
such that $\{ f\restriction\beta \mid \beta<\h(T) \}$ is a cofinal branch through $T$.

By a \emph{streamlined subtree} $S$ of a streamlined tree $T$, we mean a subset $S \subseteq T$ that is itself a streamlined tree;
equivalently, a streamlined subtree $S$ of $T$ is a downward-closed subset of the ambient tree $T$.
\end{defn}

Following \cite[Convention~2.6]{paper23}, we shall identify a streamlined tree $T$ with the abstract tree $(T,{\stree})$.

\begin{defn}\label{def42}
\begin{enumerate}
\item A \emph{streamlined $\kappa$-tree} is a streamlined tree $T\s{}^{<\kappa}H_\kappa$ such that
$0<|T_\alpha|<\kappa$ for every $\alpha<\kappa$.
\item A \emph{streamlined $\kappa$-Aronszajn tree} is a streamlined $\kappa$-tree $T$
such that $\mathcal B(T)=\emptyset$.
\item A \emph{streamlined $\kappa$-Kurepa tree} is a streamlined $\kappa$-tree $T$ such that $|\mathcal B(T)|>\kappa$.\footnote{We follow the definition given in~\cite[Definition~II.5.16]{kunen1980set},
which is satisfactory for our purposes.
In some contexts,
it may be useful to impose that a $\kappa$-Kurepa tree is \emph{slim} (see Definition~\ref{def-prop-trees} below),
in order to exclude trivial examples
such as the complete splitting binary tree of height a strongly inaccessible cardinal
(see~\cite[p.~317]{devlin-book}).
}
\item A \emph{streamlined $\kappa$-Souslin} tree is a streamlined $\kappa$-Aronszajn tree $T$
with no \emph{antichains} of size $\kappa$, that is,
for every $A\in[T]^\kappa$, there are distinct $s, t \in A$ that are \emph{comparable}.\footnote{For $s,t \in T$,
we say that $s$ and $t$ are \emph{comparable} iff $s \subseteq t$ or $t \subseteq s$;
otherwise they are \emph{incomparable}.
An \emph{antichain} $A$ in $T$ is a subset $A \subseteq T$ such that for all $s,t \in A$, if $s \neq t$ then $s$ and $t$ are incomparable.}
\end{enumerate}
\end{defn}

Section~2 of~\cite{paper23}, entitled ``How to construct a Souslin tree the right way'',
offers a comprehensive exposition of the subject of constructing $\kappa$\nobreakdash-Souslin trees
and the challenges involved,
culminating in Subsections 2.6--2.7
with a detailed description of a very simple construction of a $\kappa$\nobreakdash-Souslin tree,
proving the following basic result:

\begin{fact}[{\cite[Proposition~2.18]{paper23}}]\label{simple-Souslin-construction}
$\xbox^-(\kappa) \land \diamondsuit(\kappa)$ implies the existence of a $\kappa$-Souslin tree.
\end{fact}

In~\cite[p.~1965]{paper22}, one finds a comparison between the classical non-smooth approach to Souslin-tree construction
(requiring nonreflecting stationary sets)
and the modern approach using the proxy principles
(enabling construction of a $\kappa$-Souslin tree in models where every stationary subset of $\kappa$ reflects).
Further advantages of the proxy principles in the context of Souslin-tree construction are described in~\cite[\S1]{paper23}.

By a slightly more elaborate construction,
it is possible to weaken the hypotheses of Fact~\ref{simple-Souslin-construction} considerably, as follows:
\begin{fact}[{\cite[Corollary~6.7]{paper23}}]\label{careful-Souslin-construction}
$\p^\bullet(\kappa,\kappa, {\sq^*}, 1, \{\kappa\}, \kappa)$ implies the existence of a $\kappa$-Souslin tree.
\end{fact}

Recalling the meaning of $\xbox^-(\kappa)$ and $\p^\bullet(\ldots)$ as given in Subsection~\ref{simplifications},
we see that the main improvements of Fact~\ref{careful-Souslin-construction} over Fact~\ref{simple-Souslin-construction}
consist of
weakening the parameter $\mu$ from $2$ to $\kappa$, as well as
weakening the prediction principle $\diamondsuit(\kappa)$ to the arithmetic hypothesis $\kappa^{<\kappa}=\kappa$.

\subsection{Properties of trees}\label{tree-properties}
The literature is rich with additional properties that $\kappa$-trees may possess. Let us discuss some of them.
\begin{defn}\label{def-prop-trees}
A streamlined tree $T$ of height $\kappa$ is said to be:
\begin{itemize}
\item \emph{normal} iff for all $\alpha<\beta<\kappa$ and $x\in T_\alpha$,
there is some $y\in T_\beta$ such that $x\stree y$;\footnote{Trees with this property are also called \emph{well-pruned}; see~\cite[Definition~II.5.10]{kunen1980set}.}
\item \emph{binary} iff $T \s {}^{<\kappa}2$;
\item \emph{$\varsigma$-splitting} (for an ordinal $\varsigma<\kappa$) iff
every node in $T$ admits at least $\varsigma$ many immediate successors;\footnote{For two nodes $x,y$
in a streamlined tree $T$,
we say that $y$ is an \emph{immediate successor} of $x$ iff $x\stree y$ and $\dom(y)=\dom(x)+1$;
equivalently, iff $y = x^\smallfrown\langle\iota\rangle$ for some $\iota$.}
\item \emph{splitting} iff it is 2-splitting;
\item \emph{prolific} iff, for all $\alpha<\kappa$ and $x\in T_\alpha$,
$\{ x^\smallfrown\langle\iota\rangle \mid \iota<\max \{\omega, \alpha\}\}\s T_{\alpha+1}$;
\item \emph{slim} iff $|T_\alpha| \leq \max \{\left|\alpha\right|, \aleph_0\}$ for every ordinal $\alpha$;
\item \emph{$\chi$-complete} iff,
for any $\stree$-increasing sequence $\eta$, of length $<\chi$, of nodes of $T$,
the (unique) limit of the sequence, $\bigcup\im(\eta)$, is also in $T$;
\item \emph{full} iff for every $\alpha\in\acc(\kappa)$, $|\mathcal B(T\restriction \alpha)\setminus T_\alpha|\le1$;
\item \emph{rigid} iff its only automorphism is the trivial one;
\item \emph{$\chi$-coherent} iff $|\{ \alpha\in\dom(x)\cap\dom(y)\mid x(\alpha)\neq y(\alpha)\}|<\chi$
for all $x,y\in T$;
\item \emph{coherent} iff it is $\omega$-coherent;
\item \emph{uniformly homogeneous} iff for all $y\in T$
and $x\in T\restriction\dom(y)$,
the union of $x$ and $y\restriction(\dom(y)\setminus\dom(x))$ (which is usually denoted by $x*y$) is in $T$;
\item \emph{uniformly coherent} iff it is coherent and uniformly homogeneous;
\item \emph{regressive} iff there exists a map $\rho:T\rightarrow T$ satisfying the following:
\begin{itemize}
\item for every non-minimal $x\in T$, $\rho(x)\stree x$;
\item for all $\alpha\in\acc(\kappa)$ and $x\neq y$ from $T_\alpha$,
either $(x,\rho(y))$ or $(\rho(x),y)$ is a pair of incomparable nodes;
\end{itemize}
\item \emph{club-regressive} (respectively, \emph{stationarily-regressive}) iff it is regressive,
as witnessed by a map $\rho$ satisfying the following additional feature:
\begin{itemize}
\item for every $\alpha \in E^\kappa_{>\omega}$, there exists a club (respectively, stationary) subset $e_\alpha\subseteq\alpha$
such that for all incomparable $x$ and $y$ from $T\restriction(e_\alpha\cup\{\alpha\})$,
either $(x,\rho(y))$ or $(\rho(x),y)$ is a pair of incomparable nodes;
\end{itemize}

\item \emph{special}
iff there exists a map $\rho:T\rightarrow T$ satisfying the following:
\begin{itemize}
\item for every non-minimal $x\in T$, $\rho(x)\stree x$;
\item for every $y\in T$, $\rho^{-1}\{y\}$ is covered by less than $\kappa$ many antichains;
\end{itemize}
\item \emph{specializable} iff there exists a forcing extension with the same cardinal structure up to and including $\kappa$,
in which $T$ is special;
\item \emph{almost-Kurepa} iff it is a $\kappa$-tree and $|\mathcal B(T)|>\kappa$
holds in the forcing extension by $(T,{\supseteq})$.
\end{itemize}
\end{defn}

Any $\kappa$-Aronszajn tree contains a $\kappa$-subtree that is normal.
Inspecting the proofs of Facts \ref{simple-Souslin-construction} and~\ref{careful-Souslin-construction},
we observe that the constructed trees themselves are normal,
and indeed all Souslin trees showcased here are normal.
As a result of taking the simplest approach in the construction of Fact~\ref{simple-Souslin-construction},
the tree is also club-regressive,
as explored in~\cite[Proposition~2.3]{paper22}.
\begin{remark}\label{rmk46}
Any special $\kappa$-tree is a $\kappa$-Aronszajn tree that is not $\kappa$-Souslin.
An $\aleph_1$-tree is specializable iff it is Aronszajn.
Coherent trees are regressive;
regressive streamlined subtrees of $^{<\kappa}\omega$ are slim;
regressive $\kappa$-trees contain no $\nu$-Cantor subtrees for any infinite cardinal $\nu$;
stationarily-regressive $\kappa$-trees contain no $\nu$-Aronszajn subtrees for all $\nu\in\reg(\kappa)$;
slim splitting trees are not $\aleph_1$-complete;
a full splitting $\aleph_2$-tree is neither slim nor $\aleph_1$-complete;
the existence of a binary $\kappa$-Souslin tree is equivalent to the existence of a prolific $\kappa$-Souslin tree;
prolific $\kappa$-Souslin trees (are $\omega$-splitting and)
induce strong colorings on $\kappa$ in a very transparent way.
\end{remark}

\begin{defn}[{\cite[Definition~1.2]{rinot20}}]\label{ascentlaver} Suppose that $X$ is a streamlined $\kappa$-tree,
and $\mathcal F\s\mathcal P(\mu)$ for some cardinal $\mu$.
An \emph{$(\mathcal F,X)$-ascent path} through a streamlined tree $T$ of height $\kappa$ is
a system $\vec f =\langle f_x\mid x\in X\rangle$ such that for all $x,y \in X$:
\begin{enumerate}
\item $f_x : \mu\rightarrow T_{\dom(x)}$ is a function;
\item if $x\stree y$, then $\{ i<\mu\mid f_x (i) \stree f_y (i) \}\in \mathcal F$;
\item if $x \neq y$ and $\dom(x) = \dom(y)$, then $\{ i<\mu\mid f_x (i) \neq f_y (i) \}\in \mathcal F$.
\end{enumerate}

If $(X,{\stree})$ is isomorphic to $(\kappa,{\in})$, then $\vec f$ is simply said to be an \emph{$\mathcal F$-ascent path}.
If, in addition, $\mathcal F$ is equal to $\mathcal F^{\bd}_\mu:=\{ Z\s\mu\mid \sup(\mu\setminus Z)<\mu\}$, then $\vec f$ is said to be a \emph{$\mu$-ascent path}.
\end{defn}

The following weakening of $\mu$-ascent path was isolated by L\"ucke:
\begin{defn}[{\cite[Definition~1.3]{lucke}}]\label{ascending}
For a cardinal $\mu$,
a \emph{$\mu$-ascending path} through a streamlined tree $T$ of height $\kappa$ is
a sequence $\vec{f} =\langle f_\alpha\mid \alpha<\kappa \rangle$ such that for all $\alpha<\beta<\kappa$:
\begin{enumerate}
\item $f_\alpha : \mu\rightarrow T_{\alpha}$ is a function;
\item there are $i,j<\mu$ such that $f_\alpha(i) \stree f_\beta(j)$.
\end{enumerate}
\end{defn}
\begin{remark}\label{remark-specializable} Ascent paths and ascending paths are combinatorial features that play an important role in
enabling constructions of both specializable and non-specializable trees.
First, under certain `mismatch' conditions,
these paths prevent a tree from being special.
Specifically, for a tree $T$ of height $\kappa$, any of the following implies that $T$ is non-special:
\begin{enumerate}
\item $T$ admits a $\mu$-ascent path for some $\mu\in\reg(\lambda)\setminus\{\cf(\lambda)\}$,
where $\kappa=\lambda^+$;\footnote{This is what the proof of \cite[Lemma~3]{Sh:279} shows.}
\item $T$ admits a $\mu$-ascending path for some $\mu<\cf(\sup(\reg(\kappa)))$.\footnote{This is \cite[Corollary~1.7]{lucke}.}
\setcounter{condition}{\value{enumi}}
\end{enumerate}

But as these mismatch conditions remain valid in any forcing extension with the same cardinal structure,
they in fact imply that the tree $T$ is non-specializable!
Complementary to that, by~\cite[Theorem~1.19(1)]{MR3724382}, for every singular cardinal $\lambda$,
$\square_\lambda$ yields a special $\lambda^+$-tree with a $\cf(\lambda)$-ascent path.
As the existence of a $\mu$-ascent path through a tree $T$ implies the existence of a $\mu'$-ascending path for all $\mu'\geq\mu$,
it follows that the above mismatch conditions are sharp.

As for the other direction, we have the following:
\begin{enumerate}
\setcounter{enumi}{\value{condition}}
\item Assuming $\lambda^{<\lambda}=\lambda$,
every tree $T$ of height $\lambda^+$ that has no $\mu$-ascending path for every $\mu<\lambda$
is specializable.\footnote{By the forward implication of \cite[Theorem~1.11]{lucke}, under these assumptions,
the forcing poset $\mathbb P_\lambda(T)$ to add a specializing map $f:T\rightarrow\lambda$ via approximations of size less than $\lambda$
has the $\lambda^+$-cc.}
\end{enumerate}

The upshot is that to construct a non-specializable tree, it suffices to construct one with a mismatching ascent/ascending path,
and to construct a specializable tree at the successor $\lambda^+$ of a (regular) cardinal $\lambda=\lambda^{<\lambda}$,
it is necessary and sufficient that the tree admit no narrow ascending path.
\end{remark}

\begin{notation}[$i^\text{th}$ component, {\cite[Notation~4.3]{paper32}}]\label{notationcomp}
For every function $x:\alpha\rightarrow{}^\tau H_\kappa$ and every $i<\tau$,
we let $(x)_i:\alpha\rightarrow H_\kappa$ stand for $\langle x(\beta)(i)\mid \beta<\alpha\rangle$.
\end{notation}
\begin{defn}[{\cite[Definition~4.4]{paper32}}]\label{derived}
Suppose that $T \s {}^{<\kappa}H_\kappa$ is a streamlined tree,
and $\tau$ is a nonzero ordinal.
\begin{itemize}
\item For a sequence $\vec{s}=\langle s_i \mid i < \tau \rangle$ of nodes of $T$,
we let
$$T(\vec{s}):=\{x\in{}^{<\kappa}({}^\tau H_\kappa)\mid\forall i<\tau\,[(x)_i\cup s_i\in T]\},$$
noting that $T(\vec{s})$ is a streamlined subtree of ${}^{<\kappa}({}^\tau H_\kappa)$.
\item A \emph{$\tau$-derived tree of $T$}
is a tree of the form $T(\vec{s})$
for some injective sequence $\vec{s}=\langle s_i \mid i < \tau \rangle$ of nodes of $T$ on which the map $i\mapsto\dom(s_i)$ is constant.
\end{itemize}
\end{defn}

\begin{defn}[{\cite[Definition~4.5]{paper32}}]
A streamlined $\kappa$-tree $T$ is \emph{$\chi$-free} iff for every nonzero $\tau<\chi$, all of the $\tau$-derived trees of $T$ are $\kappa$-Souslin.
$T$ is \emph{free} iff it is $\omega$-free.
\end{defn}

\begin{remark}\label{rmk412} A $\lambda$-free $\lambda^+$-tree is specializable;
$3$-free trees are rigid;
$2$-free $\kappa$-trees are $\kappa$-Souslin and normal;
full normal $\kappa$-Aronszajn trees are rigid on every club \cite[Observation~2.2(4)]{paper62}.
\end{remark}

\begin{defn}[The levels of vanishing branches, \cite{paper48,paper58}]
For a streamlined $\kappa$-tree $T$:
\begin{itemize}
\item $ V^-(T):=\{\alpha\in\acc(\kappa)\mid \mathcal B(T\restriction\alpha)\neq T_\alpha\}$;
\item $ V(T)$ denotes the set of all $\alpha\in\acc(\kappa)$ such that
for every $x\in T\restriction\alpha$, there exists $f\in\mathcal B(T\restriction\alpha)\setminus T_\alpha$
with $x\stree f$.
\end{itemize}
\end{defn}

\begin{remark}\label{rmk414}
$V(T) \subseteq V^-(T)$;
if $T$ is uniformly homogeneous, then $V(T)=V^-(T)$;
if $V(T)$ is cofinal in $\kappa$, then $T$ is normal;
if $T$ is splitting and full, then $V(T)$ is empty;
if $T$ is uniformly coherent, then $V(T)=E^\kappa_\omega$.
\end{remark}

\subsection{Binary vs.~prolific and slim vs.~complete}\label{slimvscomplete}
As demonstrated in~\cite[\S6.1]{paper23},
there is an obvious way of transforming any proxy-based construction
of a prolific $\kappa$-tree into a construction of a binary $\kappa$-tree (or, more generally, a construction of a $\varsigma$-splitting subtree of ${}^{<\kappa}\varsigma$, for any fixed $\varsigma\in[2,\kappa)$),
and vice versa.
In addition, there are abstract translations of $\kappa$-trees into binary $\kappa$-trees as may be found in the appendix of \cite{rinot20}.

Likewise, there is a transparent way of transforming any proxy-based construction
of slim tree into a construction of a complete tree, and vice versa.
This is demonstrated by the construction of a $\chi$-complete $\kappa$-Souslin tree from
$\xbox^-(E^\kappa_{\geq\chi}) \land \diamondsuit(\kappa)$, where $\kappa$ is $({<}\chi)$-closed,
in~\cite[Proposition~2.4]{paper22}.
By taking some extra care in the construction,
\cite[Proposition~2.2]{paper26} shows that we can replace $\xbox^-(E^\kappa_{\geq\chi})$
with the weaker instance $\xbox^*(E^\kappa_{\geq\chi})$.
As a rule of thumb, the construction of slim trees requires $\mu\le\aleph_1$;
on the other hand, for a $\chi$-complete tree we require $\kappa$ to be $({<}\chi)$-closed
and also require the parameter $\mathcal{S}$ to contain some subset (modulo nonstationary) of $E^\kappa_{\geq\chi}$.

\subsection{The tables}\label{thetables}
We now turn to present a few tables summarizing various $\kappa$-Souslin trees constructed in the literature using instances of the proxy principles.
Note that the monotonicity features of the proxy principles
suggest an informal way of comparing two $\kappa$-trees $T$ and $S$ by viewing $T$ as `weaker' than $S$
provided that $T$ can be obtained from a vector of parameters weaker than
the one necessary for the construction of $S$.
This informal understanding becomes more precise through the observation that the content of Remarks \ref{rmk46}, \ref{rmk412} and \ref{rmk414}
indeed corresponds with Subsection~\ref{monotonicity}.

All trees showcased here are streamlined, normal, and splitting.
Throughout the tables in this section,
$\chi$ stands for an infinite regular cardinal $<\kappa$ such that $\kappa$ is $({<}\chi)$-closed.
On a first pass, the reader may simply assume $\chi=\aleph_0$.
In this case, $\sqx$ coincides with $\sq$ (recall Subsection~\ref{monotonicity}), and any mention of $E^\kappa_{\ge\chi}$ may be replaced by $\kappa$, since the two sets are equal modulo nonstationary.
We note that in many of the cited references the trees are constructed from $\p(\pvec)$ with some finitary value of $\sigma$,
but as explained in Subsection~\ref{subsection:bullet},
all such constructions can be carried out from the weaker $\p^\bullet(\pvec)$.

\medskip

Our first table,
Table~\ref{table:mu=nu=kappa},
presents $\kappa$-Souslin trees $T$ constructed from $\p^\bullet(\kappa,\allowbreak\mu,\mathcal{R},\theta,\mathcal{S},\nu)$
with the weak values $\mu=\nu=\kappa$.
In~(3), $X$ is any given streamlined $\kappa$-tree.
In~(4), $S\s \acc(\kappa)\cap E^\kappa_{<\chi}$.
In (5), $S\s\acc(\kappa)$.
In (6), $S\s \acc(\kappa)$ and we also assume that $\kappa$ is strongly inaccessible.

\begin{table}[H]
\begin{tabular}{l|l|c|c|c|l}
&Citation&$\mathcal R$&$\theta$&$\mathcal S$&Type of $\kappa$-Souslin tree\\ \hline\hline

(1)&\cite[Thm~6.8]{paper23}&$\sqleft{\chi}^*$&$1$&$\{E^\kappa_{\geq\chi}\}$& $\chi$-complete\\ \hline

(2)&\cite[Thm~6.32]{paper23}&$\sqleft{\chi}$&$\kappa$&$\{E^\kappa_{\geq\chi}\}$&$\chi$-complete, \\
&&&&& uniformly homogeneous \\\hline

(3)&\cite[Thm~3.7]{paper58}&$\sq$&$1$&$\{\kappa\}$&$V(T)\supseteq V^-(X)\cap E^\kappa_{>\omega}$\\ \hline

(4)&\cite[Thm~4.4]{paper58}&$\sqleftup{S}$&$1$&$\{S\cup E^\kappa_{\ge\chi}\}$&$V(T)\cap E^\kappa_{<\chi}=S$\\
&&&&& \hfill $=V^-(T)\cap E^\kappa_{<\chi}$ \\ \hline

(5)&\cite[Thm~4.3]{paper58}&$\sqleftup{S}$&$1$&$\{\kappa\}$&$V(T)\supseteq S$\\ \hline

(6)&\cite[Thm~4.8]{paper58}&$\sqleftup{S}$&$1$&$\{S\}$&$V(T)=S=V^-(T)$\\ \hline

\end{tabular}
\smallskip
\caption{The case $\mu=\nu=\kappa$.}
\label{table:mu=nu=kappa}
\end{table}

Our second table,
Table~\ref{table:2<mu<kappa},
presents $\kappa$-Souslin trees constructed from $\p^\bullet(\kappa,\mu,\allowbreak\mathcal{R},\theta,\mathcal{S},\nu)$
with a moderate strengthening of the value of $\mu$,
and still no restriction on $\nu$.

\begin{table}[H]
\begin{tabular}{l|c|c|c|c|l}
Citation&$\mu$&$\mathcal R$&$\theta$&$\mathcal{S}$&Type of $\kappa$-Souslin tree\\ \hline\hline

\cite[Thm~6.11]{paper23}&$\mu^{\ind}$&$\sq$&$1$&$\{E^\kappa_{\ge\chi}\}$&$\chi$-complete with a $\mu$-ascent path\\ \hline

See \S\ref{slimvscomplete} above & $\aleph_1$ & $\sq$ & $\kappa$ & $\{\kappa\}$ & slim, uniformly homogeneous \\ \hline

See \S\ref{slimvscomplete} above & $\aleph_1$ & $\sq^*$ & $1$ & $\{\kappa\}$ & slim \\ \hline

\end{tabular}
\smallskip
\caption{Cases where $2<\mu<\kappa$ and $\nu=\kappa$.}
\label{table:2<mu<kappa}
\end{table}

Our third table,
Table~\ref{table:mu=kappa;nu=2}, presents $\kappa$-Souslin trees $T$ constructed from $\p^\bullet(\kappa,\mu,\allowbreak\mathcal{R},\theta,\mathcal{S},\nu)$,
where $\mu$ is assigned the weak value $\kappa$,
but $\nu$ is now strengthened to $2$.
In (1) and~(2), $\lambda<\kappa$ is an infinite cardinal,
and $\mathcal S^*:=\{E^\kappa_{\geq\chi}\cap E^\kappa_{>\Lambda}\mid\Lambda<\lambda\}$;
note that if $\kappa=\lambda^+$ where $\lambda^{<\lambda}=\lambda$,
then the trees obtained are specializable by Remark~\ref{remark-specializable}(3).

\begin{table}[H]
\begin{tabular}{l|l|l|l|c|l}
&{Citation} &$\mathcal R$&$\theta$&$\mathcal S$&Type of $\kappa$-Souslin tree\\ \hline\hline

(1)&\cite[Thm~6.17]{paper23}&$\sqleft{\chi}$&$1$&$\mathcal S^*$&$\chi$-complete, rigid, \\
&&&&&$\forall \Lambda<\lambda$, $T$ has no $\Lambda$-ascending path\\ \hline

(2)&\cite[Thm~6.14]{paper23}&$\sqleft{\chi}^*$& $1$&$\mathcal S^*$&$\chi$-complete,\\
&&&&&$\forall \Lambda<\lambda$, $T$ has no $\Lambda$-ascending path\\ \hline

(3)&\cite[Thm~6.27]{paper23}&$\sqleft{\chi}$&$\kappa$&$\{E^\kappa_{\geq\chi}\}$&$\chi$-complete, $\chi$-free\\ \hline

(4)&Thm~\ref{thm44} below &$\sq$ &$\kappa$ &$\{\kappa\}$ &pairwise-Souslin family of trees \\\hline

\end{tabular}
\caption{The case $\mu=\kappa$ and $\nu=2$.}
\label{table:mu=kappa;nu=2}
\end{table}

Our fourth table,
Table~\ref{table:mu=nu=2}, presents $\kappa$-Souslin trees $T$ constructed from
$\p^\bullet(\kappa,\mu,\allowbreak\mathcal{R},\theta,\mathcal{S},\nu)$
where both $\mu$ and $\nu$ are strengthened to $2$.
In~(1), $X$ stands for a given streamlined $\kappa$-tree.
In~(4), the original paper does not mention being club-regressive, but this can easily be verified.
In~(5), $X$ stands for a given $\p^-(\pvec)$-respecting binary $\kappa$-tree.
In~(6), $X$ stands for a given $\p^-(\pvec)$-respecting $\kappa$-tree with no $\kappa$-sized antichains,
and $\eta\le\chi$.
In~(7), $\kappa=\lambda^+$ is a successor cardinal, $n>1$ is an integer,
and the order-type limitation is $\xi=\lambda$.

\begin{table}[H]
\begin{tabular}{l|l|l|l|c|l}

&{Citation}&$\mathcal R$&$\theta$&$\mathcal S$&Type of $\kappa$-Souslin tree\\ \hline\hline
(1)&\cite[Thm~3.7]{paper58}&$\sq^*$&1&$\{\kappa\}$& $V(T)\supseteq V^-(X)$\\ \hline
(2)&\cite[Prop~2.3]{paper22}&$\sq$&1&$\{\kappa\}$& club-regressive\\ \hline
(3)&\cite[Prop~2.5]{paper22}&$\sq$&$\kappa$&$\{\kappa\}$&club-regressive, uniformly coherent\\ \hline
(4)&\cite[Thm~6.2]{rinot20}&$\sq$&$\kappa$&$\{E^\kappa_{\ge\chi}\}$& club-regressive, $\chi$-free\\ \hline
(5)&\cite[Thm~7.2]{yadai}&$\sq$&$\kappa$&$\{E^\kappa_{\ge\chi}\}$&has $|\mathcal B(X)|$-many automorphisms; \\
&&&&&if $X$ is Kurepa, \\
&&&&&then $T$ is almost-Kurepa\\ \hline
(6)&\cite[Thm~5.10]{paper58}&$\sqleft{\eta}$&$\kappa$&$\{E^\kappa_{\ge\chi}\}$&uniformly homogeneous, \\ &&&&& $E^\kappa_{\geq\chi}$-regressive,\\
&&&&&$\chi$-complete, $\chi$-coherent; \\
&&&&& $X\otimes T$ is $\kappa$-Souslin and \\
&&&&& $\p^-(\kappa,2,{\sqleft{\eta}},\kappa,\{E^\kappa_{\geq\chi}\},2)$-respecting
\\\hline
(7)&Thm~\ref{thm710} below&$\sq$&$\kappa$&$\{\kappa\}$&club-regressive, $n$-free,
\\
&&&&& with special $n$-power \\ \hline

\end{tabular}
\caption{The case $\mu=\nu=2$.}
\label{table:mu=nu=2}
\end{table}

Our last table, Table~\ref{table:sigma=omega}, presents $\kappa$-Souslin trees $T$ with $(\mathcal F^\eta_\Lambda,X)$-ascent paths constructed from
$\p(\kappa,\mu,\mathcal{R},\theta,\mathcal{S},\nu,\sigma)$
with $\mu=\nu=2$ together with the strong value $\sigma = \omega$.
$\mathcal F^\eta_\Lambda$ stands for the filter $\{Z\s\Lambda\mid |\Lambda\setminus Z|<\eta\}$,
and here $\eta\leq\Lambda$ are both in $\reg(\kappa)$.
$X$ stands for a given streamlined $\kappa$-tree.
If $X$ admits a $\kappa$-branch (e.g., $X={}^{<\kappa}1$),
then since $\mathcal{F}^\eta_\Lambda$ is a subset of $\mathcal{F}^\bd_\Lambda$ that projects to $\mathcal{F}^\bd_\eta$,
the tree $T$ admits both a $\Lambda$-ascent path and an $\eta$-ascent path.
By Remark~\ref{remark-specializable}(1)\&(2), then, $T$ is non-specializable provided that
either $\Lambda\neq\cf(\sup(\reg(\kappa)))$ or $\Lambda\neq\eta$.
In (1), (3), and~(6), we require $X$ to be slim;
in (5) and~(6), $X$ must be $\p^-(\pvec)$-respecting.
In~(6), $\chi\leq\eta$.

\begin{table}[H]
\begin{tabular}{l|l|l|l|c|l}

&{Citation}&$\mathcal R$&$\theta$&$\mathcal S$&Type of $\kappa$-Souslin tree\\
& & & & & with an $(\mathcal F^\eta_\Lambda,X)$-ascent path\\ \hline\hline
(1)&\cite[Thm~4.2]{rinot20}&$\sq$&1&$\{\kappa\}$& slim with $\eta=\Lambda=\omega$\\ \hline
(2)&\cite[Thm~4.3]{rinot20}&$\sq$&1&$\{E^\kappa_{\geq\chi}\}$& $\chi$-complete with $\eta=\Lambda=\omega$\\ \hline
(3)&\cite[Thm~5.1]{rinot20}&$\sq_\eta$&$\Lambda$&$\{\kappa\}$& slim\\ \hline
(4)&\cite[Thm~5.3]{rinot20}&$\sq_\eta$&$\Lambda$&$\{E^\kappa_{\geq\chi}\}$& $\chi$-complete\\ \hline

(5)&\cite[Thm~6.3]{rinot20}&$\sq_\eta$&$\kappa$&$\{E^\kappa_{\geq\chi}\}$& $\min\{\chi,\eta\}$-free, $\chi$-complete\\ \hline

(6)&\cite[Thm~6.5]{rinot20}&$\sq_\eta$&$\kappa$&$\{E^\kappa_{\geq\chi}\}$& slim, $\chi$-free\\ \hline
\end{tabular}
\smallskip
\caption{The case $\mu=\nu=2$ and the strong value $\sigma=\omega$.}
\label{table:sigma=omega}
\end{table}

There are a few more tree constructions that do not fit the above tables.
We list them below and refer the reader to the original papers for any missing definitions.

\begin{fact}[{\cite[Theorem~6.1]{rinot20}}]
Suppose that $\p(\kappa, 2,{\sqsubseteq}, \kappa, \{E^\kappa_{\ge\chi}\},2,\omega)$ holds,
and $\kappa$ is (${<}\chi$)-closed.
For every infinite cardinal $\theta$ such that $\theta^+<\chi$,
there is a slim $(\chi,\theta^+)$-free $\kappa$-Souslin tree with an injective $\mathcal F^{\fin}_{\theta}$-ascent path.
\end{fact}

\begin{fact}[{\cite[Theorem~6.4]{rinot20}}]
Suppose that $\cf(\nu)=\nu<\theta^+<\chi<\kappa$ are infinite cardinals,
$\kappa$ is (${<}\chi$)-closed, and $\p(\kappa, 2,{\sqsubseteq}, \kappa, \{E^\kappa_{\ge\chi}\},2,\omega)$ holds.
Then there exists a $\nu$-free, $(\chi,\theta^+)$-free, $\chi$-complete $\kappa$-Souslin tree with an injective $\mathcal F^{\nu}_\theta$-ascent path.
\end{fact}
\begin{fact}[{\cite[Theorem~6.6]{rinot20}}]
Suppose that $\cf(\nu)=\nu<\theta^+<\chi<\kappa$ are infinite cardinals,
$\kappa$ is (${<}\chi$)-closed, and $\p(\kappa, 2,{\sqsubseteq}, \kappa, \{E^\kappa_{\ge\chi}\},2,\omega)$ holds.
Then there exists a slim $\nu$-free, $(\chi,\theta^+)$-free, $\kappa$-Souslin tree with an injective $\mathcal F^{\nu}_\theta$-ascent path.
\end{fact}

\begin{fact}[{\cite[Theorem~7.5]{yadai}}]
Suppose that:
\begin{itemize}
\item $\kappa^{<\kappa}=\kappa$;
\item $S\s E^\kappa_{\cf(\theta)}$ is a stationary subset of $\kappa$;
\item $\p^-(\kappa,\kappa,{\sqleftup{S}},1,\{S\},2)$ holds and is witnessed by a sequence $\langle\mathcal C_\alpha\mid\alpha<\kappa\rangle$ for which
$B:=\{ \alpha\in\acc(\kappa)\mid |\mathcal C_\alpha|=1\}$ covers $E^\kappa_{> \cf(\theta)}$, and $\acc(\bigcup\mathcal C_\alpha)\s B$ for every $\alpha\in B$.
\end{itemize}

Then there exists a $\kappa$-Souslin tree with a $\theta$-ascent path.
\end{fact}

\begin{fact}[{\cite[Theorem~4.4]{paper62}}] Suppose that:
\begin{itemize}
\item $\kappa$ is a strongly inaccessible cardinal;
\item $S\s E^\kappa_{>\omega}$ is stationary, and $\diamondsuit^*_S(\kappa\textup{-trees})$ holds;
\item $\p^-(\kappa,2,{\sq},1,\{S\})$ holds.
\end{itemize}

Then there is a family $\mathcal T$ of $2^\kappa$ many binary, full $\kappa$-trees
such that $\bigotimes\mathcal T'$ is $\kappa$-Souslin for every nonempty $\mathcal T'\in[\mathcal T]^{<\kappa}$.
\end{fact}

\begin{fact}[{\cite[Theorem~5.1]{paper62}}] Suppose that:
\begin{itemize}
\item $\kappa=\lambda^+=2^\lambda$ for $\lambda$ a regular uncountable cardinal;
\item $\square^B_\lambda$ and $\diamondsuit(\lambda)$ both hold;\footnote{$\square^B_\lambda$ asserts the existence of a $\lambda$-bounded $\sq_\lambda$-coherent $C$-sequence over $\lambda^+$.}
\item $\p^-(\kappa,2,{\sq_\lambda},\kappa,\{E^\kappa_\lambda\})$ holds.
\end{itemize}

Then there is a family $\mathcal T$ of $2^\kappa$ many binary, full $\kappa$-trees
such that $\bigotimes\mathcal T'$ is $\kappa$-Souslin for every nonempty $\mathcal T'\in[\mathcal T]^{<\lambda}$.
\end{fact}

\begin{fact}[{cf.~\cite[Corollary~1.15]{rinot20}}]
Suppose that:
\begin{itemize}
\item $\kappa=\lambda^+$ for some regular uncountable cardinal $\lambda$;
\item $\kappa$ is $({<}\lambda)$-closed;
\item $\nu\in\reg(\kappa)$;
\item $\theta$ is a cardinal such that $\nu<\theta^+<\lambda$;
\item $\p(\kappa,2,{\sq},\kappa,\{E^\kappa_\lambda\},2,\omega)$ holds.
\end{itemize}

Then there is a $\kappa$-Souslin tree $T$ such that for every infinite cardinal $\Lambda$:
\begin{itemize}
\item If $\Lambda<\nu$ or $\theta<\Lambda<\lambda$,
then for every selective ultrafilter $\mathcal{U}$ over $\Lambda$,
the reduced power $T^\Lambda/\mathcal{U}$ is an almost-Souslin $\kappa$-Aronszajn tree;
\item If $\nu\le\cf(\Lambda)\le\Lambda\le\theta$,
then for every uniform ultrafilter $\mathcal{U}$ over $\Lambda$,
the reduced power $T^\Lambda/\mathcal{U}$ admits a $\kappa$-branch;
\item If $\Lambda=\lambda$, then for every $(\omega,\Lambda)$-regular ultrafilter $\mathcal{U}$ over $\Lambda$,
the reduced power $T^\Lambda/\mathcal{U}$ is not a $\kappa$-tree.
\end{itemize}
\end{fact}
\section{A large family of Souslin trees}\label{section:Zakrzewski}
In \cite{MR638747}, Zakrzewski constructed from $\diamondsuit(\aleph_1)$ a family of $2^{\aleph_1}$ many $\aleph_1$-Souslin trees
such that the product of any finitely (nonzero) many of them is again Souslin.
The main result of this section (Theorem~\ref{thm44} below) generalizes Zakrzewski's theorem in various ways.
First, let us recall the definition of a product of trees.
\begin{defn}\label{def-classical-product-tree}
For a sequence of streamlined trees $\langle T^j \mid j<\tau \rangle$,
the product tree $\bigotimes_{j<\tau} T^j$
is defined to be the poset $\mathbf T=({T}, {<_{{T}}})$,
where:
\begin{itemize}
\item $T:=\{ \vec x\in \prod_{j<\tau}T^j\mid j\mapsto\dom(\vec x(j))\text{ is constant}\}$, and
\item $\vec{x} <_{{T}} \vec{y}$ iff $\vec x(j)\stree \vec y(j)$ for every $j<\tau$.
\end{itemize}
\end{defn}
\begin{remark} The tree $\mathbf T$ is easily seen to be isomorphic to a streamlined tree via Notation~\ref{notationcomp},
but in order to ease on the reader, we stick here to the classical representation of products.
\end{remark}

Second, to motivate Theorem~\ref{thm44},
we state a sample corollary that does not mention products.
In what follows, two streamlined $\kappa$-trees $S$ and $T$ are \emph{far} iff there exist no streamlined $\kappa$-subtrees of $S$ and $T$ that are club-isomorphic.
\begin{cor} If $\diamondsuit^+(\aleph_1)$ holds, then there exists a streamlined $\aleph_1$-Aronszajn tree $T$
admitting $\aleph_2$-many pairwise far streamlined $\aleph_1$-Souslin subtrees.
\end{cor}
\begin{proof} Recall that $\diamondsuit^+(\aleph_1)$ entails
the existence of an $\aleph_1$-Kurepa tree.
In addition, by \cite[Corollary~1.10]{paper22}, $\diamondsuit(\aleph_1)$ implies $\p(\aleph_1,2,{\sq},\aleph_1)$.
Thus, by Corollary~\ref{cor46} below (using $\kappa=\aleph_1$ and $\mho=\aleph_2)$,
there exists a streamlined $\aleph_1$\nobreakdash-Aronszajn tree admitting $\aleph_2$-many streamlined $\aleph_1$-subtrees such that the product of any two of them is $\aleph_1$-Souslin.
Finally, note that, in general, if $S$ and $T$ are streamlined $\kappa$-trees whose product is $\kappa$-Souslin,
then they are far. Indeed, in this case, for all streamlined $\kappa$-subtrees $S'$ of $S$ and $T'$ of $T$ and for every club $D\s\kappa$,
the product of $S'\restriction D$ and $T'\restriction D$ is $\kappa$-Souslin,
and then \cite[Proposition~2.6]{paper62} implies that there is no weak embedding from $S'\restriction D$ to $T'\restriction D$,
let alone an isomorphism.
\end{proof}

Looking at the preceding corollary, one may wonder whether the conclusion can be strengthened to make the ultimate tree Souslin, as well. The next proposition shows that this is impossible.
\begin{prop} Suppose $\langle T^\eta\mid \eta<\kappa\rangle$ is a sequence of streamlined $\kappa$-subtrees of a given streamlined $\kappa$-Souslin tree $T$.
Then there are $\eta<\rho<\kappa$ such that $|T^\eta\cap T^\rho|=\kappa$.
\end{prop}
\begin{proof} We shall make use of the following standard feature of Souslin trees.
\begin{claim} Suppose that $S$ is a streamlined $\kappa$-subtree of $T$. Then there exists some $s\in S$ such that $\{ t\in T\mid s\s t\}$ is a subset of $S$.
\end{claim}
\begin{proof} Suppose not.
In particular, for every $\alpha<\kappa$, we may find a pair $(s_\alpha,t_\alpha)$ such that:
\begin{itemize}
\item $s_\alpha\in S_\alpha$;
\item $t_\alpha\in T\setminus S$ with $s_\alpha\s t_\alpha$.
\end{itemize}

Fix a sparse enough set $A\in[\kappa]^\kappa$ such that for every pair $\alpha<\beta$ of ordinals from $A$, $\dom(t_\alpha)<\beta$.
Since $T$ is a $\kappa$-Souslin tree, we may pick a pair $\alpha<\beta$ of ordinals from $A$ such that $t_\alpha\s t_\beta$. As also $s_\beta\s t_\beta$ and $\dom(t_\alpha)<\beta=\dom(s_\beta)$,
it follows that $t_\alpha\s s_\beta$. But $s_\beta$ belongs to the streamlined tree $S$, which must mean that $t_\alpha\in S$, contradicting the choice of $t_\alpha$.
\end{proof}

For each $\eta<\kappa$, pick some $s_\eta\in T$ such that $\{ t\in T\mid s_\eta\s t\}$ is a subset of $T^\eta$.
As $T$ is a $\kappa$-Souslin tree, the set $N:=\{ s\in T\mid |\{ t\in T\mid s\s t\}|<\kappa\}$ has size less than $\kappa$.
Using again that $T$ is a $\kappa$-Souslin tree, we may then find $\eta<\rho<\kappa$ such that $s_\eta\s s_\rho$ and $s_\rho\notin N$.
So $\{ t\in T\mid s_\rho\s t\}$ is a subset of $T^\eta\cap T^\rho$ of size $\kappa$.
\end{proof}

We now arrive at the main technical result of this section.
To recover Zakrzewski's theorem, consider the case $\kappa:=\aleph_1$, $K:={}^{<\kappa}2$ and $\mathcal S:=\{\kappa\}$.

\begin{thm}\label{thm44} Suppose:
\begin{itemize}
\item $K\s{}^{<\kappa}H_\kappa$ is a normal streamlined tree of height $\kappa$;
\item $\mathcal S$ is a nonempty collection of stationary subsets of $\kappa$;
\item $\p(\kappa,\kappa,{\sq},\kappa,\mathcal S,2)$ holds.
\end{itemize}

Then there exists a system $\langle T^\eta \mid\eta\in\mathcal B(K)\rangle$ of prolific normal streamlined
$\kappa$-Souslin trees satisfying all of the following:
\begin{enumerate}
\item For every nonzero cardinal $\tau$
such that $\kappa$ is $\tau$-closed and such that there exists $\S\in\mathcal S$ for which $\S\setminus E^\kappa_{>\tau}$ is nonstationary,
for every injective sequence $\langle \eta_j\mid j<\tau\rangle$ of elements of $\mathcal B(K)$,
the product tree $\bigotimes_{j<\tau}T^{\eta_j}$ is again $\kappa$-Souslin;
\item The union of these trees $T:=\bigcup\{ T^\eta\mid \eta\in\mathcal B(K)\}$ has no $\kappa$-branches;
\item If $K$ is a $\kappa$-tree, then $T$ is a $\kappa$-Aronszajn tree.
\end{enumerate}
\end{thm}
\begin{proof}
Fix a sequence $\vec{\mathcal C}=\langle \mathcal C_\alpha\mid \alpha<\kappa\rangle$ witnessing $\p^-(\kappa,\kappa,{\sq},\kappa,\mathcal S,2)$.
Without loss of generality, $0\in\bigcap_{0<\alpha<\kappa}\bigcap\mathcal C_\alpha$.
As $\diamondsuit(\kappa)$ holds, fix sequences $\langle \Omega_\beta\mid\beta<\kappa\rangle$ and $\langle R_i\mid i<\kappa\rangle$
together witnessing $\diamondsuit(H_\kappa)$, as in Fact~\ref{diamondhkappa}.
Let $\pi:\kappa\rightarrow\kappa$ be such that $\beta\in R_{\pi(\beta)}$ for all $\beta<\kappa$.
Let $\lhd_\kappa$ be some well-ordering of $H_\kappa$ of order-type $\kappa$,
and let $\phi:\kappa\leftrightarrow H_\kappa$ witness the isomorphism $(\kappa,\in)\cong(H_\kappa,\lhd_\kappa)$.
Put $\psi:=\phi\circ\pi$.

We shall construct a system $\langle L^\eta\mid \eta\in K\rangle$ such that,
for all $\alpha<\kappa$ and $\eta\in K_\alpha$:
\begin{itemize}
\item[(i)] $L^\eta\in[{}^\alpha\kappa]^{<\kappa}$;
\item[(ii)] for every $\beta<\alpha$, $L^{\eta\restriction\beta}=\{t\restriction\beta\mid t\in L^\eta\}$.
\end{itemize}

By convention, for every $\alpha\in\acc(\kappa)$ such that $\langle L^\eta\mid \eta\in K\restriction\alpha\rangle$ has already been defined,
and for every $\eta\in K_\alpha$, we shall let $T^\eta:=\bigcup_{\beta<\alpha}L^{\eta\restriction\beta}$,
so that $T^\eta$ is a streamlined tree of height $\alpha$ whose $\beta^{\text{th}}$ level is $L^{\eta\restriction\beta}$ for all $\beta<\alpha$.
Throughout, we shall ensure that each such tree $T^\eta$ is normal and prolific.

\medskip

The construction of the system $\langle L^\eta\mid \eta\in K\rangle$ is by recursion on $\dom(\eta)$.
We start by letting $L^\emptyset:=\{\emptyset\}$.
Next, for every $\alpha<\kappa$ such that $\langle L^\eta\mid \eta\in K_\alpha\rangle$ has already been defined,
for every $\eta\in K_{\alpha+1}$, we let
$$L^\eta:=\{ t{}^\smallfrown\langle\iota\rangle\mid t\in L^{\eta\restriction\alpha}, \iota<\max\{\omega,\alpha\}\}.$$

Suppose now that $\alpha\in\acc(\kappa)$ is such that $\langle L^\eta\mid \eta\in K\restriction\alpha\rangle$ has already been defined.
For each $C\in\mathcal C_\alpha$, we shall define a matrix
$$\mathbb B^C=\langle b_x^{\alpha,C\cap\beta,\eta}\mid \beta\in C, \eta\in K_\beta, x\in T^\eta\restriction C\cap(\beta+1) \rangle$$
ensuring that $x\s b_x^{\alpha,\bar D,\bar\eta}\s b_x^{\alpha,D,\eta}\in L^\eta$ whenever $\bar\eta\s\eta$ and $\bar D\sq D$.\footnote{\label{fnc}This also implies that the matrix is continuous,
i.e., for $\beta\in\acc(C)$, $\eta\in K_\beta$ and $x\in T^\eta\restriction(C\cap\beta)$, it is the case that $b_x^{\alpha,C\cap\beta,\eta}=\bigcup\{b_x^{\alpha,C\cap\bar\beta,\eta\restriction\bar\beta}\mid \bar\beta\in C\cap\beta\setminus\dom(x)\}$.}
Then, for all $C\in\mathcal C_\alpha$, $\eta\in K_\alpha$ and $x\in T^\eta\restriction C$,
it will follow that $\mathbf b_x^{C,\eta}:=\bigcup_{\beta\in C\setminus\dom(x)}b_x^{\alpha,C\cap\beta,\eta\restriction\beta}$ is an element of $\mathcal B(T^\eta)$ extending $x$,
and we shall let
\begin{equation}\tag{$\star$}\label{promise}L^{\eta}:=\{\mathbf{b}_x^{C,\eta}\mid C\in\mathcal C_\alpha, x\in T^{\eta}\restriction C\}.\end{equation}

Let $C\in\mathcal C_\alpha$. We now turn to define the components of the matrix $\mathbb B^C$ by recursion on $\beta\in C$.
So suppose that $\beta\in C$ is such that
$$\mathbb B^C_{<\beta}:=\langle b_x^{\alpha,C\cap\bar\beta,\eta}\mid \bar\beta\in C\cap\beta, \eta\in K_{\bar\beta}, x\in T^\eta\restriction C\cap(\bar\beta+1) \rangle$$
has already been defined.

$\br$ For all $\eta\in K_\beta$ and $x\in T^\eta$ such that $\dom(x)=\beta$, let $b_x^{\alpha,C\cap\beta,\eta} := x$.

$\br$ For all $\eta\in K_\beta$ and $x\in T^\eta$ such that $\dom(x) \in C\restriction\beta$,
there are two main cases to consider:

$\br\br$ Suppose that $\beta\in \nacc(C)$ and denote $\beta^-:=\sup(C\cap\beta)$.

$\br\br\br$ Suppose $\beta\in \acc(\kappa)$
and there exists a nonzero cardinal $\tau$ such that all of the following hold:
\begin{enumerate}
\item There exists a sequence $\langle \eta_j\mid j<\tau\rangle$ of nodes in $K_\beta$,
and a maximal antichain $A$ in the product tree $\bigotimes_{j<\tau}T^{\eta_j}$
such that\footnote{As $\beta\in\acc(\kappa)$, it is the case that $\tau$,
$\langle \eta_j\mid j<\tau\rangle$ and $A$ are uniquely determined by~$\Omega_\beta$.}
$$\Omega_\beta=\{(\langle \eta_j\restriction\epsilon\mid j<\tau\rangle,A\cap{}^\tau(^\epsilon\kappa))\mid \epsilon<\beta\};$$
\item $\psi(\beta)$ is a sequence $\langle x_j\mid j<\tau\rangle$ such that $x_j\in T^{\eta_j\restriction\beta^-}\restriction(C\cap\beta^-)$ for every $j<\tau$;
\item There exists a unique $j<\tau$ such that $\eta_j=\eta$ and $x_j=x$.
\end{enumerate}
In this case, by Clauses (1) and (2), the following set is nonempty
$$Q^{C,\beta} := \left\{ \vec t\in \prod\nolimits_{j<\tau}L^{\eta_j}\,\middle|\, \exists\vec s\in A\forall j<\tau\left[ (\vec s(j)\cup b^{\alpha,C\cap\beta^-,\eta_j\restriction\beta^-}_{x_j})\s\vec t(j)\right]\right\},$$
so we let $\vec t:=\min(Q^{C,\beta},\lhd_\kappa)$,
and then we let $b^{\alpha,C\cap\beta,\eta}_x:=\vec t(j)$ for the unique index $j$ of Clause~(3).
It follows that $b^{\alpha,C\cap\beta^-,\eta\restriction\beta^-}_{x}\s\vec t(j)= b^{\alpha,C\cap\beta,\eta}_x$.

$\br\br\br$ Otherwise, let $b_x^{\alpha,C\cap\beta,\eta}$ be the $\lhd_\kappa$-least element of $L^{\eta}\setminus\{\Omega_\beta\}$ extending $b_x^{\alpha,C\cap\beta^-,\eta\restriction\beta^-}$.
As our trees thus far are normal and splitting (in fact, prolific), this is well-defined.

$\br\br$ Suppose that $\beta\in\acc(C)$. Then we define $b_x^{\alpha,C\cap\beta,\eta} := \bigcup\{b_x^{\alpha,C\cap\bar\beta,\eta\restriction\bar\beta}\mid \bar\beta\in C\cap\beta\setminus\dom(x)\}$.
We must show that the latter belongs to $L^{\eta}$.
By the coherence feature of $\vec{\mathcal C}$ and since $\beta\in\acc(C)$, it is the case that $C\cap\beta\in\mathcal C_\beta$,
so, by \eqref{promise}, it suffices to prove that $b_x^{\alpha,C\cap\beta,\eta}=\mathbf{b}_x^{C\cap\beta,\eta}$.
Proving the latter amounts to showing that
$b_x^{\alpha,C\cap\delta,\eta\restriction\delta}=b_x^{\beta,C\cap\delta,\eta\restriction\delta}$ for all $\delta\in C\cap\beta\setminus\dom(x)$.
This is taken care of by the following claim, formalizing
the fact that our construction is following the \emph{microscopic approach}.

\begin{claim}\label{cl441} $\mathbb B^C_{<\beta}=\mathbb B^{C\cap\beta}$. That is, the following matrices coincide:
\begin{itemize}
\item $\langle b_y^{\alpha,C\cap\delta,\rho}\mid \delta\in C\cap\beta, \rho\in K_{\delta}, y\in T^\rho\restriction C\cap(\delta+1) \rangle$;
\item $\langle b_y^{\beta,C\cap\delta,\rho}\mid \delta\in C\cap\beta, \rho\in K_{\delta}, y\in T^\rho\restriction C\cap(\delta+1) \rangle$.
\end{itemize}
\end{claim}
\begin{proof} For the scope of this proof we denote $C\cap\beta$ by $D$.
Now, by induction on $\delta\in D$, we prove that
$$\langle b_y^{\alpha,D\cap\delta,\rho}\mid \rho\in K_\delta, y\in T^\rho\restriction D\cap(\delta+1) \rangle=\langle b_y^{\beta,D\cap\delta,\rho}\mid \rho\in K_\delta, y\in T^\rho\restriction D\cap(\delta+1) \rangle.$$

The base case $\delta=\min(D)=0$ is immediate since $b_\emptyset^{\alpha,\emptyset,\emptyset}=\emptyset=b_\emptyset^{\beta,\emptyset,\emptyset}$.
The limit case $\delta\in\acc(D)$ follows from the continuity of the matrices under discussion as remarked in Footnote~\ref{fnc},
with the exception of those $y$'s such that $\dom(y)=\delta$,
but in this case, $b_y^{\alpha,D\cap\delta,\rho}=y=b_y^{\beta,D\cap\delta,\rho}$ for all $\rho\in K_\delta$.

Finally, assuming that $\delta^-<\delta$ are two successive elements of $D$ such that
$$\begin{aligned}&\langle b_y^{\alpha,D\cap\delta^-,\rho}\mid \rho\in K_{\delta^-}, y\in T^\rho\restriction D\cap({\delta^-}+1) \rangle\\=~&\langle b_y^{\beta,D\cap\delta^-,\rho}\mid \rho\in K_{\delta^-}, y\in T^\rho\restriction D\cap({\delta^-}+1) \rangle,\end{aligned}$$
we argue as follows. Given $\zeta\in K_\delta$ and $z\in T^\zeta\restriction D\cap(\delta+1)$, there are a few possible options.
If $\dom(z)=\delta$, then $b_z^{\alpha,D\cap\delta,\zeta}=z=b_z^{\beta,D\cap\delta,\zeta}$, and we are done.
If $\dom(z)<\delta$, then $\dom(z)\le\delta^-$ and, by the above construction,
for every $\gamma\in\{\alpha,\beta\}$, the value of
$b_z^{\gamma,D\cap\delta,\zeta}$ is completely determined by $\delta$, $\langle L^\rho\mid\rho\in K\restriction(\delta+1)\rangle$, $\Omega_\delta$, $D$,
$\psi(\delta)$, $\zeta$, $x$, and
$\langle b_y^{\gamma,D\cap\delta^-,\rho}\mid \rho\in K_{\delta^-},\allowbreak y\in T^\rho\restriction (D\cap\delta^-)\rangle$
in such a way that our inductive assumptions imply that $b_z^{\alpha,D\cap\delta,\zeta}=b_z^{\beta,D\cap\delta,\zeta}$.
\end{proof}

This completes the definition of the matrix $\mathbb B^C$,
from which we derive
$\mathbf b_x^{C,\eta}:=\bigcup_{\beta\in C\setminus\dom(x)}b_x^{\alpha,C\cap\beta,\eta\restriction\beta}$ for all $\eta\in K_\alpha$ and $x\in T^\eta\restriction C$.
Finally, we define $L^\eta$ as per \eqref{promise}.

\begin{claim}\label{c441} For all $\eta\in K_\alpha$, $C\in\mathcal C_\alpha$, and $t\in \{\mathbf{b}_x^{C,\eta}\mid x\in T^\eta\restriction C\}$, there exists a tail of $\varepsilon\in C$ such that $t=\mathbf b^{C,\eta}_{t\restriction\varepsilon}$.
\end{claim}
\begin{proof}
Let $t$ be as above. Fix an $x\in T^\eta\restriction C$ such that $t=\mathbf{b}_x^{C,\eta}$.
Then, by the nature of the above construction, $t\restriction\dom(x)=x$,
and for every given $\varepsilon\in C\setminus\dom(x)$, it is the case that
$t\restriction\varepsilon=b^{\alpha,C\cap\varepsilon,\eta\restriction\varepsilon}_x$,
and of course $b^{\alpha,C\cap\varepsilon,\eta\restriction\varepsilon}_{t\restriction\varepsilon}=t\restriction\varepsilon$.
That is, $b^{\alpha,C\cap\varepsilon,\eta\restriction\varepsilon}_x=b^{\alpha,C\cap\varepsilon,\eta\restriction\varepsilon}_{t\restriction\varepsilon}$.
Then, a similar analysis to that of Claim~\ref{cl441} yields that for every $\delta\in C\setminus\varepsilon$,
$b^{\alpha,C\cap\delta,\eta\restriction\delta}_x=b^{\alpha,C\cap\delta,\eta\restriction\delta}_{t\restriction\varepsilon}$.
Altogether, $t=\mathbf b^{C,\eta}_{x}=\mathbf b^{C,\eta}_{t\restriction\varepsilon}$.
\end{proof}

At the end of the above process, for every $\eta\in\mathcal B(K)$, we have obtained a normal streamlined prolific $\kappa$-tree $T^\eta:=\bigcup_{\alpha<\kappa}L^{\eta\restriction\alpha}$
whose $\alpha^{\text{th}}$ level is $L^{\eta\restriction\alpha}$.

\begin{claim}\label{c413} Suppose:
\begin{itemize}
\item $\tau$ is nonzero cardinal such that $\kappa$ is $\tau$-closed;
\item $\S\in\mathcal S$ is such that $\S\setminus E^\kappa_{>\tau}$ is nonstationary;
\item $\langle \eta_j\mid j<\tau\rangle$ is an injective sequence of elements of $\mathcal B(K)$.
\end{itemize}

The product tree $\bigotimes_{j<\tau}T^{\eta_j}$ is a $\kappa$-Souslin tree.
\end{claim}
\begin{proof} For the sake of this proof, denote $\bigotimes_{j<\tau}T^{\eta_j}$ by $\mathbf T=(T,{<_T})$.
As $\kappa$ is $\tau$-closed, $\mathbf T$ is a (splitting, normal) $\kappa$-tree.
Thus, to show that it is a $\kappa$-Souslin tree, it suffices to establish that it has no antichains of size $\kappa$.
To this end, let $A$ be a maximal antichain in $\mathbf T$.

Set $\Omega:=\{(\langle \eta_j\restriction\epsilon\mid j<\tau\rangle,A\cap{}^\tau(^\epsilon\kappa))\mid \epsilon<\kappa\}$.
As an application of $\diamondsuit(H_\kappa)$,
using the parameter $p:=\{\phi, A,\Omega,\langle T^{\eta_j}\mid j<\tau\rangle\}$,
we get that for every $i<\kappa$, the following set is stationary in $\kappa$:
$$B_i:=\{\beta\in R_i\cap\acc(\kappa)\mid \exists \mathcal M\prec H_{\kappa^+}\,(p\in \mathcal M, \beta=\kappa\cap\mathcal M, \Omega_\beta=\Omega\cap \mathcal M)\}.$$

Note that, for every $\beta\in \bigcup_{i<\kappa}B_i$,
it is the case that $T\restriction\beta\s\phi[\beta]$.

Fix a large enough $\delta<\kappa$ for which the map $j\mapsto\eta_j\restriction\delta$ is injective over $\tau$.
By the choice of $\vec{\mathcal C}$, we may now find an ordinal $\alpha\in \S\cap E_{>\tau}^\kappa$ above $\delta$ such that
$\mathcal C_\alpha$ is a singleton, say $\mathcal C_\alpha=\{C_\alpha\}$, and, for all $i<\alpha$,
$$\sup(\nacc(C_\alpha)\cap B_i)=\alpha.$$

In particular, $T\restriction\alpha\s\phi[\alpha]$.
Set $\bar\eta_j:=\eta_j\restriction\alpha$ for each $j<\tau$,
and note that $T\restriction\alpha=\bigotimes_{j<\tau}T^{\bar\eta_j}$.

\begin{subclaim} $A\s T\restriction\alpha$. In particular, $|A|<\kappa$.
\end{subclaim}
\begin{proof}
It suffices to show that every node in $T_\alpha$ extends some element of the antichain $A$.
To this end, let $\vec y=\langle y_j\mid j<\tau\rangle$ be an arbitrary node in $T_\alpha$.
By \eqref{promise}, for each $j<\tau$, we may find some $x_j\in T^{\bar\eta_j}\restriction C_\alpha$ such that $y_j=\mathbf b_{x_j}^{C_\alpha,\bar\eta_j}$.
By Claim~\ref{c441} and the fact that $\alpha\in E_{>\tau}^\kappa$,
we may assume the existence of a large enough $\gamma\in C_\alpha\setminus(\delta+1)$ such that $\dom(x_j)=\gamma$ for all $j<\tau$.
In particular, $\vec x:=\langle x_j\mid j<\tau\rangle$ is a node in $T\restriction\alpha\s\phi[\alpha]$.
Fix some $i<\alpha$ such that $\phi(i)=\vec x$,
and then pick a large enough $\beta\in \nacc(C_\alpha)\cap B_i$ for which $\beta^-:=\sup(C_\alpha\cap\beta)$ is bigger than $\gamma$.
Note that $\psi(\beta)=\phi(\pi(\beta))=\phi(i)=\vec x$,
$\langle \bar\eta_j\restriction\beta\mid j<\tau\rangle$ is an injective sequence,
and
$$\delta<\gamma<\beta^-<\beta<\alpha.$$

Let $\mathcal M\prec H_{\kappa^+}$ be a witness for $\beta$ being in $B_i$.
Clearly,
\begin{itemize}
\item $T\cap\mathcal M=T\restriction\beta=\bigotimes_{j<\tau}T^{\bar\eta_j\restriction\beta}$,
\item $A\cap\mathcal M=A\cap(T\restriction\beta)$ is a maximal antichain in $T\restriction\beta$, and
\item $\Omega_\beta=\Omega\cap \mathcal M=\{(\langle \eta_j\restriction\epsilon\mid j<\tau\rangle,A\cap{}^\tau(^\epsilon\kappa))\mid \epsilon<\beta\}$.
\end{itemize}

It thus follows that for every $j<\tau$,
$b^{\alpha,C_\alpha\cap\beta,\bar\eta_j\restriction\beta}_{x_j}=\vec t(j)$, where $\vec t=\min(Q^{C_\alpha,\beta},\allowbreak\lhd_\kappa)$.
In particular, we may fix some $\vec s\in A$ such that, for every $j<\tau$,
$$(\vec s(j)\cup b^{\alpha,C_\alpha\cap\beta^-,\bar\eta_j\restriction\beta^-}_{x_j})\s\vec t(j)=b^{\alpha,C_\alpha\cap\beta,\bar\eta_j\restriction\beta}_{x_j}\s\mathbf b_{x_j}^{C_\alpha,\bar\eta_j}=y_j.$$
So $\vec s<_T \vec y$. As $\vec s$ is an element of $A$, we are done.
\end{proof}
This completes the proof.
\end{proof}

As a final step, we consider the tree $T:=\bigcup\{ T^\eta\mid \eta\in\mathcal B(K)\}$.
Evidently, the $\alpha^{\text{th}}$ level of $T$ is the union of $|K_\alpha|$ many sets of size less than $\kappa$.
Thus, if $K$ is a $\kappa$-tree, then so is $T$.

\begin{claim} $T$ has no $\kappa$-branches.
\end{claim}
\begin{proof} Towards a contradiction, suppose that $f\in\mathcal B(T)$.
Fix an $i<\alpha$ such that $\phi(i)=\emptyset$.
As an application of $\diamondsuit(H_\kappa)$,
we get that the following set is stationary in $\kappa$:
$$B_i:=\{\beta\in R_i\mid f\restriction\beta=\Omega_\beta\}.$$
By the choice of $\vec{\mathcal C}$, we may now find an ordinal $\alpha\in\acc(\kappa)$ such that
$\mathcal C_\alpha$ is a singleton, say $\mathcal C_\alpha=\{C_\alpha\}$, and
$$\sup(\nacc(C_\alpha)\cap B_i)=\alpha.$$

Recalling \eqref{promise}, fix $\eta\in K_\alpha$ and $x\in T^{\eta}\restriction C_\alpha$ such that $f\restriction\alpha=\mathbf b_{x}^{C_\alpha,\eta}$.
Pick a large enough $\beta\in \nacc(C_\alpha)\cap B_i$ for which $\beta^-:=\sup(C_\alpha\cap\beta)$ is bigger than $\dom(x)$.
As $\psi(\beta)=\phi(\pi(\beta))=\phi(i)=\emptyset$,
it is the case that $b_x^{\alpha,C_\alpha\cap\beta,\eta\restriction\beta}$ is an element of $L^{\eta\restriction\beta}\setminus\{\Omega_\beta\}$.
In particular, $\mathbf b_x^{C_\alpha,\eta}\restriction\beta\neq \Omega_\beta$, contradicting the fact that $\mathbf b_x^{C_\alpha,\eta}=f\restriction\alpha$ and $\Omega_\beta=f\restriction\beta$.
\end{proof}

This completes the proof.
\end{proof}
\begin{remark} It is tedious yet not impossible to verify that for every nonzero cardinal $\tau$ such that $\kappa$ is $\tau$-closed
and such that there exists $\S\in\mathcal S$ for which $\S\setminus E^\kappa_{>\tau}$ is nonstationary,
for every injective sequence $\langle \eta_j\mid j<\tau\rangle$ of elements of $\mathcal B(K)$,
not only that the product tree $\bigotimes_{j<\tau}T^{\eta_j}$ is $\kappa$-Souslin,
but in fact, all of its $\tau$-derived trees are $\kappa$-Souslin. In particular, for every $\eta\in\mathcal B(K)$, $T^\eta$ is a free $\kappa$-Souslin tree.
\end{remark}

We now arrive at the following strong form of Theorem~\ref{thma}:
\begin{cor}\label{cor46} Suppose that $\p(\kappa,\kappa,{\sq},\kappa,\{\kappa\},2)$ holds, and let $\varsigma\in[2,\kappa)$.
For every cardinal $\mho$, if there exists a $\kappa$-tree with $\mho$-many cofinal branches,
then there exists a $\varsigma$-splitting $\kappa$-Aronszajn subtree of ${}^{<\kappa}\varsigma$ admitting $\mho$-many streamlined $\kappa$-Souslin subtrees
such that the product of any finitely (nonzero) many of them is again Souslin.
\end{cor}
\begin{proof} By a standard fact (see \cite[Lemma~2.5]{paper23}), if there exists a $\kappa$-tree with $\mho$-many cofinal branches,
then there exists one $K$ that is streamlined. Now, appeal to Theorem~\ref{thm44} with $K$ and $\mathcal S:=\{\kappa\}$,
bearing in mind Subsection~\ref{slimvscomplete}.
\end{proof}

\section{A free Souslin tree with a special power}\label{section:special}
Throughout this section, $\kappa=\lambda^+$ for a fixed infinite cardinal $\lambda$.
Our goal here is to present a construction of a $\kappa$-Souslin tree whose derived trees are special.
In the context of successor cardinals, the general definition of special trees given in Subsection~\ref{tree-properties} (following~\cite[p.~266]{MR908147})
coincides with a classical definition, as follows.

\begin{fact}[{\cite[Theorem~14]{MR0793235}}; see also {\cite[Theorem~16]{MR3274402}}]
A tree of height $\lambda^+$ is \emph{special} iff it may be covered by $\lambda$ many antichains.
\end{fact}

Consider the linear order $\mathbb{Q}_\lambda$ consisting of all nonempty \emph{finite} sequences of ordinals in $\lambda$,
with the ordering $q\ll p$ iff either $p\subsetneq q$ or $q(n)<p(n)$
for the least $n<\omega$ such that $q(n)\neq p(n)$.
This is a linear order of size $\lambda$ having no first or last element,
satisfying that in-between any two of its elements there are $\lambda$-many elements including an increasing sequence of order-type $\lambda$,
and that all of its subsets of size less than $\cf(\lambda)$ have an upper bound.\footnote{This definition may be found in \cite[p.~273]{MR776625} and \cite[Conventions~4.1(4)]{sh:221}.
We warn the reader that some authors denote by $\mathbb Q_\lambda$ a different linear order that shares some of the above features, but whose size is sensitive to cardinal-arithmetic assumptions.}

As $|\mathbb Q_\lambda| = \lambda$,
to show that a tree $(T,{<_T})$ of height $\lambda^+$ is special,
it suffices to exhibit a strictly increasing map $f:T\rightarrow\mathbb{Q}_\lambda$,
i.e., such that $s<_T t$ implies $f(s)\ll f(t)$.
In fact, the existence of a special $\lambda^+$-tree is equivalent to the existence of a $\lambda^+$-tree $T$
admitting a strictly increasing map $f:T\rightarrow\mathbb{Q}_\lambda$.\footnote{Jensen \cite{MR309729}
proved that the existence of a special $\lambda^+$-tree is equivalent to $\square^*_\lambda$,
whereas the streamlined $\lambda^+$-tree $T(\rho_0)$ obtained by conducting walks on ordinals along a $\square^*_\lambda$-sequence
can be shown to admit a strictly increasing map to $\mathbb Q_\lambda$.}

\begin{defn}
For a subset $T\s{}^{<\kappa}H_\kappa$ and a nonzero cardinal $\chi$:
\begin{itemize}
\item Denote $T^\chi:=\{\vec x:\chi\stackrel{1-1}{\longrightarrow} T\mid ({\dom}\circ{\vec x})\text{ is constant}\}$.
The ordering $<_{T^\chi}$ of $T^\chi$ is defined as follows:\footnote{If $T$ is a streamlined tree, then
elements of $T^\chi$ are sometimes referred to as \emph{injective level sequences},
and $(T^\chi,{<_{T^\chi}})$ should be understood as the union of all $\chi$-derived trees of $T$.
We shall soon construct a streamlined $\lambda^+$-tree $T$ such that $(T^\chi,{<_{T^\chi}})$ admits a strictly increasing map to $\mathbb Q_\lambda$,
from which it will follow that all $\chi$-derived trees of $T$ are special.}
$$\vec x<_{T^\chi}\vec y\iff \bigwedge_{i<\chi}\vec x(i)\stree\vec y(i).$$
\item For $\vec x\in T^\chi$, let $\iota(\vec x):=\min\{\alpha<\kappa\mid \langle\vec x(i)\restriction\alpha\mid i<\chi\rangle\text{ is injective}\}$.
\end{itemize}
\end{defn}

\begin{notation}
Dual to Notation~\ref{notationcomp},
for every $y\in{}^\chi({}^\alpha H_\kappa)$, we let $\fork y$ denote the unique function from $\alpha$ to ${}^\chi H_\kappa$
satisfying $(\fork y)_i=y(i)$ for every $i<\chi$.
This will allow us to move back and forth between the classical and streamlined representations of derived trees.
\end{notation}

\begin{notation}
For every $T\in H_\kappa$, denote $\beta(T):=0$ unless there is $\beta<\kappa$
such that $T\s{}^{\le\beta}H_\kappa$ and $T \nsubseteq {}^{<\beta}H_\kappa$,
in which case, we let $\beta(T):=\beta$ for this unique $\beta$.
\end{notation}

Fix some well-ordering $\lhd_\kappa$ of $H_\kappa$.
We collect here a couple of \emph{actions} from \cite[\S6.2]{paper23} which will be used in the upcoming construction.
The readers can verify to themselves that additional actions from the same reference can be incorporated into the upcoming proof.

\begin{defn}\label{actions}
\begin{enumerate}
\item The default extension function,
$\defaultaction:(H_\kappa)^2\rightarrow H_\kappa$, is defined as follows.
Let $\defaultaction(x,T):=x$, unless $$Q:=\{ z\in T_{\beta(T)}\mid x\s z\}$$ is nonempty,\footnote{Recall that $T_{\beta}$ stands for $\{ x\in T\mid \dom(x)=\beta\}$.}
in which case, we let $\defaultaction(x,T):=\min(Q, {\lhd_\kappa})$.

\item The function for sealing antichains, $\sealantichain:(H_\kappa)^3\rightarrow H_\kappa$,
is defined as follows.
Let $\sealantichain(x,T,\mho):=\defaultaction(x,T)$,
unless $$Q:=\{ z\in T_{\beta(T)}\mid \exists y\in\mho\,( x\cup y\s z)\}$$ is nonempty,
in which case, we let $\sealantichain(x,T,\mho):=\min(Q, {\lhd_\kappa})$.
\end{enumerate}
\end{defn}

The following is obvious.

\begin{lemma}\label{extendfact}
Suppose $T,\mho \in H_\kappa$, where $T$ is a normal streamlined subtree of ${}^{\le\beta(T)}H_\kappa$.
For every $x\in T$, $\sealantichain(x,T,\mho)$ is a node in $T_{\beta(T)}$ extending~$x$.
\qed
\end{lemma}

Hereafter, $\chi$ denotes some cardinal in $[2,\omega]$.
The next batch of definitions is motivated by the Abraham--Shelah--Solovay construction from \cite[\S4]{sh:221}.
Note, however, that the approach taken here is eventually quite different from the one in~\cite{sh:221},
since it works uniformly for $\lambda$ both singular and regular.

\begin{defn} We define three maps $\varphi_0,\varphi_1,\varphi_2:\mathbb Q_\lambda\rightarrow\mathbb Q_\lambda$ via:
\begin{itemize}
\item $\varphi_2(q):=\begin{cases}
\langle\xi+1\rangle,&\text{if }q=\langle\xi\rangle;\\
p,&\text{if }q=p^\smallfrown\langle\xi\rangle\text{ for }p\neq\emptyset.
\end{cases}$
\vspace{5pt}
\item $\varphi_1(q):=\begin{cases}
\langle\xi+1,0\rangle,&\text{if }q=\langle\xi\rangle;\\
p^\smallfrown\langle\xi+1\rangle,&\text{if }q=p^\smallfrown\langle\xi\rangle\text{ for }p\neq\emptyset.
\end{cases}$
\vspace{5pt}
\item $\varphi_0(q):=\begin{cases}
\langle\xi+1,0,0\rangle,&\text{if }q=\langle\xi\rangle;\\
p^\smallfrown\langle\xi+1,0\rangle,&\text{if }q=p^\smallfrown\langle\xi\rangle\text{ for }p\neq\emptyset.
\end{cases}$
\end{itemize}
\end{defn}
\begin{remark}\label{rmk417} For every $q$ in $\mathbb Q_\lambda$,
all of the following hold:
\begin{itemize}
\item $q\ll \varphi_0(q)\ll \varphi_1(q)\ll \varphi_2(q)$;
\item $\varphi_2(q)=\varphi_2(\varphi_1(q))$;
\item $\varphi_1(q)=\varphi_2(\varphi_0(q))$.
\end{itemize}
\end{remark}

\begin{defn}[Elevators] For a streamlined tree $T$,
a map $f:T^\chi\rightarrow\mathbb Q_\lambda$,
two maps $\varphi,\psi:\mathbb Q_\lambda\rightarrow\mathbb Q_\lambda$,
and ordinals $\beta<\alpha$,
we say that a function $e:T_\beta\rightarrow T_\alpha$ is a \emph{$(\varphi,\psi)$-elevator} (with respect to $f$)
iff the two hold:
\begin{enumerate}
\item $y\stree e(y)$ for every $y\in T_\beta$, and
\item for every $\langle y_i\mid i<\chi\rangle\in (T_\beta)^\chi$,
$$(\varphi\circ f)(\langle e(y_i)\mid i<\chi\rangle)=(\psi\circ f)(\langle y_i\mid i<\chi\rangle).$$
\end{enumerate}

A map $e$ satisfying just Clause~(1) will be simply referred to as an \emph{elevator}.
\end{defn}

\begin{defn}[Coordination]\label{def_order} Let $T$ be a streamlined tree, and let $f$ be a function from $T^\chi$ to $\mathbb Q_\lambda$.
For a pair of ordinals $\beta<\alpha$, we say that $T_\beta$ and $T_\alpha$ are \emph{coordinated} (with respect to $f$) iff for all nonzero $n<\chi$
and all $\langle z_j\mid j<n\rangle\in (T_\alpha)^n$ such that $\iota(\langle z_j\mid j<n\rangle)\leq\beta$, the following three hold:
\begin{itemize}
\item[(i)] there exists a $(\varphi_2,\varphi_1)$-elevator $e_1:T_\beta\rightarrow T_\alpha$ such that $e_1(z_j\restriction\beta)=z_j$ for all $j<n$;
\item[(ii)] there exists a $(\varphi_2,\varphi_2)$-elevator $e_2:T_\beta\rightarrow T_\alpha$ such that $e_2(z_j\restriction\beta)=z_j$ for all $j<n$;
\item[(iii)] if $\alpha=\beta+1$, then for every $m<3$, there exists an $(\id,\varphi_m)$-elevator $e:T_\beta\rightarrow T_\alpha$ such that $e(z_j\restriction\beta)=z_j$ for all $j<n$.
\end{itemize}
\end{defn}

We are now ready to prove the main result of this section, which also yields Theorem~\ref{thmb}.

\begin{thm}\label{thm710}
Suppose that $\p_\lambda(\lambda^+,2,{\sq},\allowbreak\lambda^+)$ holds.
Let $\chi\in[2,\omega]$ with $\chi<\cf(\lambda)$.
Then there is a normal $\chi$-free, slim, prolific, club-regressive,
streamlined $\lambda^+$-Souslin tree $T$ such that $(T^\chi,{<_{T^\chi}})$ admits a strictly increasing map to $\mathbb Q_\lambda$.
\end{thm}
\begin{proof}
Recall that $\kappa=\lambda^+$. Fix $\vec{C}=\langle C_\alpha\mid\alpha<\kappa\rangle$ witnessing $\p^-_\lambda(\kappa,2,{\sq},\kappa)$.
For every $\alpha\in\acc(\kappa)$, let
$$D_\alpha:=\{0\}\cup\{\eta+1\mid \eta\in\nacc(C_\alpha)\}\cup\acc(C_\alpha),$$
so that $D_\alpha$ is a club in $\alpha$ for which $D_\alpha\cap\acc(\kappa)=\acc(D_\alpha)=\acc(C_\alpha)$, so that $\nacc(D_\alpha)\s\nacc(\alpha)$.
Evidently, $\vec D:=\langle D_\alpha\mid\alpha\in\acc(\kappa)\rangle$ is yet another $\lambda$-bounded coherent $C$-sequence.
Recall that by Definition~\ref{Ccharacteristics}, $A(\vec{D})=\{\alpha\in\acc(\kappa)\mid \forall\beta\in(\alpha,\kappa)\,(\alpha\notin\acc(D_\beta))\}$,
so since $\vec D$ is $\lambda$-bounded and coherent, we have here $\{\alpha\in\acc(\kappa)\mid \otp(D_\alpha)=\lambda\}\s A(\vec{D})$.

As $\diamondsuit(\kappa)$ holds, fix sequences $\langle \Omega_\beta\mid\beta<\kappa\rangle$ and $\langle R_i\mid i<\kappa\rangle$
together witnessing $\diamondsuit(H_\kappa)$, as in Fact~\ref{diamondhkappa}.
Let $\pi:\kappa\rightarrow\kappa$ be such that $\beta\in R_{\pi(\beta)}$ for all $\beta<\kappa$.
Let $\lhd_\kappa$ be some well-ordering of $H_\kappa$ of order-type $\kappa$,
and let $\phi:\kappa\leftrightarrow H_\kappa$ witness the isomorphism $(\kappa,\in)\cong(H_\kappa,\lhd_\kappa)$.
Put $\psi:=\phi\circ\pi$.

By recursion on $\alpha<\kappa$, we shall construct $\langle (T_\alpha,f_\alpha)\mid \alpha<\kappa\rangle$
such that $T_\alpha$ will end up being the $\alpha^{\text{th}}$ level of the ultimate normal, slim, streamlined tree $T$,
and $f_\alpha:(T_\alpha)^\chi\rightarrow\mathbb Q_\lambda$ will form the $\alpha^{\text{th}}$ level of the ultimate strictly increasing map $f:T^\chi\to\mathbb{Q}_\lambda$.
We shall also make sure that for all $\beta<\alpha<\kappa$, $T_\beta$ and $T_\alpha$ be coordinated.

By convention, for every $\alpha<\kappa$ such that $\langle (T_\beta,f_\beta)\mid \beta<\alpha\rangle$ has already been defined,
for every $C\s\alpha$, we shall let $T\restriction C:=\bigcup_{\beta\in C}T_\beta$.

The recursion starts by setting $T_0:=\{\emptyset\}$ and letting $f_0$ be the empty function.
Next, given $\alpha<\kappa$ such that $(T_\alpha,f_\alpha)$ has already been successfully defined,
set $T_{\alpha+1}:=\{t^\smallfrown\langle \tau\rangle\mid t\in T_\alpha, \tau<\max\{\omega,\alpha\}\}$.
Before we can define $f_{\alpha+1}:(T_{\alpha+1})^\chi\rightarrow\mathbb Q_\lambda$, we shall need the following claim.
\begin{claim} Let $S_{\alpha+1}:=\bigcup_{0<n<\chi}\{s\in (T_{\alpha+1})^n\mid \iota(s)\le\alpha\}$.
Then there exists a matrix $\langle A_{s,m}\mid s\in S_{\alpha+1},~m<3\rangle$ such that,
for every $(s,m)\in S_{\alpha+1}\times 3$, the following three hold:
\begin{enumerate}
\item $\im(s)\s A_{s,m}\s T_{\alpha+1}$;
\item for every $y\in T_\alpha$, there exists a unique $z\in A_{s,m}$ extending $y$;
\item for every $(s',m')\in S_{\alpha+1}\times 3\setminus\{(s,m)\}$, we have $|A_{s,m}\cap A_{s',m'}|<\chi$.
\end{enumerate}
\end{claim}
\begin{proof} Denote $\alpha':=\max\{\omega,\alpha\}$.
For each $s \in S_{\alpha+1}$, denote $Z_s :=\{z(\alpha)\mid z\in\im(s)\}$,
which is an element of $[\alpha']^{<\chi}$.
As $\chi\le\omega$ and as we have been constructing a slim tree, the definition of $T_{\alpha+1}$ implies that
we may fix an enumeration $\langle s_\gamma\mid \gamma<|\alpha'|\rangle$ of $S_{\alpha+1}$.
Recursively construct a sequence $\langle (\tau_{\gamma,0},\tau_{\gamma,1},\tau_{\gamma,2})\mid \gamma<|\alpha'|\rangle$ consisting of elements of $\alpha'\times\alpha'\times\alpha'$
by letting
$\{\tau_{\gamma,0},\tau_{\gamma,1},\tau_{\gamma,2}\}$ consist of the first three elements of $\alpha'$ that do not belong to the following set:
$$\{\tau_{\beta,m}\mid \beta<\gamma,~m<3\}\cup \bigcup\nolimits_{\beta\le\gamma}Z_{s_\beta}.$$
Finally, for all $\gamma<|\alpha'|$ and $m<3$, let
$$A_{s_\gamma,m}:= \im(s_\gamma)\cup\{y^\smallfrown\langle\tau_{\gamma,m}\rangle\mid y\in T_\alpha\setminus\{ (z\restriction\alpha)\mid z\in\im(s_\gamma)\}\}.$$

It is clear that the above definition takes care of Clauses (1) and (2).
To verify Clause~(3), let $(s,m),(s',m')$ in $S_{\alpha+1}\times 3$ be given.
Find $\beta,\gamma<|\alpha'|$ such that $s=s_\beta$ and $s'=s_\gamma$.
Without loss of generality, $\beta\le\gamma$.
Now, if $|A_{s,m}\cap A_{s',m'}|\ge\chi$, then we may pick $z\in A_{s,m}\cap A_{s',m'}\setminus\im(s')$,
so that $z(\alpha)\in Z_{s_\beta}\cup\{\tau_{\beta,m}\}$ and $z(\alpha)=\tau_{\gamma,m'}$.
As $\tau_{\gamma,m'}\notin Z_{s_\beta}$, it follows that $\tau_{\beta,m}=\tau_{\gamma,m'}$ and hence $(s,m)=(s',m')$.
\end{proof}

Fix a matrix $\langle A_{s,m}\mid s\in S_{\alpha+1},~ m<3\rangle$ as in the preceding claim.
To define $f_{\alpha+1}:(T_{\alpha+1})^\chi\rightarrow \mathbb{Q}_\lambda$,
let $\vec w=\langle w_i\mid i<\chi\rangle$ in $(T_{\alpha+1})^\chi$ be given.
There are three cases to consider:

$\br$ If $\iota(\vec{w})=\alpha+1$, then let $f_{\alpha+1}(\vec{w}):=\langle 0\rangle$.

$\br$ If $\iota(\vec{w})\leq\alpha$ and $\{w_i\mid i<\chi\}\s A_{s,m}$ for some $(s,m)\in S_{\alpha+1}\times 3$, then
by Clause~(3) of the above claim, the pair $(s,m)$ is unique, so we let
$$f_{\alpha+1}(\vec{w}):=(\varphi_m\circ f_\alpha)(\langle w_i\restriction\alpha\mid i<\chi\rangle).$$

$\br$ Otherwise, since $\mathbb Q_\lambda$ has no maximal elements,
let $f_{\alpha+1}(\vec{w})$ be some element of $\mathbb Q_\lambda$ which is bigger than $f_\alpha(\langle w_i\restriction\alpha\mid i<\chi\rangle)$.

Altogether, for every $\vec{w}\in(T_{\alpha+1})^\chi$
such that $\iota(\vec{w})\leq\alpha$,
it is the case that
$f_\alpha(\langle w_i\restriction\alpha\mid i<\chi\rangle)\ll f_{\alpha+1}(\vec{w})$.

\begin{claim} \begin{enumerate}
\item $T_\alpha$ and $T_{\alpha+1}$ are coordinated;
\item For every $\beta<\alpha$, $T_\beta$ and $T_{\alpha+1}$ are coordinated.
\end{enumerate}
\end{claim}
\begin{proof} (1) Consider any given nonzero $n<\chi$ and $\langle z_j\mid j<n\rangle\in (T_{\alpha+1})^n$ with $\iota(\langle z_j\mid j<n\rangle)\leq\alpha$.
In particular, $s:=\langle z_j\mid j<n\rangle$ is in $S_{\alpha+1}$.
We go over the clauses of Definition~\ref{def_order} in reverse order:
\begin{itemize}
\item[(iii)] Given $m<3$,
define an elevator $e:T_\alpha\rightarrow T_{\alpha+1}$ by letting, for every $y\in T_\alpha$, $e(y)$ be the unique $z\in A_{s,m}$ extending $y$.
As $\im(s)=\{z_j\mid j<n\}\s A_{s,m}$, for every $j<n$, $e(z_j\restriction\alpha)$ must be $z_j$.
We claim that $e$ is an $(\id,\varphi_m)$-elevator.
Indeed, for any $\langle y_i\mid i<\chi\rangle\in (T_\alpha)^\chi$, $\{e(y_i)\mid i<\chi\}\s A_{s,m}$,
so by the definition of $f_{\alpha+1}$:
$$f_{\alpha+1}(\langle e(y_i)\mid i<\chi\rangle)=(\varphi_m\circ f_\alpha)(\langle y_i\mid i<\chi\rangle)).$$
\item[(ii)] By Clause~(iii), we may fix an $(\id,\varphi_1)$-elevator $e_2:T_\alpha\rightarrow T_{\alpha+1}$ such that $e_2(z_j\restriction\alpha)=z_j$
for all $j<n$. By Remark~\ref{rmk417}, it is also a $(\varphi_2,\varphi_2)$-elevator.
\item[(i)] By Clause~(iii), we may fix an $(\id,\varphi_0)$-elevator $e_1:T_\alpha\rightarrow T_{\alpha+1}$ such that $e_1(z_j\restriction\alpha)=z_j$
for all $j<n$. By Remark~\ref{rmk417}, it is also a $(\varphi_2,\varphi_1)$-elevator.
\end{itemize}

(2) Fix $\beta<\alpha$,
nonzero $n<\chi$, and $\langle z_j\mid j<n\rangle\in (T_{\alpha+1})^n$ with $\iota(\langle z_j\mid j<n\rangle)\leq\beta$.
We go over the clauses of Definition~\ref{def_order}:

\begin{itemize} \item[(i)] By Clause~(1) of this claim, fix a $(\varphi_2,\varphi_2)$-elevator
$e:T_\alpha\rightarrow T_{\alpha+1}$ such that $e(z_j\restriction\alpha)=z_j$ for all $j<n$.
In addition, as $T_\beta$ and $T_\alpha$ are coordinated,
fix a $(\varphi_2,\varphi_1)$-elevator $e_1:T_\beta\rightarrow T_\alpha$
such that $e_1(z_j\restriction\beta)=z_j\restriction\alpha$ for all $j<n$.
Set $E_1:=e\circ e_1$.
Then $E_1(z_j\restriction\beta)=e(e_1(z_j\restriction\beta))=e(z_j\restriction \alpha)=z_j$ for all $j<n$.
In addition, for all $y\in T_\beta$, $y\stree e_1(y)\stree e(e_1(y))=E_1(y)$.
Finally, for every $\langle y_i\mid i<\chi\rangle\in(T_\beta)^\chi$,
$$\begin{aligned}(\varphi_2\circ f_{\alpha+1})(\langle E_1(y_i)\mid i<\chi\rangle)&=(\varphi_2\circ f_{\alpha+1})(\langle e(e_1(y_i))\mid i<\chi\rangle)\\
&=(\varphi_2\circ f_\alpha)(\langle e_1(y_i)\mid i<\chi\rangle)\\
&=(\varphi_1\circ f_\beta)(\langle y_i\mid i<\chi\rangle).\end{aligned}$$

Thus, $E_1$ is a $(\varphi_2,\varphi_1)$-elevator, as sought.

\item[(ii)] Replace $1$ by $2$ throughout the above proof.
\item[(iii)] This clause is satisfied vacuously in this case, as $\alpha+1\neq \beta+1$.\qedhere
\end{itemize}
\end{proof}

Next, suppose that we have reached an $\alpha\in\acc(\kappa)$ such that $\langle (T_\beta,f_\beta)\mid \beta<\alpha\rangle$ has already been successfully defined.
For each $x\in T\restriction D_\alpha$, we shall define a branch $\mathbf{b}_x^\alpha$ through $\bigcup_{\beta<\alpha}T_\beta$, and then let
\begin{equation}\tag{$\star$}\label{promise1}T_\alpha:=\{\mathbf{b}_x^\alpha\mid x\in T\restriction D_\alpha\}.\end{equation}

The branch $ \mathbf{b}_x^ \alpha$ will be obtained as the limit $\bigcup\im(b_x^\alpha)$ of a sequence $\langle b_x^\alpha(\beta)\mid\beta\in D_\alpha\setminus\dom(x)\rangle$ of nodes such that:

\begin{itemize}
\item for every $\beta\in D_\alpha\setminus\dom(x)$, $b_x^\alpha(\beta)\in T_\beta$;
\item for every pair $\beta<\beta'$ of ordinals in $D_\alpha\setminus\dom(x)$, $x\s b_x^\alpha(\beta)\stree b_x^\alpha(\beta')$;
\item for every $\beta\in\acc(D_\alpha\setminus\dom(x))$, $b_x^\alpha(\beta)=\bigcup(\im(b_x^\alpha\restriction\beta))$.
\end{itemize}

The construction is by recursion on $\beta\in D_\alpha$,
where at stage $\beta$ we shall be determining $b_x^\alpha(\beta)$ for all $x\in T\restriction(D_\alpha\cap (\beta+1))$ simultaneously.

$\br$ For $\beta:=\min(D_\alpha)$,
it is the case that $\beta=0$ and so $T_\beta=\{x\}$ for $x:=\emptyset$.
Therefore, we set $b_x^\alpha(\beta):=x$.

$\br$ Suppose that we are given a nonzero $\beta\in\nacc(D_\alpha)$ such that
$b_x^\alpha\restriction\beta$ has already been defined for all $x\in T\restriction(D_\alpha\cap\beta)$.
We need to define $b_x^\alpha(\beta)$ for all $x\in T\restriction(D_\alpha\cap(\beta+1))$.
For every $x\in T_\beta$, we just let $b_x^\alpha(\beta):=x$.
Our next task is defining $b_x^\alpha(\beta)$ for all $x\in T\restriction(D_\alpha\cap\beta)$.
To this end, we introduce the following pieces of notation.

\begin{indef}\label{gfamily}
Denote $\beta^-:=\sup(D_\alpha\cap\beta)$
and note that by the definition of $D_\alpha$,
it is the case that $\beta=\eta+1$ for a unique $\eta\in C_\alpha\setminus\beta^-$.
Now, let $\mathcal E_\beta^\alpha$ denote the collection of all
$(\id,\varphi_1)$-elevators $e:T_{\beta^-}\rightarrow T_{\beta}$ satisfying
that if there exists a nonzero $n<\chi$ such that $\psi(\eta)\in(T\restriction(D_\alpha\cap\beta^-))^n$,
then letting $\bar T:=(T\restriction(\beta+1))(\psi(\eta))$ (using Definition~\ref{derived}),
for every $j<n$,
$$e(b_{\psi(\eta)(j)}^\alpha(\beta^-))=(\sealantichain(\fork{\langle b_{\psi(\eta)(j)}^\alpha(\beta^-)\mid j<n\rangle},\bar T,\Omega_\eta))_j.$$
\end{indef}

Note that the definitions of $\beta^-,\eta$ and $\mathcal E_\beta^\alpha$ are all determined by no more than the following objects:
\begin{itemize}
\item $\langle (T_\gamma,f_\gamma)\mid \gamma\le\beta\rangle$,
\item $\psi(\eta)$,
\item $D_\alpha\cap(\beta+1)$, and possibly also on
\item $\langle b_{\psi(\eta)(j)}^\alpha(\beta^-)\mid j<n\rangle$ and $\Omega_\eta$,
where we note in particular that $\psi(\eta)(j) \in T\restriction(D_\alpha\cap\beta^-)$ for all $j<n$.
\end{itemize}

\begin{claim}\label{E-nonempty} $\mathcal E_\beta^\alpha$ is nonempty.
\end{claim}
\begin{proof} If there exists some nonzero $n<\chi$ such that $\psi(\eta)\in(T\restriction(D_\alpha\cap\beta^-))^n$,
then let
$$z:=\sealantichain(\fork{\langle b_{\psi(\eta)(j)}^\alpha(\beta^-)\mid j<n\rangle},\bar T,\Omega_\eta).$$
Otherwise, just set $(n,z):=(0,\emptyset)$. Now, there are two options:

\begin{description}
\item[The case $\eta=\beta^-$]
As $T_\eta$ and $T_{\beta}$ are coordinated and $\beta=\eta+1$,
we fix an $(\id,\varphi_1)$-elevator $e:T_\eta\rightarrow T_{\beta}$ such that $e((z)_j\restriction\eta)=(z)_j$ for all $j<n$.
Since $\eta=\beta^-$, $e$ demonstrates that $\mathcal E_\beta^\alpha$ is nonempty.

\item[The case $\eta>\beta^-$] As $T_{\beta^-}$ and $T_\eta$ are coordinated, we first fix a $(\varphi_2,\varphi_1)$-elevator $e_1:T_{\beta^-}\rightarrow T_\eta$ such that
$e_1(b_{\psi(\eta)(j)}^\alpha(\beta^-))=(z)_j\restriction\eta$ for all $j<n$.
Then, as $T_\eta$ and $T_{\beta}$ are coordinated and $\beta=\eta+1$,
we also fix an $(\id,\varphi_2)$-elevator $e_2:T_\eta\rightarrow T_{\beta}$ such that $e_2((z)_j\restriction\eta)=(z)_j$ for all $j<n$.
Set $e:=e_2\circ e_1$.
Clearly, $e(b_{\psi(\eta)(j)}^\alpha(\beta^-))=(z)_j$ for all $j<n$.
In addition, for every $\langle y_i\mid i<\chi\rangle\in(T_{\beta^-})^\chi$,
$$\begin{aligned}f_{\beta}(\langle e(y_i)\mid i<\chi\rangle)&=f_{\beta}(\langle e_2(e_1(y_i))\mid i<\chi\rangle)\\&=(\varphi_2\circ f_\eta)(\langle e_1(y_i)\mid i<\chi\rangle)\\
&=(\varphi_1\circ f_{\beta^-})(\langle y_i\mid i<\chi\rangle).\end{aligned}$$
So, $e$ demonstrates that $\mathcal E_\beta^\alpha$ is nonempty.\qedhere
\end{description}
\end{proof}
Let $e:=\min(\mathcal E_\beta^\alpha,\lhd_\kappa)$,
and then define $b_x^\alpha(\beta):=e(b_x^\alpha(\beta^-))$ for every $x\in T\restriction(D_\alpha\cap\beta)$.

$\br$ Suppose that we are given a $\beta\in\acc(D_\alpha)$ such that
$b_x^\alpha\restriction\beta$ has already been defined for all $x\in T\restriction(D_\alpha\cap\beta)$.
For every $x\in T\restriction (D_\alpha\cap(\beta+1))$,
we let
$$b_x^\alpha(\beta):=\begin{cases}
x,&\text{if }x\in T_\beta;\\
\bigcup\im(b_x^\alpha\restriction\beta),&\text{if }x\in T\restriction (D_\alpha\cap \beta).
\end{cases}$$

In order to argue that $b_x^\alpha(\beta)$ is indeed an element of $T_\beta$,
it suffices to prove the following claim.

\begin{claim}\label{claim4204} For every $x\in T\restriction (D_\alpha\cap \beta)$, $b_x^\alpha(\beta)=\mathbf{b}_x^\beta$.
\end{claim}
\begin{proof} It suffices to show that for every
$x\in T\restriction (D_\alpha\cap \beta)$, for every $\gamma\in(D_\alpha\cap\beta)\setminus\dom(x)$,
it is the case that $b_x^\alpha(\gamma)=\mathbf b_x^\beta\restriction\gamma$.
Recalling that $\beta\in\acc(D_\alpha)$ and that $\vec D$ is coherent,
we infer that $D_{\alpha}\cap\beta=D_{\beta}$.
Thus, it suffices to show that for every
$x\in T\restriction D_\beta$, for every $\gamma\in D_\beta\setminus\dom(x)$,
it is the case that $b_x^\alpha(\gamma)=b_x^\beta(\gamma)$.
The proof is by induction on $\gamma \in D_\beta$ simultaneously for all $x\in T\restriction D_\beta$, as follows:

$\br$ The base case $\gamma=0$ is obvious, as $b_\emptyset^\alpha(0)=\emptyset=b_\emptyset^\beta(0)$.

$\br$ Suppose that $\gamma^-<\gamma$ are successive elements of $D_\beta$ and that, for every $x\in T\restriction (D_\beta\cap\gamma)$,
$b_x^\alpha(\gamma^-)=b_x^\beta(\gamma^-)$.
By our construction, it is the case that $b_x^\alpha(\gamma):=e_\alpha(b_x^\alpha(\gamma^-))$ for $e_\alpha:=\min(\mathcal E^\alpha_\gamma,\lhd_\kappa)$
and that $b_x^\beta(\gamma):=e_\beta(b_x^\beta(\gamma^-))$ for $e_\beta:=\min(\mathcal E^\beta_\gamma,\lhd_\kappa)$.
Reading the comment right after Notation~\ref{gfamily}, it is clear that in this case $\mathcal E^\alpha_\gamma=\mathcal E^\beta_\gamma$,
and so $e_\alpha=e_\beta$, and thus $b_x^\alpha(\gamma)=b_x^\beta(\gamma)$ for all $x\in T\restriction (D_\beta\cap\gamma)$.
For $x\in T_\gamma$, clearly $b_x^\alpha(\gamma)=x=b_x^\beta(\gamma)$.

$\br$ Suppose $\gamma\in \acc(D_\beta)$.
For every $x\in T\restriction (D_\beta\cap\gamma)$ such that the two sequences $b_x^\alpha$ and $b_x^\beta$ agree up to $\gamma$, it is the case that they have the same unique limit,
so that $b_x^\alpha(\gamma)=b_x^\beta(\gamma)$.
For $x\in T_\gamma$, again, $b_x^\alpha(\gamma)=x=b_x^\beta(\gamma)$.
\end{proof}

We are done defining $\langle \mathbf{b}_x^\alpha\mid x\in T\restriction D_\alpha\rangle$,
and so we define $T_\alpha$ as per $\eqref{promise1}$.
\begin{claim}\label{claim885} For all $x\in T\restriction D_\alpha$ and $\gamma\in D_\alpha\setminus\dom(x)$, $\mathbf b_x^\alpha=\mathbf b_{b_x^\alpha(\gamma)}^\alpha$.
\end{claim}
\begin{proof} By the canonical nature of the above construction.
\end{proof}
Now, to define $f_\alpha:(T_\alpha)^\chi\rightarrow\mathbb Q_\lambda$,
let $\vec{w}=\langle w_i\mid i<\chi\rangle$ in $(T_\alpha)^\chi$ be given.
For each $i<\chi$, find $x_i\in T\restriction D_\alpha$ of minimal height such that $\mathbf{b}_{x_i}^\alpha=w_i$.
Set $\gamma:=\sup\{\dom(x_i)\mid i<\chi\}$, and notice that it follows from Claim~\ref{claim885} that $\iota(\vec{w})\leq\gamma$. There are two main cases to consider:
\begin{itemize}
\item[$\br$] If $\gamma=\alpha$, then $\cf(\alpha)=\chi=\omega$.
Now, there are two subcases here:
\begin{itemize}
\item[$\br\br$] If $\iota(\vec{w})=\alpha$, then it is harmless to let $f_\alpha(\vec{w}):=\langle0\rangle$.
\item[$\br\br$] If $\iota(\vec{w})<\alpha$,
then since $\cf(\alpha)\le\chi<\cf(\lambda)$,
the set $F:=\{f_\delta(\langle {w_i\restriction\delta}\mid i<\chi\rangle) \mid \iota(\vec{w})\leq\delta<\alpha\}$ is bounded in $\mathbb Q_\lambda$; fix some upper bound $q$ of $F$, and let $f_\alpha(\vec{w}):=q$.
\end{itemize}
\item[$\br$] If $\gamma<\alpha$,
then denote $(\varphi_1\circ f_\gamma)(\langle w_i\restriction\gamma\mid i<\chi\rangle)$ by $p^\smallfrown\langle\xi\rangle$.
Note that the definition of $\varphi_1$ ensures that $p$ is nonempty.
Now, again there are two subcases:
\begin{itemize}
\item[$\br\br$] If $\alpha\in A(\vec{D})$, then let $f_\alpha(\vec{w}):=p$.
\item[$\br\br$] Otherwise, let $f_\alpha(\vec{w}):=p^\smallfrown\langle\xi+\sigma\rangle$,
where $\sigma := \otp(D_\alpha\setminus({\gamma+1}))$.\footnote{Note that $\sigma\le\otp(D_\alpha)<\lambda$ since $\alpha\notin A(\vec D)$.}
\end{itemize}
\end{itemize}
This completes the definition of the function $f_\alpha$.
Next, we must verify that $\bigcup_{\beta\le\alpha}f_\beta$ is strictly increasing,
and that $T_\beta$ and $T_\alpha$ are coordinated for every $\beta<\alpha$.
It will be easier to show once we have established the following claim.

\begin{claim}\label{c4206} Let $\langle w_i\mid i<\chi\rangle\in (T_\alpha)^\chi$.
For each $i<\chi$, find $x_i\in T\restriction D_\alpha$ of minimal height such that $\mathbf{b}_{x_i}^\alpha=w_i$.
Suppose that $\gamma:=\sup\{\dom(x_i)\mid i<\chi\}$ is smaller than $\alpha$,
and let $p^\smallfrown\langle\xi\rangle$ denote $(\varphi_1\circ f_\gamma)(\langle w_i\restriction\gamma\mid i<\chi\rangle)$. Then:
\begin{enumerate}
\item For every $\beta\in D_\alpha\setminus\gamma$,
$$(\varphi_2\circ f_\beta)(\langle w_i\restriction \beta\mid i<\chi\rangle)=p;$$
\item For every $\beta\in D_\alpha\setminus(\gamma+1)$,
$$f_\beta(\langle w_i\restriction \beta\mid i<\chi\rangle)=p^\smallfrown\langle\xi+\sigma\rangle,$$
where $\sigma := \otp(D_\alpha\cap\beta\setminus(\gamma+1))$.
\end{enumerate}

\end{claim}
\begin{proof} (1) The conclusion for $\beta>\gamma$ follows from Clause~(2) below.
As for $\beta=\gamma$, note that by Remark~\ref{rmk417},
$$(\varphi_2\circ f_\beta)(\langle w_i\restriction \beta\mid i<\chi\rangle)=(\varphi_2\circ\varphi_1\circ f_\beta)(\langle w_i\restriction \beta\mid i<\chi\rangle)=\varphi_2(p^\smallfrown\langle\xi\rangle)=p.$$

(2) For each $i<\chi$, write $y_i:=w_i\restriction\gamma$, and
note that $\mathbf{b}_{y_i}^\alpha=w_i$ by Claim~\ref{claim885}.
We now prove the claim by induction on $\beta\in D_\alpha\setminus(\gamma+1)$:
\begin{description}
\item[Base] Suppose $\beta=\min(D_\alpha\setminus(\gamma+1))$,
so that $\otp(D_\alpha\cap\beta\setminus(\gamma+1))=0$.
Recalling the construction, there exists some $(\id,\varphi_1)$-elevator $e:T_{\gamma}\rightarrow T_{\beta}$ (coming from $\mathcal E^\alpha_\beta$)
such that $b_{x}^\alpha(\beta)=e(b_{x}^\alpha(\gamma))$ for every $x\in T_{\gamma}$.
Thus, $$\begin{aligned}f_{\beta}(\langle b_{y_i}^\alpha(\beta)\mid i<\chi\rangle)&=(\varphi_1\circ f_{\gamma})(\langle b_{y_i}^\alpha({\gamma})\mid i<\chi\rangle)\\
&=(\varphi_1\circ f_{\gamma})(\langle y_i\mid i<\chi\rangle)\\
&=(\varphi_1\circ f_{\gamma})(\langle w_i\restriction\gamma\mid i<\chi\rangle)\\
&=p^\smallfrown\langle\xi\rangle\\
&=p^\smallfrown\langle\xi+\sigma\rangle,\end{aligned}$$
since $\sigma=0$.

\item[Successor step] Suppose $\beta\in\nacc(D_\alpha)$ is such that $\beta^-:=\sup(D_\alpha\cap\beta)$
is bigger than $\gamma$ and
satisfies $$f_{\beta^-}(\langle b_{y_i}^\alpha(\beta^-)\mid i<\chi\rangle)=p^\smallfrown\langle\xi+\sigma\rangle,$$
where $\sigma := \otp(D_\alpha\cap\beta^-\setminus(\gamma+1))$.
Recalling the construction, there exists some $(\id,\varphi_1)$-elevator $e:T_{\beta^-}\rightarrow T_{\beta}$
such that $b_{x}^\alpha(\beta)=e(b_{x}^\alpha(\beta^-))$ for every $x\in T\restriction(D_\alpha\cap\beta)$.
Thus, $$\begin{aligned}f_{\beta}(\langle b_{y_i}^\alpha(\beta)\mid i<\chi\rangle)&=(\varphi_1\circ f_{\beta^-})(\langle b_{y_i}^\alpha({\beta^-})\mid i<\chi\rangle)\\
&=\varphi_1(p^\smallfrown\langle\xi+\sigma\rangle)\\
&=p^\smallfrown\langle\xi+\sigma+1\rangle,\end{aligned}$$
as sought.

\item[Limit step] Suppose $\beta\in \acc(D_\alpha\setminus\gamma)$.
By Claim~\ref{claim4204}, $\langle b_{y_i}^\alpha(\beta)\mid i<\chi\rangle=\langle \mathbf{b}_{y_i}^\beta\mid i<\chi\rangle$.
The definition of $f_\beta(\langle w_i\restriction \beta\mid i<\chi\rangle)$ goes through the following considerations.
For each $i<\chi$, find $\bar x_i\in T\restriction D_\beta$ of minimal height such that $\mathbf{b}_{\bar x_i}^\beta=w_i\restriction\beta$,
and then set $\bar\gamma:=\sup\{\dom(\bar x_i)\mid i<\chi\}$.
As $\mathbf{b}_{y_i}^\alpha\restriction\beta=w_i\restriction\beta$ for every $i<\chi$,
we infer that $\bar\gamma\le\gamma$.
To see that also $\bar\gamma\ge\gamma$, note that for every $i<\chi$,
by Claim~\ref{claim4204}, $\mathbf{b}_{\bar x_i}^\alpha\restriction\beta=\mathbf{b}_{\bar x_i}^\beta$, and hence
$\mathbf{b}_{\bar x_i}^\alpha\restriction\gamma=\mathbf{b}_{\bar x_i}^\beta\restriction\gamma=w_i\restriction\gamma=y_i$,
and then Claim~\ref{claim885} implies that $\mathbf{b}_{\bar x_i}^\alpha=\mathbf{b}_{y_i}^\alpha=w_i$.

Now, since $\bar\gamma=\gamma<\beta$ and $\beta\notin A(\vec{D})$, it is the case that $f_\beta(\langle w_i\restriction\beta\mid i<\chi\rangle)=p^\smallfrown\langle\xi+\sigma\rangle$
where $\sigma := \otp(D_\beta\setminus(\gamma+1))$.
But $D_\beta=D_\alpha\cap\beta$ and hence
$\sigma = \otp(D_\alpha\cap\beta\setminus(\gamma+1))$.
\qedhere
\end{description}
\end{proof}

\begin{claim} Let $\vec{w}=\langle w_i\mid i<\chi\rangle$ in $(T_\alpha)^\chi$ be given.
Then $f_\beta(\langle w_i\restriction\beta\mid i<\chi\rangle)\ll f_\alpha(\vec{w})$ for every $\beta\in[\iota(\vec{w}),\alpha)$.
\end{claim}
\begin{proof}
By the induction hypothesis on $\langle (T_\beta,f_\beta)\mid\beta<\alpha\rangle$, to show that $f_\beta(\langle w_i\restriction\beta\mid i<\chi\rangle)\ll f_\alpha(\vec{w})$ for a tail of $\beta<\alpha$,
it suffices to prove that this is the case for cofinally many $\beta<\alpha$.
For each $i<\chi$, find $x_i\in T\restriction D_\alpha$ of minimal height such that $\mathbf{b}_{x_i}^\alpha=w_i$.
Set $\gamma:=\sup\{\dom(x_i)\mid i<\chi\}$, so that $\gamma\in D_\alpha\cup\{\alpha\}$.
By the definition of $f_\alpha$, we may avoid trivialities and assume that $\gamma<\alpha$.
In this case,
we let $p^\smallfrown\langle\xi\rangle$ denote $(\varphi_1\circ f_\gamma)(\langle w_i\restriction\gamma\mid i<\chi\rangle)$,
and observe that by Claim~\ref{c4206}(2), it suffices to prove that for every
$\sigma < \otp(D_\alpha\setminus(\gamma+1))$,
$$p^\smallfrown\langle\xi+\sigma\rangle\ll f_\alpha(\vec{w}).$$
However, the definition of $f_\alpha$ makes it clear that this is indeed the case.
\end{proof}
\begin{claim} Let $\epsilon<\alpha$. Then $T_\epsilon$ and $T_\alpha$ are coordinated.
\end{claim}
\begin{proof}
Consider any given nonzero $n<\chi$ and $\langle z_j\mid j<n\rangle\in (T_\alpha)^n$
with $\iota(\langle z_j\mid j<n\rangle)\leq\epsilon$.
Before going over the clauses of Definition~\ref{def_order},
let us first establish the following crucial subclaim.
\begin{subclaim} There are $\beta\in(\epsilon,\alpha)$ and a $(\varphi_2,\varphi_2)$-elevator $e: T_\beta\rightarrow T_\alpha$ such that $\langle e(z_j\restriction\beta)\mid j<n\rangle=\langle z_j\mid j<n\rangle$.
\end{subclaim}
\begin{proof}
For each $j<n$, fix $\bar x_j\in T\restriction D_\alpha$ of minimal height such that $z_j=\mathbf{b}_{\bar x_j}^\alpha$.
As $n$ is finite, we may let $\beta:=\min(D_\alpha\setminus\max\{ \dom(\bar x_j),\epsilon+1\mid j<n\})$.

$\br$ If $\alpha\notin A(\vec{D})$, then define an elevator $e:T_\beta\rightarrow T_\alpha$ via $e(y):=\mathbf{b}_y^\alpha$.
By Claim~\ref{claim885}, $e(z_j\restriction\beta)=z_j$ for every $j<n$.

To see that $e$ is a $(\varphi_2,\varphi_2)$-elevator,
let $\langle y_i\mid i<\chi\rangle\in (T_\beta)^\chi$.
For each $i<\chi$, denote $w_i:=e(y_i)$,
and find $ x_i\in T\restriction D_\alpha$ of minimal height such that $\mathbf{b}_{ x_i}^\alpha=w_i$.
As $w_i=\mathbf b^\alpha_{y_i}$, the minimality of $x_i$ implies that $\dom( x_i)\le\dom(y_i)$,
so that $\gamma:=\sup\{\dom( x_i)\mid i<\chi\}$ satisfies $\gamma\le\beta<\alpha$.
Let $p^\smallfrown\langle\xi\rangle$ denote $(\varphi_1\circ f_{\gamma})(\langle w_i\restriction\gamma\mid i<\chi\rangle)$,
and denote $\vec{w} := \langle w_i\mid i<\chi\rangle$.

On one hand, since $\beta\in D_\alpha\setminus\gamma$, Claim~\ref{c4206}(1) asserts that
$(\varphi_2\circ f_\beta)(\langle {w_i\restriction\beta}\mid i<\chi\rangle)=p$.
On the other hand, by the definition of $f_\alpha$,
it is the case that $f_\alpha(\vec{w})=p^\smallfrown\langle\xi+\sigma\rangle$
for some $\sigma\le\otp(D_\alpha)$,
and hence $(\varphi_2\circ f_\alpha)(\vec{w})=p$.
Altogether, $(\varphi_2\circ f_\alpha)(\vec{w})=(\varphi_2\circ f_\beta)(\langle w_i\restriction\beta\mid i<\chi\rangle)$.

$\br$ If $\alpha\in A(\vec{D})$, then as $T_\beta$ and $T_{\beta+1}$ are coordinated,
let us fix an $(\id,\varphi_0)$-elevator $e_0:T_\beta\rightarrow T_{\beta+1}$
such that $e_0(z_j\restriction\beta)=z_j\restriction(\beta+1)$ for every $j<n$.
Set $\delta:=\min(D_\alpha\setminus(\beta+2))$.
As $T_{\beta+1}$ and $T_\delta$ are coordinated,
fix a $(\varphi_2,\varphi_2)$-elevator $e_2:T_{\beta+1}\rightarrow T_\delta$
such that $e_2(z_j\restriction(\beta+1))=z_j\restriction\delta$ for every $j<n$.
Finally, define an elevator $e:T_\beta\rightarrow T_\alpha$ via $e(y):=\mathbf{b}_{e_2(e_0(y))}^\alpha$.
By Claim~\ref{claim885}, $e(z_j\restriction\beta)=z_j$ for every $j<n$.

To see that $e$ is a $(\varphi_2,\varphi_2)$-elevator,
let $\langle y_i\mid i<\chi\rangle\in (T_\beta)^\chi$.
For each $i<\chi$, denote $w_i:=e(y_i)$ and find $ x_i\in T\restriction D_\alpha$ of minimal height such that $\mathbf{b}_{ x_i}^\alpha=w_i$.
As $w_i=\mathbf b^\alpha_{e_2(e_0(y_i))}$, the minimality of $x_i$ implies that $\dom( x_i)\le\dom(e_2(e_0(y_i)))$,
so that $\gamma:=\sup\{\dom( x_i)\mid i<\chi\}$ satisfies $\gamma\le\delta<\alpha$.

Now, by the definition of $f_\alpha$,
letting $p^\smallfrown\langle\xi\rangle$ denote $(\varphi_1\circ f_{\gamma})(\langle w_i\restriction\gamma\mid i<\chi\rangle)$,
it is the case that $f_\alpha(\vec{w})=p$, for $\vec{w} := \langle w_i\mid i<\chi\rangle$.
In addition, since $\delta\in D_\alpha\setminus\gamma$, Claim~\ref{c4206}(1) asserts that
$(\varphi_2\circ f_\delta)(\langle w_i\restriction\delta\mid i<\chi\rangle)=p$.

By Remark~\ref{rmk417} and the choice of $e_0$ and $e_2$:
$$\begin{aligned}(\varphi_2\circ f_\beta)(\langle w_i\restriction\beta\mid i<\chi\rangle)&=(\varphi_2\circ \varphi_2\circ\varphi_0\circ f_\beta)(\langle w_i\restriction\beta\mid i<\chi\rangle)\\
&=(\varphi_2\circ \varphi_2\circ f_{\beta+1})(\langle e_0(w_i\restriction\beta)\mid i<\chi\rangle)\\
&=(\varphi_2\circ \varphi_2\circ f_{\beta+1})(\langle w_i\restriction(\beta+1)\mid i<\chi\rangle)\\
&=(\varphi_2\circ \varphi_2\circ f_{\delta})(\langle e_2(w_i\restriction(\beta+1))\mid i<\chi\rangle)\\
&=(\varphi_2\circ \varphi_2\circ f_\delta)(\langle w_i\restriction\delta\mid i<\chi\rangle)\\
&=\varphi_2(p)=(\varphi_2\circ f_\alpha)(\vec{w}),\end{aligned}$$
as sought.
\end{proof}

Let $\beta$ and $e:T_\beta\rightarrow T_\alpha$ be given by the subclaim.
We now go over the clauses of Definition~\ref{def_order}:
\begin{itemize} \item[(i)]
By the induction hypothesis, $T_\epsilon$ and $T_\beta$ are coordinated,
so we may fix a $(\varphi_2,\varphi_1)$-elevator $e_1:T_\epsilon\rightarrow T_\beta$ such that $e_1(z_j\restriction\epsilon)=z_j\restriction\beta$ for all $j<n$.
Set $E_1:=e\circ e_1$.
Then $E_1(z_j\restriction\epsilon)=e(e_1(z_j\restriction\epsilon))=e(z_j\restriction \beta)=z_j$ for all $j<n$.
In addition, for every $y\in T_\epsilon$, $y\stree e_1(y)\stree e(e_1(y))=E_1(y)$.
Finally, for every $\langle y_i\mid i<\chi\rangle\in(T_\epsilon)^\chi$,
$$\begin{aligned}(\varphi_2\circ f_\alpha)(\langle E_1(y_i)\mid i<\chi\rangle)&=(\varphi_2\circ f_\alpha)(\langle e(e_1(y_i))\mid i<\chi\rangle)\\
&=(\varphi_2\circ f_\beta)(\langle e_1(y_i)\mid i<\chi\rangle)\\
&=(\varphi_1\circ f_\epsilon)(\langle y_i\mid i<\chi\rangle).\end{aligned}$$
Thus, $E_1$ is a $(\varphi_2,\varphi_1)$-elevator, as sought.
\item[(ii)] Replace $1$ by $2$ throughout the above proof.
\item[(iii)] $\alpha$ is a limit ordinal, so the requirement is satisfied vacuously.\qedhere
\end{itemize}
\end{proof}

At the end of the above process, we have obtained a normal, slim, prolific, streamlined $\kappa$-tree $T:=\bigcup_{\alpha<\kappa}T_\alpha$
such that $f:=\bigcup_{\alpha<\kappa}f_\alpha$ is a strictly increasing map from $(T^\chi,{<_{T^\chi}})$ to $\mathbb Q_\lambda$.
The proof that $T$ is club-regressive follows that of \cite[Claim~2.3.4]{paper22};
thus we are left with proving that $T$ is $\chi$-free.
To this end, let $\vec s=\langle s_j\mid j<n\rangle\in T^n$ be given for some nonzero $n<\chi$,
and suppose that $A$ is a maximal antichain in $T(\vec s)$.
Let $\epsilon$ denote the unique element of $\{\dom(s_j)\mid j<n\}$.
Consider the club $E:=\{\alpha\in\acc(\kappa\setminus\epsilon)\mid (T\restriction\alpha)^n\s \phi[\alpha]\}$.

As the sequences $\langle \Omega_\beta\mid\beta<\kappa\rangle$ and $\langle R_i\mid i<\kappa\rangle$
together witness $\diamondsuit(H_\kappa)$,
for each $i<\kappa$,
we obtain from~\cite[Subclaim~4.1.4.1]{rinot20} that the following set is stationary in $\kappa$:
$$B_i:=\{\eta\in R_i\mid A\cap (T(\vec s)\restriction\eta)=\Omega_\eta\text{ is a maximal antichain in }T(\vec s)\restriction\eta\}.$$

Finally, using the hitting feature of the proxy sequence, pick some $\alpha\in E$
such that, for all $i<\alpha$, $$\sup(\nacc(C_\alpha)\cap B_i\cap\acc(\kappa))=\alpha.$$

\begin{claim} $A\s T(\vec s)\restriction\alpha$. In particular, $|A|<\kappa$.
\end{claim}
\begin{proof} Let $w$ be an arbitrary node in the $\alpha^{\text{th}}$ level of $T(\vec s)$,
and we shall show that it extends an element of $A$.
Recalling \eqref{promise1}, for each $j<n$, we may fix some $x_j\in T\restriction D_\alpha$ such that $(w)_j=\mathbf{b}_{x_j}^\alpha$.
As $n$ is finite, by Claim~\ref{claim885}, we may assume the existence of some $\gamma\in D_\alpha\setminus\epsilon$ such that $s_j\s x_j$ and $\dom(x_j)=\gamma$ for all $j<n$.
In particular, $\langle x_j\mid j<n\rangle \in (T_\gamma)^n$ and $\bar T:=T(\langle x_j\mid j<n\rangle)$ is a normal streamlined subtree of $T(\vec s)$.

Next, as $\alpha\in E$, we may find some $i<\alpha$ with $\phi(i)=\langle x_j\mid j<n\rangle$.
Pick a large enough $\eta\in\nacc(C_\alpha)\cap B_i\cap\acc(\kappa)$ such that $\sup(D_\alpha\cap\eta)>\gamma$.
Denote $\beta:=\eta+1$ and $\beta^-:=\sup(D_\alpha\cap\beta)$, so that $\gamma<\beta^-\le\eta<\beta$ with $\beta\in D_\alpha$.
Note that $\beta^-<\eta$, since otherwise, $\eta=\beta^-\in D_\alpha\cap\acc(\kappa)=\acc(C_\alpha)$,
contradicting the fact that $\eta\in\nacc(C_\alpha)$.

Recalling Notation~\ref{gfamily} and Claim~\ref{E-nonempty}, let $e:=\min(\mathcal E^\alpha_\beta,\lhd_\kappa)$,
so that $b_{x_j}^\alpha(\beta)=e(b_{x_j}^\alpha(\beta^-))$ for every $j<n$.
As $\psi(\eta)=\phi(\pi(\eta))=\phi(i)=\langle x_j\mid j<n\rangle$
and the latter is indeed an element of $(T\restriction(D_\alpha\cap\beta^-))^n$,
we get that for every $j<n$,
$$e(b_{x_j}^\alpha(\beta^-))=(\sealantichain(\fork{\langle b_{x_j}^\alpha(\beta^-)\mid j<n\rangle},\bar T\restriction(\beta+1),\Omega_\eta))_j.$$

By the choice of $\eta$,
$A\cap (T(\vec s)\restriction\eta)=\Omega_\eta$ is a maximal antichain in $T(\vec s)\restriction\eta$,
so there exists a $y\in\Omega_\eta$ that is comparable to $\fork{\langle b_{x_j}^\alpha(\beta^-)\mid j<n\rangle}$.
Furthermore, any node witnessing this comparability belongs to $\bar T$.
Together with the normality of $\bar T$, it follows that the following set is nonempty:
$$Q:=\{ z\in \bar T_{\beta}\mid \exists y\in\Omega_\eta\,( \fork{\langle b_{x_j}^\alpha(\beta^-)\mid j<n\rangle}\cup y\s z)\}.$$
Denote $z:=\min(Q, {\lhd_\kappa})$,
and let $y\in\Omega_\eta$ be a witness for $z\in Q$.
Recalling Definition~\ref{actions}(2), this means that for every $j<n$,
$$(w)_j\restriction\beta=b_{x_j}^\alpha(\beta)=e(b_{x_j}^\alpha(\beta^-))=(z)_j,$$
and hence $y\stree z\stree w$. As $y\in\Omega_\eta\s A$, we infer that $w$ indeed extends an element of $A$.
\end{proof}
This completes the proof.
\end{proof}

\begin{cor}\label{cor611} Suppose that $\lambda$ is a singular cardinal satisfying both $\square_\lambda$ and $2^\lambda=\lambda^+$.
Then for every positive integer $n$,
there exists a streamlined $\lambda^+$-Souslin tree $T$ satisfying the two:
\begin{itemize}
\item all $n$-derived trees of $T$ are Souslin;
\item all $(n+1)$-derived trees of $T$ are special.
\end{itemize}
\end{cor}
\begin{proof} By \cite[Corollary~3.10]{paper22}, for a singular cardinal $\lambda$,
$\p_\lambda(\lambda^+,2,{\sq},\lambda^+)$ is equivalent to the conjunction of $\square_\lambda$ and $2^\lambda=\lambda^+$.
Now, appeal to Theorem~\ref{thm710} with $\chi:=n+1$.
\end{proof}

It is not hard to see that assuming $\lambda=\lambda^{<\lambda}$,
every $\lambda^+$-tree whose square is special is in particular specializable.
We do not know of an example of a specializable $\lambda^+$-Souslin tree for $\lambda$ singular,
and so we ask whether the tree given by Corollary~\ref{cor611} is (or can be tweaked to be) specializable.

\section*{Acknowledgments}
The first author was supported by the Shamoon College of Engineering Young Research Grant YR/08/Y21/T2/D3.
The second and third authors were supported by the Israel Science Foundation (grant agreement 203/22).
The third author was also supported by the European Research Council (grant agreement ERC-2018-StG 802756).

Results from this paper were presented by the first author at the ninth \emph{European Set-Theory Conference},
M\"unster, September 2024.
We thank the organizers of the conference for providing a joyful and stimulating environment,
and the participants for their feedback.
In particular, we appreciate discussions with Pedro Marun,
which prompted us to include several clarifications in the final stages of the paper's preparation.


\begin{thebibliography}{CFM03}

\bibitem[ADS78]{Sh:81}
Uri Abraham, Keith~J. Devlin, and Saharon Shelah.
\newblock {The consistency with CH of some consequences of Martin's axiom plus $2^{\aleph_0}>\aleph_1$}.
\newblock {\em Israel Journal of Mathematics}, 31(1):19--33, 1978.

\bibitem[AS93]{Sh:403}
Uri Abraham and Saharon Shelah.
\newblock {A $\Delta^2_2$ well-order of the reals and incompactness of $L(Q^\mathrm{MM})$}.
\newblock {\em Annals of Pure and Applied Logic}, 59(1):1--32, 1993.

\bibitem[ASS87]{sh:221}
Uri Abraham, Saharon Shelah, and Robert~M. Solovay.
\newblock {Squares with diamonds and Souslin trees with special squares}.
\newblock {\em Polska Akademia Nauk. Fundamenta Mathematicae}, 127(2):133--162, 1987.

\bibitem[BDS86]{Sh:236}
Shai Ben-David and Saharon Shelah.
\newblock {Nonspecial Aronszajn trees on $\aleph_{\omega+1}$}.
\newblock {\em Israel Journal of Mathematics}, 53(1):93--96, 1986.

\bibitem[BR17a]{paper22}
Ari~Meir Brodsky and Assaf Rinot.
\newblock A microscopic approach to {S}ouslin-tree constructions. {P}art {I}.
\newblock {\em Ann. Pure Appl. Logic}, 168(11):1949--2007, 2017.

\bibitem[BR17b]{rinot20}
Ari~Meir Brodsky and Assaf Rinot.
\newblock Reduced powers of {S}ouslin trees.
\newblock {\em Forum Math. Sigma}, 5(e2):1--82, 2017.

\bibitem[BR19a]{paper29}
Ari~Meir Brodsky and Assaf Rinot.
\newblock Distributive {A}ronszajn trees.
\newblock {\em Fund. Math.}, 245(3):217--291, 2019.

\bibitem[BR19b]{paper26}
Ari~Meir Brodsky and Assaf Rinot.
\newblock More notions of forcing add a {S}ouslin tree.
\newblock {\em Notre Dame J. Form. Log.}, 60(3):437--455, 2019.

\bibitem[BR19c]{paper32}
Ari~Meir Brodsky and Assaf Rinot.
\newblock A remark on {S}chimmerling's question.
\newblock {\em Order}, 36(3):525--561, 2019.

\bibitem[BR21]{paper23}
Ari~Meir Brodsky and Assaf Rinot.
\newblock A microscopic approach to {S}ouslin-tree constructions. {P}art {II}.
\newblock {\em Ann. Pure Appl. Logic}, 172(5):Paper No. 102904, 65, 2021.

\bibitem[Bro14]{MR3274402}
Ari~Meir Brodsky.
\newblock A theory of stationary trees and the balanced {B}aumgartner--{H}ajnal--{T}odorcevic theorem for trees.
\newblock {\em Acta Math. Hungar.}, 144(2):285--352, 2014.

\bibitem[CFM03]{MR1976595}
James Cummings, Matthew Foreman, and Menachem Magidor.
\newblock The non-compactness of square.
\newblock {\em J. Symbolic Logic}, 68(2):637--643, 2003.

\bibitem[CL17]{MR3620068}
Sean Cox and Philipp L\"ucke.
\newblock Characterizing large cardinals in terms of layered posets.
\newblock {\em Ann. Pure Appl. Logic}, 168(5):1112--1131, 2017.

\bibitem[CM11]{MR2811288}
James Cummings and Menachem Magidor.
\newblock Martin's maximum and weak square.
\newblock {\em Proc. Amer. Math. Soc.}, 139(9):3339--3348, 2011.

\bibitem[Dev79]{devlin1979variations}
Keith~J. Devlin.
\newblock Variations on $\diamondsuit$.
\newblock {\em The Journal of Symbolic Logic}, 44(1):51--58, 1979.

\bibitem[Dev84]{devlin-book}
Keith~J. Devlin.
\newblock {\em Constructibility}.
\newblock Perspectives in Mathematical Logic. Springer-Verlag, Berlin, 1984.

\bibitem[DJ74]{MR384542}
Keith~J. Devlin and H{\aa}vard {Johnsbr{\aa}ten}.
\newblock {\em The {S}ouslin problem}.
\newblock Lecture Notes in Mathematics, Vol. 405. Springer-Verlag, Berlin-New York, 1974.

\bibitem[FL88]{MR925267}
Matthew Foreman and Richard Laver.
\newblock Some downwards transfer properties for {$\aleph_2$}.
\newblock {\em Adv. in Math.}, 67(2):230--238, 1988.

\bibitem[Fre84]{fremlin1984consequences}
D.~H. Fremlin.
\newblock {\em Consequences of {M}artin's axiom}, volume~84 of {\em Cambridge Tracts in Mathematics}.
\newblock Cambridge University Press, Cambridge, 1984.

\bibitem[Fri06]{MR2374763}
Sy-David Friedman.
\newblock Large cardinals and {$L$}-like universes.
\newblock In {\em Set theory: recent trends and applications}, volume~17 of {\em Quad. Mat.}, pages 93--110. Dept. Math., Seconda Univ. Napoli, Caserta, 2006.

\bibitem[Gra80]{MR2940957}
Charles~William Gray.
\newblock {\em Iterated forcing from the strategic point of view}.
\newblock ProQuest LLC, Ann Arbor, MI, 1980.
\newblock Thesis (Ph.D.)--University of California, Berkeley.

\bibitem[GS64]{MR168484}
Haim Gaifman and E.~P. Specker.
\newblock Isomorphism types of trees.
\newblock {\em Proc. Amer. Math. Soc.}, 15:1--7, 1964.

\bibitem[HLH17]{MR3730566}
Yair Hayut and Chris Lambie-Hanson.
\newblock Simultaneous stationary reflection and square sequences.
\newblock {\em J. Math. Log.}, 17(2):1750010, 27, 2017.

\bibitem[IR23]{paper47}
Tanmay Inamdar and Assaf Rinot.
\newblock Was {U}lam right? {I}: {B}asic theory and subnormal ideals.
\newblock {\em Topology Appl.}, 323(C):Paper No. 108287, 53pp, 2023.

\bibitem[IR24a]{paper46}
Tanmay Inamdar and Assaf Rinot.
\newblock A club guessing toolbox {I}.
\newblock {\em Bull. Symb. Log.}, 30(3):303--361, 2024.

\bibitem[IR24b]{paper53}
Tanmay Inamdar and Assaf Rinot.
\newblock Was {U}lam right? {I}{I}: {S}mall width and general ideals.
\newblock {\em Algebra Universalis}, 85(2):Paper No. 14, 47pp, 2024.

\bibitem[IR25]{paper71}
Tanmay Inamdar and Assaf Rinot.
\newblock Walks on uncountable ordinals and non-structure theorems for higher {A}ronszajn lines.
\newblock Submitted June 2025.
\newblock \verb"http://assafrinot.com/paper/71"

\bibitem[Jen72]{MR309729}
R.~Bj{\"o}rn Jensen.
\newblock The fine structure of the constructible hierarchy.
\newblock {\em Ann. Math. Logic}, 4(3):229--308, 1972.
\newblock With a section by Jack Silver.

\bibitem[JJ74]{MR419229}
Ronald~B. Jensen and H{\aa}vard {Johnsbr{\aa}ten}.
\newblock A new construction of a non-constructible {$\Delta _{3}^{1}$} subset of {$\omega $}.
\newblock {\em Fund. Math.}, 81:279--290, 1974.

\bibitem[K{\"o}n03]{MR2013395}
Bernhard K{\"o}nig.
\newblock Local coherence.
\newblock {\em Ann. Pure Appl. Logic}, 124(1-3):107--139, 2003.

\bibitem[Kru23]{MR4530628}
John Krueger.
\newblock A large pairwise far family of {A}ronszajn trees.
\newblock {\em Ann. Pure Appl. Logic}, 174(4):Paper No. 103236, 12, 2023.

\bibitem[Kun80]{kunen1980set}
Kenneth Kunen.
\newblock {\em Set theory}, volume 102 of {\em Studies in Logic and the Foundations of Mathematics}.
\newblock North-Holland Publishing Co., Amsterdam-New York, 1980.
\newblock An introduction to independence proofs.

\bibitem[LH17]{MR3724382}
Chris Lambie-Hanson.
\newblock Aronszajn trees, square principles, and stationary reflection.
\newblock {\em MLQ Math. Log. Q.}, 63(3-4):265--281, 2017.

\bibitem[LHR19]{paper28}
Chris Lambie-Hanson and Assaf Rinot.
\newblock Reflection on the coloring and chromatic numbers.
\newblock {\em Combinatorica}, 39(1):165--214, 2019.

\bibitem[LHR21]{paper35}
Chris Lambie-Hanson and Assaf Rinot.
\newblock Knaster and friends {II}: {T}he {C}-sequence number.
\newblock {\em J. Math. Log.}, 21(1):2150002, 54, 2021.

\bibitem[L{\"{u}}c17]{lucke}
Philipp L{\"{u}}cke.
\newblock Ascending paths and forcings that specialize higher {A}ronszajn trees.
\newblock {\em Fund. Math.}, 239(1):51--84, 2017.

\bibitem[ML12]{magidor2012strengths}
Menachem Magidor and Chris {Lambie-Hanson}.
\newblock On the strengths and weaknesses of weak squares.
\newblock {\em Appalachian Set Theory: 2006--2012}, 406:301--330, 2012.

\bibitem[Rin11]{rinot_s01}
Assaf Rinot.
\newblock {J}ensen's diamond principle and its relatives.
\newblock In {\em Set theory and its applications}, volume 533 of {\em Contemp. Math.}, pages 125--156. Amer. Math. Soc., Providence, RI, 2011.

\bibitem[Rin15]{rinot12}
Assaf Rinot.
\newblock Chromatic numbers of graphs - large gaps.
\newblock {\em Combinatorica}, 35(2):215--233, 2015.

\bibitem[Rin17]{paper24}
Assaf Rinot.
\newblock Higher {S}ouslin trees and the {GCH}, revisited.
\newblock {\em Adv. Math.}, 311(C):510--531, 2017.

\bibitem[Rin19]{paper37}
Assaf Rinot.
\newblock {S}ouslin trees at successors of regular cardinals.
\newblock {\em MLQ Math. Log. Q.}, 65(2):200--204, 2019.

\bibitem[Rin22]{paper51}
Assaf Rinot.
\newblock On the ideal ${J}[\kappa]$.
\newblock {\em Ann. Pure Appl. Logic}, 173(2):Paper No. 103055, 13pp, 2022.

\bibitem[RS17]{rinot21}
Assaf Rinot and Ralf Schindler.
\newblock Square with built-in diamond-plus.
\newblock {\em J. Symbolic Logic}, 82(3):809--833, 2017.

\bibitem[RS23]{paper48}
Assaf Rinot and Roy Shalev.
\newblock A guessing principle from a {S}ouslin tree, with applications to topology.
\newblock {\em Topology Appl.}, 323(C):Paper No. 108296, 29pp, 2023.

\bibitem[RYY24]{paper62}
Assaf Rinot, Shira Yadai, and Zhixing You.
\newblock Full {S}ouslin trees at small cardinals.
\newblock {\em J. Lond. Math. Soc. (2)}, 110(1):e12957, 2024.

\bibitem[RYY25]{paper58}
Assaf Rinot, Shira Yadai, and Zhixing You.
\newblock The vanishing levels of a tree.
\newblock \emph{Canad. J. Math.}, accepted July 2025.
\newblock \verb"https://doi.org/10.4153/S0008414X2510148X"

\bibitem[RZ23]{paper45}
Assaf Rinot and Jing Zhang.
\newblock Strongest {T}ransformations.
\newblock {\em Combinatorica}, 43(1):149--185, 2023.

\bibitem[Sch95]{schimmerling1995combinatorial}
Ernest Schimmerling.
\newblock Combinatorial principles in the core model for one {W}oodin cardinal.
\newblock {\em Annals of Pure and Applied Logic}, 74(2):153--201, 1995.

\bibitem[Sha24]{MR4833803}
Roy Shalev.
\newblock More minimal non-{$\sigma$}-scattered linear orders.
\newblock {\em Eur. J. Math.}, 10(4):Paper No. 74, 2024.

\bibitem[SS88]{Sh:279}
Saharon Shelah and Lee~J. Stanley.
\newblock {Weakly compact cardinals and nonspecial Aronszajn trees}.
\newblock {\em Proc. Amer. Math. Soc.}, 104(3):887--897, 1988.

\bibitem[Tod84]{MR776625}
Stevo Todor\v{c}evi\'{c}.
\newblock Trees and linearly ordered sets.
\newblock In {\em Handbook of set-theoretic topology}, pages 235--293. North-Holland, Amsterdam, 1984.

\bibitem[Tod85]{MR0793235}
Stevo Todor\v{c}evi\'{c}.
\newblock Partition relations for partially ordered sets.
\newblock {\em Acta Math.}, 155(1-2):1--25, 1985.

\bibitem[Tod87]{MR908147}
Stevo Todor{\v{c}}evi{\'c}.
\newblock Partitioning pairs of countable ordinals.
\newblock {\em Acta Math.}, 159(3--4):261--294, 1987.

\bibitem[Tod07]{TodWalks}
Stevo Todor\v{c}evi\'c.
\newblock {\em Walks on ordinals and their characteristics}, volume 263 of {\em Progress in Mathematics}.
\newblock Birkh\"auser Verlag, Basel, 2007.

\bibitem[Yad23]{yadai}
Shira Yadai.
\newblock {\em Exploration of the Microscopic Approach for {S}ouslin-Tree Constructions}.
\newblock 2023.
\newblock Thesis (M.Sc.), Bar-Ilan University.

\bibitem[Zak81]{MR638747}
Marek Zakrzewski.
\newblock Weak product of {S}ouslin trees can satisfy the countable chain condition.
\newblock {\em Bull. Acad. Polon. Sci. S\'{e}r. Sci. Math.}, 29(3-4):99--102, 1981.

\end{thebibliography}
\end{document}